\documentclass[10pt]{smfart}

\RequirePackage[T1]{fontenc}
\RequirePackage{amsfonts,latexsym,amssymb}

\RequirePackage{smfenum}
\RequirePackage[frenchb]{babel}
\addto\extrasfrenchb{\bbl@nonfrenchitemize
\bbl@nonfrenchspacing}

\RequirePackage{mathrsfs}
\let\mathcal\mathscr

\usepackage[matrix,arrow]{xy}
\usepackage{url}
\usepackage{accents}
\usepackage{enumitem}
\usepackage{fge}
\usepackage{bbm}
\newcommand{\moins}{\mathbin{\fgebackslash}}

\theoremstyle{plain}
\numberwithin{equation}{section}
\newtheorem{prop}[equation]{\propname}
\newtheorem{theo}[equation]{\theoname}

\newtheorem{coro}[equation]{\coroname}

\newtheorem{lemm}[equation]{\lemmname}
\theoremstyle{definition}
\theoremstyle{remark}
\newtheorem{defi}[equation]{\definame} 
\newtheorem{rema}[equation]{\remaname}

\newtheorem{ques}[equation]{Question}

\def\paskunas{{Pa\v{s}k\={u}nas}}

\let\cal\mathcal
\let\mathcal\mathscr
\let\goth\mathfrak

\def\pp{\goth p}
\def\Q{{\bf Q}} \def\Z{{\bf Z}}
\def\zp{\Z_p} \def\qp{\Q_p}
\def\C{{\bf C}}
\def\N{{\bf N}}
\def\O{{\cal O}}
\def\G{{\cal G}} 

\def\dual{{\boldsymbol *}}
\def\bmu{{\boldsymbol\mu}}
\def\bPi{{\boldsymbol\Pi}}

\def\Qbar{\overline{\bf Q}}
\def\epsilon{\varepsilon}
\def\epsilon{\varepsilon}
\def\mv{{\cal M}^\varpi}

\def\bcris{{\rm B}_{{\rm cris}}}

\def\piqp{{\bf P}^1}

\def\matrice#1#2#3#4{{\big(\begin{smallmatrix}#1&#2\\ #3&#4\end{smallmatrix}\big)}}

\newcommand{\gl}{{\rm GL}}
\newcommand{\SL}{{\rm SL}}
\newcommand{\colim}{\operatorname{colim}}
\newcommand{\R}{\mathrm {R} }
\newcommand{\gp}{\operatorname{gp} }
\newcommand{\coker}{\operatorname{coker} }
\newcommand{\proet}{\operatorname{pro\acute{e}t}  }
\newcommand{\eet}{\operatorname{\acute{e}t} }
\newcommand{\dlog}{\operatorname{dlog} }
\newcommand{\Spf}{\operatorname{Spf} }
\newcommand{\Hom}{\operatorname{Hom} }
\newcommand{\proeet}{ \operatorname{pro\acute{e}t} }
\newcommand{\sff}{{\mathcal{F}}}
\newcommand{\sg}{{\mathcal{G}}}
\newcommand{\scc}{{\mathcal{C}}}
\newcommand{\so}{{\mathcal O}}
\newcommand{\sm}{{\mathcal{M}}}
\newcommand{\wt}{\widetilde}
\newcommand{\wh}{\widehat}
\numberwithin{equation}{section}

\begin{document}
\title[Factorisation de la cohomologie de la tour de Drinfeld]
{Factorisation de la cohomologie \'etale $p$-adique de la tour de Drinfeld}
\author{Pierre Colmez}
\address{CNRS, IMJ-PRG, Sorbonne Universit\'e, 4 place Jussieu, 75005 Paris, France}
\email{pierre.colmez@imj-prg.fr}
\author{Gabriel Dospinescu}
\address{CNRS, UMPA, \'Ecole Normale Sup\'erieure de Lyon, 46 all\'ee d'Italie, 69007 Lyon, France}
\email{gabriel.dospinescu@ens-lyon.fr}
\author{Wies{\l}awa Nizio{\l}}
\address{CNRS, IMJ-PRG, Sorbonne Universit\'e, 4 place Jussieu, 75005 Paris, France}
\email{wieslawa.niziol@imj-prg.fr}
\begin{abstract}
Si $F$ est une extension finie de $\Q_p$, Drinfeld a d\'efini une tour de rev\^etements
de l'espace analytique $\piqp\moins\piqp(F)$ (demi-plan de Drinfeld).
Si $F=\Q_p$,
nous donnons une d\'ecomposition de la cohomologie \'etale $p$-adique g\'eom\'etrique
de cette tour de rev\^etements
analogue \`a la d\'ecomposition d'Emerton de la cohomologie compl\'et\'ee de la tour des courbes modulaires.
Un ingr\'edient crucial est un th\'eor\`eme de finitude pour la cohomologie \'etale arithm\'etique modulo~$p$
dont la preuve utilise le foncteur de Scholze, des ingr\'edients globaux et
un calcul de cycles proches qui permet de prouver que cette cohomologie
modulo~$p$ est de pr\'esentation finie. Ce dernier r\'esultat est valable pour $F$ quelconque; pour $F\neq \Q_p$,
il implique que les repr\'esentations de $\gl_2(F)$ obtenues \`a partir de la cohomologie de la tour de Drinfeld
ne sont pas admissibles contrairement au cas $F=\Q_p$.
\end{abstract}
\begin{altabstract}
For a finite extension $F$ of  ${\mathbf Q}_p$, Drinfeld defined a tower of coverings of 
$\piqp\moins\piqp(F)$ (the Drinfeld half-plane).
  For  $F = {\mathbf Q}_p$, we describe a decomposition of the $p$-adic geometric \'etale cohomology of this tower analogous to Emerton's decomposition of completed cohomology of the tower of modular curves.
 A crucial ingredient  is a finiteness theorem for the arithmetic \'etale cohomology   
modulo $p$ whose proof uses Scholze's functor, global ingredients, and
a computation of nearby cycles which makes it possible to prove
that this cohomology  has finite presentation.
This last result holds for all $F$; for  $F\neq  {\mathbf Q}_p$, it implies that the representations of  ${\rm GL}_2(F)$ obtained from the cohomology of the Drinfeld tower are not admissible contrary to the case $F = {\mathbf Q}_p$.
\end{altabstract}

 \date{\today}
\thanks{Les trois auteurs sont membres du projet ANR-19-CE40-0015-02 COLOSS; ils remercient
le rapporteur pour sa lecture attentive et ses suggestions}
\setcounter{tocdepth}{2}

\maketitle

{\Small
\tableofcontents
}

\section*{Introduction}
Dans~\cite{CDN1}, nous avons calcul\'e la multiplicit\'e d'une repr\'esentation $p$-adique
absolument irr\'eductible du groupe de Galois absolu $\G_{\Q_p}$ de 
$\Q_p$ dans la cohomologie \'etale 
de la tour $({\cal M}_{n,\C_p})_{n\geq 0}$
des rev\^etements du demi-plan de Drinfeld
(pour avoir une action de $\G_{\Q_p}$ et pas seulement du groupe de Weil $W_{\Q_p}$, il faut remplacer
${\cal M}_{n,\C_p}$ par son quotient\footnote{Celui-ci est l'extension des scalaires \`a $\C_p$
d'une vari\'et\'e analytique ${\cal M}_{n,\Q_p}^p$ d\'efinie sur $\Q_p$.
Notons que quotienter par $p$ induit des restrictions sur les repr\'esentations
qui interviennent (cf.~note~\ref{restr}); 
on fait appara\^{\i}tre les autres repr\'esentations en tordant par des caract\`eres
  comme dans \cite[\S\,5.1]{CDN1}.}
${\cal M}_{n,\C_p}^p$ par $\matrice{p}{0}{0}{p}\in G:=\gl_2(\Q_p)$). 
 La r\'eponse fait intervenir la correspondance
de Langlands locale $p$-adique pour $G$, 
ainsi que la correspondance de Jacquet-Langlands locale classique, cf.~th.\,\ref{icdn1.1} ci-dessous.

Dans cet article, nous d\'eterminons, 
pour\footnote{Cette hypoth\`ese est utilis\'ee dans certains arguments globaux et aussi 
pour invoquer directement certains r\'esultats de \cite{Paskext} 
(cf.~notes~\ref{inutile1} et~\ref{inutile2} pour des pr\'ecisions); 
elle est sans doute superflue. Pour $p=3$ l'\'enonc\'e du th.\,\ref{intro10} reste correct sauf, peut-\^etre, pour la contribution des
blocs ${\cal B}$ contenant un twist de la steinberg.}
$p>3$, la structure compl\`ete de cette cohomologie en tant
que $\G_{\Q_p}\times G\times \check G$-module, o\`u $\check G$ est le groupe des unit\'es de l'alg\`ebre
des quaternions non d\'eploy\'ee sur $\Q_p$. La r\'eponse (th.\,\ref{intro10} ci-dessous)
est similaire \`a la description d'Emerton~\cite{eternalpreprint} de la cohomologie
compl\'et\'ee de la tour des courbes modulaires, les alg\`ebres de Hecke dans la description d'Emerton
\'etant remplac\'ees par des anneaux de Kisin~\cite{Kis} param\'etrant les repr\'esentations potentiellement
cristallines, \`a poids de Hodge-Tate $0$ et $1$, et type fix\'e (on fixe en fait la repr\'esentation de Weil-Deligne, pas seulement le type galoisien).

Plus pr\'ecis\'ement, l'objet que nous d\'ecrivons n'est pas la cohomologie
de ${\cal M}_{n,\C_p}^p$ mais celle
de ${\cal M}_{n,\Qbar_p}^p$ (i.e.~la
cohomologie compl\'et\'ee\footnote{Noter que c'est bien le corps des coefficients $K$ qui varie ici, pas le niveau $n$, contrairement \`a la cohomologie compl\'et\'ee des courbes modulaires. Le fait que les cohomologies de ${\cal M}_{n,\Qbar_p}^p$ et de ${\cal M}_{n,\C_p}^p$ ne sont pas les m\^emes est une manifestation de la non quasi-compacit\'e des espaces ${\cal M}_{n,\C_p}^p$.} de la tour des ${\cal M}_{n,K}^p$ pour $[K:\Q_p]<\infty$);
il est possible que
le quotient de la cohomologie de ${\cal M}_{n,\C_p}^p$ par celle de ${\cal M}_{n,\Qbar_p}^p$
soit du bruit sans signification arithm\'etique (en tout cas, nous avons \'echou\'e \`a lui trouver
une interpr\'etation). 

Un r\^ole crucial est jou\'e par un th\'eor\`eme de finitude pour la cohomologie arithm\'etique modulo $p$ de la tour de Drinfeld (th.\,\ref{Intro1}), 
dont la preuve (g\'eom\'etrique et tr\`es diff\'erente des m\'ethodes utilis\'ees dans \cite{CDN1}, 
en particulier elle ne fait pas usage de~\cite{DL}) occupe la majeure partie de l'article. 

\subsection{Les r\'esultats principaux}

 Notre r\'esultat principal (th.\,\ref{factor10}) s'\'enonce comme suit: 

\begin{theo}\label{intro10}
Si $L$ est une extension finie de $\Q_p$ assez grande\footnote{Il faut que les
${\rm JL}(M)$ qui interviennent (il n'y en a qu'un nombre fini) soient d\'efinis sur $L$.}, on dispose d'un isomorphisme de $L[\G_{\Q_p}\times G\times \check G]$-modules topologiques
$$H^1_{\eet}(\sm^p_{n,\Qbar_p},L(1))\simeq 
\oplus_M\big(\widehat{\oplus}_{\cal B} \ \bPi^\dual(\rho_{{\cal B},M})
\otimes \rho_{{\cal B},M}\otimes \check{R}_{{\cal B},M}\big)\otimes_L{\rm JL}(M),$$
o\`u $\check{R}_{{\cal B},M}$ est le $L$-dual continu de $R_{{\cal B},M}$
 et
les produits tensoriels sont
au-dessus de~$R_{{\cal B},M}$.
\end{theo}

Le lecteur trouvera des d\'efinitions pr\'ecises\footnote{\label{restr}
Nous n'expliciterons pas, dans les \'enonc\'es, les restrictions induites par le fait
d'avoir quotient\'e par $\matrice{p}{0}{0}{p}$. Ces restrictions sont les suivantes: 
$p$, vu comme \'el\'ement des centres de $G$ et $\check G$,
 agit trivialement; le d\'eterminant des repr\'esentations de $\G_{\Q_p}$, vu comme
caract\`ere de $\Q_p^\dual$, prend la valeur $1$ en $p$; le d\'eterminant de $\varphi$
agissant sur les $M$ vaut $p$.}
 des objets 
intervenant ci-dessus dans la suite de cette introduction, 
bornons-nous ici \`a mettre en avant quelques points essentiels: 

$\bullet$ La premi\`ere somme porte sur les types $M$ de niveau~$\leq n$. 
Ces types sont en bijection avec les repr\'esentations irr\'eductibles du groupe de Galois de ${\cal M}_{n,\C_p}^p$ par rapport au demi-plan de Drinfeld. Ce groupe est un quotient fini de $\check{G}$. 
On note 
${\rm JL}(M)$ la repr\'esentation irr\'eductible (de dimension finie) de $\check G$ associ\'ee \`a $M$.

$\bullet$ La somme directe compl\'et\'ee porte 
 sur les blocs ${\cal B}$
de la cat\'egorie des $k_L$-repr\'esentations lisses de longueur finie de $G$. Ces blocs sont
en bijection naturelle \cite{Pas1} avec les orbites sous ${\rm Gal}(\overline{\bf F}_p/k_L)$
des $\overline{\bf F}_p$-repr\'esentations $\rho_{\cal B}$
de $\G_{\Q_p}$, semi-simples, de dimension~$2$. La compl\'etion intervenant dans la somme directe est $p$-adique.\footnote{$\bPi^\dual(\rho_{{\cal B},M})
\otimes \rho_{{\cal B},M}\otimes \check{R}_{{\cal B},M}$ poss\`ede un r\'eseau $\G_{\Q_p}\times G\times \check G$-stable naturel, et 
$\widehat{\oplus}_{\cal B} \ \bPi^\dual(\rho_{{\cal B},M})
\otimes \rho_{{\cal B},M}\otimes \check{R}_{{\cal B},M}$ s'obtient en compl\'etant $p$-adiquement la somme directe de ces r\'eseaux et en inversant~$p$.}

$\bullet$ $R_{{\cal B},M}$ est l'anneau de Kisin param\'etrant les repr\'esentations de type $M$ 
et r\'eduction $\rho_{\cal B}$,
et $\rho_{{\cal B},M}$ est la repr\'esentation universelle associ\'ee. L'anneau $R_{{\cal B},M}$ est un quotient de la fibre g\'en\'erique de l'anneau des pseudo-caract\`eres d\'eformant la trace de $\rho_{\cal B}$. 

$\bullet$ Le $G$-module topologique $\bPi^\dual(\rho_{{\cal B},M})$ interpole\footnote{
$(R_{{\cal B},M}/{\goth m}_x)\otimes
\bPi^\dual(\rho_{{\cal B},M})\simeq \bPi(\rho_x)^\dual$.}
 les $\bPi(\rho_x)^\dual$, 
pour $x\in {\rm Spm}(R_{{\cal B},M})$, 
o\`u $V\mapsto\bPi(V)$ est la correspondance
de Langlands locale $p$-adique, $\rho_x$ est la sp\'ecialisation de $\rho_{{\cal B},M}$ en 
$x$ et tous les duaux sont topologiques\footnote{Si $W$ est un $L$-module topologique, nous
notons $W^\dual$ son dual topologique, et si $W$ est un $\O_L$-module, nous notons $W^\vee$
son dual de Pontryagin. Notons que, si $W$ est un $L$-banach, la topologie sur $W^\dual$
est la topologie de la convergence uniforme sur les compacts (pas la topologie faible...).}.

\medskip
Un ingr\'edient crucial de la preuve (en dehors des r\'esultats de \cite{CDN1} et \cite{DL}, qui sont pleinement utilis\'es) 
est le th\'eor\`eme suivant (th.\,\ref{main11}): 

\begin{theo}\label{Intro1}
Si $[K:\Q_p]<\infty$, alors $H^1_{\eet}({\cal M}_{n,K}^p,\bmu_p)$ est la duale d'une
repr\'esentation lisse admissible de $G$, de longueur finie comme $\Z_p[G]$-module.
\end{theo}
Le th\'eor\`eme ci-dessus coupl\'e \`a des arguments globaux permet d'obtenir le r\'esultat suivant
(th.\,\ref{pointwise2} et rem.\,\ref{pointwise2.1}), 
qui est un analogue modulo $p$ du r\'esultat principal de \cite{CDN1}. Si $\bar{\rho}: \mathcal{G}_{\Q_p}\to {\rm GL}_d(k_L)$ est une repr\'esentation continue, on pose: 
   $$\bPi^{\rm geo}_n(\bar{\rho}):={\rm Hom}_{k_L[\mathcal{G}_{\Q_p}]}^{\rm cont}(\bar{\rho}, H^1_{\eet}(\mathcal{M}_{n,\C_p}^{p},k_L))^{\vee}.$$

\begin{theo}\label{main2} Soit
   $\bar{\rho}: \mathcal{G}_{\Q_p}\to {\rm GL}_d(k_L)$ une repr\'esentation continue.
   
   {\rm a)} $\bPi^{\rm geo}_n(\bar{\rho})$ est une $G$-repr\'esentation lisse, de longueur finie.   
   
   {\rm b)} Si $\bar{\rho}$ est absolument irr\'eductible, alors $\bPi^{\rm geo}_n(\bar{\rho})=0$ si 
   $d>2$ et une extension successive de copies de $\bPi(\bar{\rho})$ si $d=2$, o\`u 
   $\bPi(\bar{\rho})$ est la repr\'esentation supersinguli\`ere de $G$ associ\'ee \`a $\bar{\rho}$ par la correspondance de Langlands locale modulo~$p$.
 \end{theo}

On peut se demander ce qui reste vrai si on regarde la tour des rev\^etements du demi-plan
de Drinfeld $\piqp\moins\piqp(F)$, o\`u $F$ est une extension finie de $\Q_p$ distincte de~$\Q_p$. 
Dans ce cas, nous prouvons 
le r\'esultat suivant (th.\,\ref{main10bis} et~\ref{main10}), 
en utilisant un m\'elange d'arguments g\'eom\'etriques et de th\'eorie des repr\'esentations de $\gl_2(F)$, plus pr\'ecis\'ement les r\'esultats de Hu~\cite{YHu}, Schraen~\cite{schraen}, Shotton~\cite{Sho}, 
Vign\'eras~\cite{vigne} et Wu~\cite{Wu}:

\begin{theo} \label{main3}
Soit $F$ une extension finie de $\Q_p$ distincte de
$\Q_p$. 

{\rm (i)} Si $K$ est une extension finie de $F$, 
alors $H^1_{\eet}({\cal M}_{n,K}^p,\bmu_p)^{\vee}$  est une repr\'esentation lisse de pr\'esentation
finie de $\gl_2(F)$. Si $n\geq 1$ et si $K$ est assez grande, 
cette repr\'esentation n'est pas admissible. 

{\rm (ii)}
Si $n\geq 1$, il existe
une repr\'esentation continue
$\bar{\rho}:\G_{F}\to {\rm GL}_2(k_L)$ telle que 
$\bPi^{\rm geo}_n(\bar{\rho})$ ne soit pas admissible.
\end{theo}

Le th\'eor\`eme ci-dessus et sa forme pr\'ecis\'ee du th.\,\ref{main10} (compl\'et\'e par la
rem.\,\ref{copr8}) posent des probl\`emes int\'eressants sur la forme que doit prendre
la correspondance de Langlands
locale $p$-adique pour $\gl_2(F)$, si $F\neq \Q_p$: les probl\`emes
rencontr\'es pour la classification
des repr\'esentations modulo~$p$ de $\gl_2(F)$ se refl\`etent dans la cohomologie $p$-adique de 
la tour de Drinfeld (aucun de ces probl\`emes n'appara\^{\i}t 
en $\ell$-adique, avec $\ell\ne p$).
 Comme nous l'ont fait remarquer Dotto et Emerton, il est probable que les repr\'esentations lisses de pr\'esentation finie $\bPi^{\rm geo}_n(\bar{\rho})$ (avec $\bar{\rho}: \mathcal{G}_F\to \gl_2(k_L)$ une repr\'esentation continue, et $F$ non ramifi\'e) aient un lien avec les repr\'esentations construites par Breuil et \paskunas \,\cite{BP}, ce qui fournirait une interpr\'etation g\'eom\'etrique de ces derni\`eres.
 Nous esp\'erons revenir sur ce probl\`eme dans un travail ult\'erieur.

\Subsection{Finitude de la cohomologie de la tour de Drinfeld}\label{noti1}
D\'ecrivons maintenant nos r\'esultats plus en d\'etail, en commen\c{c}ant par esquisser les preuves
des th.\,\ref{Intro1}, \ref{main2} et \ref{main3}.
\subsubsection{La tour de Drinfeld}\label{noti2}
Soient $F$ une extension finie de $\Q_p$,
$\O_F$ l'anneau de ses entiers et ${\goth m}_F$ l'id\'eal maximal de $\O_F$.
On d\'efinit les groupes:

$\bullet$ $G=\gl_2(F)$, $G_0=\gl_2(\O_F)$ et $G_n=1+{\rm M}_2({\goth m}_F^n)$ si $n\geq 1$.

$\bullet$ $\check G=D^\dual$, o\`u
$D$ est l'alg\`ebre de quaternions non d\'eploy\'ee
de centre~$F$, 
$\check{G}_0=
\O_D^\dual$ o\`u $\O_D$ est l'ordre maximal de $D$, 
$\check{G}_n=1+{\goth m}_D^n$, si $n\geq 1$ et ${\goth m}_D$ est l'id\'eal maximal de $\O_D$.

\vskip.1cm
Si $H=G,\check G,{\rm W}_F$, on dispose d'un morphisme
de groupes naturel $\nu_H:H{\hskip.5mm\to\hskip.5mm}F^\dual,$
o\`u $\nu_G=\det$, $\nu_{\check G}$ est la norme r\'eduite,
et $\nu_{{\rm W}_F}:{\rm W}_F\to {\rm W}_F^{\rm ab} \simeq F^\dual$ 
est fourni par la th\'eorie locale du corps de classes.

\vskip.2cm
Le demi-plan $p$-adique (de Drinfeld)
$\piqp_F\moins\piqp(F)$ admet une structure naturelle d'espace analytique
rigide $\Omega_{{\rm Dr},F}$ sur $F$, et une action de $G$ par homographies,
qui respecte cette structure.
Drinfeld a d\'efini \cite{Drinfeld1}
une tour de rev\^etements ${\cal M}_{n,\breve{F}}$, pour $n\in\N$,
de ce demi-plan, v\'erifiant les propri\'et\'es suivantes:

$\bullet$ ${\cal M}_{n,\breve{F}}$ est d\'efini sur $\breve{F}=\widehat{F^{\rm nr}}$
et muni d'une action de ${\rm W}_F$ compatible avec l'action naturelle sur $\breve{F}$.

$\bullet$ ${\cal M}_{n,\breve{F}}$ est muni d'une action de $G\times \check G$
commutant avec l'action de ${\rm W}_F$,
et les fl\`eches de transition ${\cal M}_{n+1,\breve{F}}\to {\cal M}_{n,\breve{F}}\to
\Omega_{{\rm Dr},F}$ sont ${\rm W}_F\times \check G\times G$-\'equivariantes (l'action de
$\check G$ sur $\Omega_{{\rm Dr},F}$ \'etant l'action triviale).

$\bullet$ ${\cal M}_{0,\breve{F}}=\Z\times\Omega_{{\rm Dr},\breve{F}}$ et,
si $n\geq 1$, alors ${\cal M}_{n,\breve{F}}$ est un rev\^etement galoisien
de ${\cal M}_{0,\breve{F}}$, de groupe de Galois $\check{G}_0/\check{G}_n$.

$\bullet$ L'ensemble $\pi_0({\cal M}_{n,\C_p})$
des composantes connexes de ${\cal M}_{n,\C_p}$ est un espace homog\`ene principal 
sous l'action de $F^\dual/(1+{\goth m}_F^n)$ et $H=G,\check G,{\rm W}_F$ 
agit sur $\pi_0({\cal M}_{n,\C_p})$ \`a travers $\nu_H:H\to F^\dual$.  

\vskip.2cm
Si $\varpi$ est une uniformisante de $F$, on note ${\cal M}_{n,\breve F}^\varpi$
le quotient de ${\cal M}_{n,\breve F}$ par $\matrice{\varpi}{0}{0}{\varpi}\in G$. On pose aussi 
$$G':=G/\matrice{\varpi}{0}{0}{\varpi}^{\mathbf{Z}}.$$
Alors ${\cal M}_{n,\breve F}^\varpi$ est l'extension des scalaires de $F$ \`a $\breve F$
d'un espace analytique ${\cal M}_{n,F}^\varpi$ d\'efini sur $F$ et muni d'une action de $G'$.
Si $K$ est une extension finie de $F$ ou $\C_p$, on note
${\cal M}_{n,K}^\varpi$ l'espace analytique sur $K$
obtenu par extension des scalaires de $F$ \`a $K$.

\subsubsection{R\'esultats de finitude}\label{noti3}
Notre premier r\'esultat (th.\,\ref{smooth}) est valable pour $F$ quelconque:
\begin{theo}\label{finit1}
Si $K$ est une extension finie de $F$ et $k\geq 1$ alors $H^q_{\eet}({\cal M}_{n,K}^\varpi,\O_L/p^k)$ est un 
$\O_L$-module profini, dont 
le dual
de Pontryagin
est une repr\'esentation lisse de~$G$, de pr\'esentation finie.
\end{theo}

 En passant sous silence des contorsions topologiques un peu p\'enibles, 
la preuve de ce r\'esultat utilise trois ingr\'edients:

$\bullet$ l'existence \cite{CDN1} d'un mod\`ele semi-stable
$G$-\'equivariant pour ${\cal M}_{n,K}^\varpi$ ($K$ assez grand),

$\bullet$ la filtration de Bloch-Kato-Hyodo \cite{BK}, \cite{Hy}
sur les cycles proches modulo $p$ pour ramener le calcul de la cohomologie
\'etale \`a celle de faisceaux coh\'erents sur la fibre sp\'eciale,

$\bullet$ les propri\'et\'es de la cat\'egorie des repr\'esentations de pr\'esentation finie de $G'$
(en particulier le fait qu'elle est ab\'elienne et stable par extensions d'apr\`es Shotton~\cite{Sho}).

\begin{rema}
(i) Si on disposait d'une dualit\'e de Poincar\'e, on pourrait prouver
la premi\`ere partie du th\'eor\`eme de mani\`ere plus naturelle: les duaux
de Pontryagin des $H^q_{\eet}({\cal M}_{n,K}^\varpi,\O_L/p^k)$
seraient les groupes de cohomologie \`a support compact, pour lesquels
Berkovich~\cite{BV} a \'etabli la lissit\'e de l'action de $G$. Malheureusement, l'existence d'une telle dualit\'e pour des espaces analytiques comme 
${\cal M}_{n,K}^\varpi$ n'est pas connue.

(ii) Il est probable qu'un tel \'enonc\'e reste valable pour la tour de Drinfeld associ\'ee \` a
${\rm GL}_d(F)$ avec $d>2$, mais deux des ingr\'edients ci-dessus (le premier et le troisi\`eme) ne sont plus disponibles dans ce cadre.

(iii) Il n'est pas du tout clair si $H^1_{\eet}({\cal M}_{n, \C_p}^\varpi, k_L)$ est profini. Les techniques utilis\'ees ci-dessus utilisent de mani\`ere cruciale le fait que l'on travaille sur une extension finie de $\Q_p$, pas sur $\C_p$ (la filtration du th.\,\ref{generalized} utilise pleinement
la finitude de la ramification).
\end{rema}

Un autre r\'esultat valable pour $F$ quelconque est le suivant (th.\,\ref{scholzevariation}):
   
      \begin{theo}\label{finitgl2}
   Si $\pi$ est un $k_L[\gl_2(F)]$-module lisse, admissible, \`a caract\`ere central, alors 
   ${\rm Hom}_{k_L[\gl_2(F)]}^{\rm cont}(\pi^\vee,H^1_{\eet}({\cal M}_{n,\C_p},k_L))$ est de dimension
finie sur $k_L$.
\end{theo}

\begin{rema}\label{finit4}
{\rm (i)} Pour $F=\Q_p$, ce th\'eor\`eme est une version modulo~$p$ du th.\,\ref{icdn1.2} ci-dessous.
Ce dernier se d\'emontre par les m\'ethodes de~\cite{CDN1}, mais il ne semble pas possible
d'en d\'eduire le th.\,\ref{finitgl2} dans ce cas.

{\rm (ii)} Il semble raisonnable de penser que le th.\,\ref{finitgl2} 
s'\'etend \`a la tour de Drinfeld pour ${\rm GL}_d(F)$ sous la forme:
{\it les groupes ${\rm Ext}^i_{k_L[\gl_d(F)]}(\pi^\vee, H^j_{\eet}({\cal M}_{n,\C_p}, k_L))$ 
sont de dimension finie sur $k_L$}. 
Ce r\'esultat (s'il est vrai) serait 
un analogue d'un r\'esultat classique pour la cohomologie $\ell$-adique, avec $\ell\ne p$.

{\rm (iii)} Si, dans l'\'enonc\'e,
on remplace ${\cal M}_{n,\C_p}$ par ${\cal M}_{n, K}$ avec $[K:\Q_p]<\infty$, 
le r\'esultat est une cons\'equence imm\'ediate du th.\,\ref{finit1}, mais
il ne semble pas possible d'en tirer directement le th.\,\ref{finitgl2}. 
\end{rema}

Le r\'esultat suivant (\S\,\ref{lg6}) est sp\'ecifique \`a $F=\Q_p$, et nettement plus difficile \`a d\'emontrer que les deux th\'eor\`emes ci-dessus:

\begin{theo}\label{finit2}
Si $[K:\Q_p]<\infty$, alors 
${\rm Hom}_{k_L[G']}^{\rm cont}(\pi^\vee,H^1_{\eet}({\cal M}_{n,K}^p,k_L))=0$ 
pour presque tout {\rm (i.e., sauf pour un nombre fini)} $k_L[G']$-module $\pi$ lisse, irr\'eductible.
\end{theo}
\begin{rema}\label{finit3}
Ce th\'eor\`eme
implique (et est \'equivalent si on injecte un peu de th\'eorie des repr\'esentations de $G$, cf. le th.\,\ref{LTappli}),
 via l'isomorphisme de Faltings entre les tours de Drinfeld et Lubin-Tate compl\'et\'ees,
 que les vecteurs lisses de\footnote{L'espace ${\rm LT}_{n,K}^p$ est obtenu \`a partir de l'espace de Lubin-Tate de niveau 
 $1+p^nM_2(\Z_p)$ de la m\^eme mani\`ere que $\mathcal{M}_{n,K}^p$ est obtenu \`a partir de 
 $\mathcal{M}_{n, \C_p}$.}  $H^1_{\eet}({\rm LT}_{n,K}^p,k_L)$ 
pour l'action de $\check{G}$ sont de dimension finie (et proviennent des composantes connexes). 
On peut se demander s'il est possible de prouver ce r\'esultat directement. 
S'il y a des vecteurs lisses ne provenant pas des composantes connexes, 
il y a des vecteurs invariants 
par $\check{G}_1\times\matrice{1+p\Z_p}{\Z_p}{p\Z_p}{1+p\Z_p}$, 
et donc dans la cohomologie de ${\rm LT}_{1,K}$, un espace dont les composantes connexes sont
des couronnes ouvertes.
Si on est tr\`es optimiste\footnote{C'est le cas pour les exemples simples que nous avons consid\'er\'es et pour l'exemple ci-dessus...}, 
on peut esp\'erer que, si $H$ est un groupe de Lie compact agissant
sur une couronne ouverte $C_K$ avec action non triviale
sur $\O^+(C_K)/{\goth m}_K$, les vecteurs $H$-invariants de $H^1_{\eet}(C_K,k_L)$ sont tr\`es petits.
\end{rema}

Comme nous n'avons pas r\'eussi \`a faire marcher les strat\'egies \'evoqu\'ees dans les rem.\,\ref{finit4}
et~\ref{finit3}, notre
preuve du th.\,\ref{finit2} emprunte des chemins plus d\'etourn\'es.  On suppose pour la suite de ce paragraphe que $F=\Q_p$. 

$\bullet$ Un ingr\'edient crucial
est une suite spectrale (dont l'existence est d\'emontr\'ee dans le m\^eme degr\'e de g\'en\'eralit\'e que celui de l'article \cite{SLT}) reliant les foncteurs de Scholze~\cite{SLT} $\pi\mapsto S^i(\pi)$
aux groupes qui nous int\'eressent. En injectant aussi le calcul de $H^i(\SL_2(\Q_p),\pi)$
par Fust~\cite{Fust}, cette suite spectrale permet de d\'emontrer
le r\'esultat suivant (cor.\,\ref{cafe52}):

\begin{prop}\label{icafe52}
Soit $\pi$ un $k_L[G]$-module lisse admissible, absolument irr\'eductible. On dispose d'un morphisme naturel\footnote{
$\mathcal{M}_{\infty,\C_p}$ est l'espace perfecto\"{\i}de compl\'et\'e de la tour des
$\mathcal{M}_{n,\C_p}$; il est isomorphe au compl\'et\'e ${\rm LT}_{\infty,\C_p}$ de la tour de Lubin-Tate,
et le foncteur de Scholze exploite le fait que ${\rm LT}_{\infty,\C_p}$ est un rev\^etement de $\piqp_{\C_p}$
de groupe de Galois $G$.}
$$S^1(\pi)
     \to {\rm Hom}_{k_L[G]}^{\rm cont}(\pi^{\vee}, H^1_{\eet}(\mathcal{M}_{\infty,\C_p},
   k_L))$$
    dont les noyau et conoyau sont
de dimension finie sur $k_L$, et qui est un isomorphisme 
   si $\pi$ n'appartient pas \`a un twist 
$\{\chi,{\rm St}\otimes\chi, I(\chi,\chi\epsilon)\}$ du bloc de la Steinberg.
\end{prop}
L'admissibilit\'e de $S^1(\pi)$ comme repr\'esentation de $\check G$
permet d'en d\'eduire (via les r\'esultats de Fust) le th.\,\ref{finitgl2}.

\smallskip
$\bullet$ Un autre ingr\'edient est une globalisation de la situation qui permet d'utiliser
les r\'esultats d'isotypie de Carayol et Scholze~\cite{SLT} pour analyser les repr\'esentations de $\G_{\Q_p}$
intervenant dans
${\rm Hom}_{k_L[G']}^{\rm cont}(\pi^\vee,H^1_{\eet}({\cal M}_{n,\C_p}^p,k_L))$
et en d\'eduire le th.\,\ref{finit2}.
 Pour pouvoir bien globaliser on utilise des r\'esultats de Gee et Kisin \cite{GK}, et pour \'etudier la cohomologie compl\'et\'ee attach\'ee aux alg\`ebres de quaternions globales (d\'eploy\'ees en toute place finie) issues de la globalisation nous utilisons les 
 travaux de \paskunas \,\cite{Pask} et \paskunas-Tung \cite{PT}.
 Ces arguments globaux permettent de contr\^oler l'action de $\mathcal{G}_{\qp}$ sur $S^1(\pi)$, et en particulier de montrer que pour toute extension finie $K$ de $\qp$ il n'y a qu'un nombre fini de $\pi$ comme dans le corollaire ci-dessus et tels que $S^1(\pi)^{\mathcal{G}_K}\ne 0$.

\medskip
En utilisant le fait (th.\,\ref{pierre1}) qu'une repr\'esentation de type fini de $G$ est de longueur finie
si et seulement si son cosocle est de longueur finie (un r\'esultat qui utilise pleinement la
classification des repr\'esentations lisses irr\'eductibles modulo~$p$ de $G$, classification
non connue si $F\neq \Q_p$, ainsi que les r\'esultats de finitude pour les groupes d'extensions entre deux repr\'esentations lisses irr\'eductibles de $G$), on d\'eduit du th.\,\ref{finit2} le r\'esultat suivant.
\begin{coro}\label{Intro3}
{\rm (i)} Si $[K:\Q_p]<\infty$, le dual de
$H^1_{\eet}({\cal M}_{n,K}^p,k_L)$ est une repr\'esentation lisse de $G$, de longueur finie.

{\rm (ii)} Si $\rho:\G_{\Q_p}\to {\rm GL}_2(k_L)$ est une repr\'esentation continue,
le dual de ${\rm Hom}_{k_L[\G_{\Q_p}]}^{\rm cont}(\rho,H^1_{\eet}({\cal M}_{n,\C_p}^p,k_L))$
est une repr\'esentation lisse de $G$, de longueur finie.
\end{coro}

\begin{rema}\label{Intro4}
(i) Dans le cas $F=\Q_p$, une $k_L$-repr\'esentation de longueur finie de $G$ est automatiquement admissible
et de pr\'esentation finie.

(ii) Dans le cas $F\neq \Q_p$, une $\overline{k}_L$-repr\'esentation de pr\'esentation finie de $G$
est admissible si et seulement si elle est de longueur finie, mais les $k_L$-repr\'esentations
irr\'eductibles supersinguli\`eres de $G$ ne sont pas forc\'ement admissibles 
et ne sont pas de pr\'esentation finie \cite{schraen}, \cite{Wu}; 
il existe m\^eme \cite{Le} de telles repr\'esentations qui deviennent de longueur infinie quand on \'etend
les scalaires \`a $\overline{k}_L$.
\end{rema}

\Subsection{Multiplicit\'es g\'eom\'etriques de repr\'esentations de $\G_{\Q_p}$ et $G$}\label{noti4}
Passons \`a la preuve du th.\,\ref{intro10}.
\subsubsection{Types et objets associ\'es}\label{noti5}
Soit $F$ une extension finie de $\Q_p$.
Soit $M$ un $L$-$(\varphi,N,\G_F)$-module (i.e. un $L\otimes \Q_p^{\rm nr}$-module
muni d'actions d'un frobenius semi-lin\'eaire $\varphi$, d'un op\'erateur
$N$ tel que $N\varphi=p\varphi N$ et d'une action semi-lin\'eaire
lisse de $\G_F$), de rang $2$. On associe\footnote{
${\rm WD}(M)$ est obtenue \`a partir de $M$ par la recette de Fontaine~\cite{FonAst},
${\rm LL}(M)$ \`a partir de ${\rm WD}(M)$ par
la correspondance de Langlands locale et ${\rm JL}(M)$ \`a partir de
${\rm LL}(M)$ par
la correspondance de Jacquet-Langlands locale (en particulier, ${\rm JL}(M)=0$
si ${\rm LL}(M)$ est une s\'erie principale, i.e.~si $M$ n'est ni supercuspidal ni sp\'ecial). 

Les caract\`eres centraux de ${\rm LL}(M)$ et ${\rm JL}(M)$ sont \'egaux et co\"{\i}ncident
avec $\det{\rm WD}(M)\cdot|\ |$ (vu comme caract\`ere
de $W_{F}^{\rm ab}\simeq F^{\dual}$, le frobenius arithm\'etique s'envoyant sur $p$).}
\`a $M$:

\quad $\bullet$ une $L$-repr\'esentation ${\rm WD}(M)$ de ${\rm WD}_{F}$,
de dimension~$2$,

\quad $\bullet$ une $L$-repr\'esentation lisse irr\'eductible ${\rm LL}(M)$ de $G$,

\quad $\bullet$ une $L$-repr\'esentation lisse irr\'eductible (de dimension finie) ${\rm JL}(M)$ de~$\check G$.

\vskip.1cm
Remarquons que, si ${\rm WD}(M)$ est irr\'eductible, les pentes de $\varphi$ sont toutes
\'egales \`a un m\^eme nombre rationnel appel\'e {\it la pente de $M$}.
On dit
que $M$ est:

\quad $\bullet$ {\it supercuspidal},
si ${\rm WD}(M)$ est irr\'eductible et de pente~$\frac{1}{2}$,

\quad $\bullet$ {\it sp\'ecial} si\footnote{\label{special}
C'est
donc un tordu ${\rm Sp}\otimes\eta$, o\`u $\eta$ est un caract\`ere lisse de
$F^\dual$, du module ${\rm Sp}$ d\'efini par
$${\rm Sp}=\Q_p^{\rm nr}e_1\oplus \Q_p^{\rm nr}e_2,\quad
\varphi(e_1)=e_1,\ \varphi(e_2)=pe_2,\hskip.2cm Ne_1=0,\ Ne_2=e_1.$$} $N\neq 0$ et si les pentes de $\varphi$
sont $0$ et $1$.
 
\quad $\bullet$ {\it de niveau~$\leq n$} si ${\rm JL}(M)$ se factorise \`a travers $\check G/\check G_n$.

\vskip.2cm
Supposons maintenant que $F=\Q_p$.
Si $M$ est supercuspidal, 
on d\'efinit le $L$-module de rang~$2$
 $$M_{\rm dR}=(\Qbar_p\otimes_{\Q_p^{\rm nr}}M)^{\G_{\Q_p}}.$$
Si ${\cal L}$ est une $L$-droite de $M_{\rm dR}$, on d\'efinit
la repr\'esentation $V_{M,{\cal L}}$ de $\G_{\Q_p}$ par 
$$V_{M,{\cal L}}={\rm Ker}\big((\bcris^+\otimes_{\Q_p^{\rm nr}} M)^{\varphi=p}\to
{\C_p}\otimes_{\Q_p} (M_{\rm dR}/{\cal L})\big).$$
Il r\'esulte de~\cite{CF} que $V_{M,{\cal L}}$
est une $L$-repr\'esentation de dimension~$2$, potentiellement semi-stable
 \`a poids $0$ et $1$, dont le $D_{\rm pst}$ est $M$, et toute telle repr\'esentation
est de la forme $V_{M,{\cal L}}$.
On note $\Pi_{M,{\cal L}}$ la repr\'esentation de $G$ associ\'ee \`a $V_{M,{\cal L}}$
par la correspondance de Langlands locale $p$-adique: 
si $\Pi\mapsto{\bf V}(\Pi)$ est le foncteur r\'ealisant la correspondance de
Langlands locale $p$-adique~\cite{gl2,CDP},
on a 
$${\bf V}(\Pi_{M,{\cal L}})\simeq V_{M,{\cal L}}$$

On dit qu'une $L$-repr\'esentation $V$ 
de $\G_{\Q_p}$ {\it est de niveau~$\leq n$} si $V\simeq V_{M,{\cal L}}$ pour un $M$
 supercuspidal de niveau~$\leq n$ et une droite ${\cal L}$ de $M_{\rm dR}$ 
(en particulier, $V$ est de dimension~$2$).

\subsubsection{Multiplicit\'es individuelles}\label{noti6}
Le th.\,\ref{icdn1.1} ci-dessous est l'un des r\'esultats principaux de~\cite{CDN1}. 
   \begin{theo}\label{icdn1.1}
Si $V\in {\rm Rep}_L\G_{\Q_p}$ est absolument irr\'eductible, de dimension~$2$, alors
    $${\rm Hom}_{{\rm W}_{\Q_p}}(V, H^1_{\rm et}(\mathcal{M}_{n, \C_p}, L(1)))\simeq
\begin{cases}\Pi_{M,{\cal L}}^\dual\otimes {\rm JL}(M) 
&{\text{si $V=V_{M,{\cal L}}$, et $M$ de niveau~$\leq n$}}\\
0 &{\text{si $V$ n'est pas de niveau~$\leq n$.}}\end{cases}$$
   \end{theo}
\begin{rema}\label{icdn1.3}
On y prouve aussi que ${\rm Hom}_{{\rm W}_{\Q_p}}(V, H^1_{\rm et}(\mathcal{M}_{n, \C_p}, L(1)))=0$
si $V$ est absolument irr\'eductible, de dimension~$\geq 3$ (ce r\'esultat
est un des plus d\'elicats de~\cite{CDN1}; voir le~th.\,\ref{aju3} pour une
preuve nettement plus limpide).
En r\'esum\'e, $H^1_{\rm et}(\mathcal{M}_{n, \C_p}, L(1))$ contient exactement les repr\'esentations
de $\G_{\Q_p}$ que l'on souhaite y voir, avec la multiplicit\'e id\'eale pour
une r\'ealisation g\'eom\'etrique de la correspondance de Langlands locale $p$-adique
pour $\gl_2(\Q_p)$.
\end{rema}
Les techniques utilis\'ees dans la preuve du th.\,\ref{icdn1.1} permettent aussi de prouver le r\'esultat
suivant (th.\,\ref{cdn1.2}) qui joue un grand r\^ole dans nos r\'esultats en famille.
   \begin{theo}\label{icdn1.2}
Si $V\in {\rm Rep}_L\G_{\Q_p}$ est absolument irr\'eductible, de dimension~$2$, alors
    $${\rm Hom}_G(\Pi(V)^\dual, H^1_{\rm et}(\mathcal{M}_{n, \C_p}, L(1)))\simeq 
\begin{cases}V\otimes {\rm JL}(M) &{\text{si $V=V_{M,{\cal L}}$, et $M$ de niveau~$\leq n$}}\\
0 &{\text{si $V$ n'est pas de niveau~$\leq n$.}}\end{cases}$$
   \end{theo}

\subsubsection{Anneaux de Kisin}\label{noti7}
Soit $M$ un $(\varphi,N,\G_{\Q_p})$-module sp\'ecial ou supercuspidal.
Soit $\delta_M$ le caract\`ere central de ${\rm LL}(M)$.
On note ${\rm Rep}^{\delta_M}\,G$ la cat\'egorie des $\O_L[G]$-modules lisses, de longueur finie
et dont le caract\`ere central est $\delta_M$.

Il y a une bijection ${\cal B}\leftrightarrow\rho_{\cal B}$ entre blocs de
${\rm Rep}^{\delta_M}\,G$ et orbites sous ${\rm Gal}(\overline{\bf F}_p/k_L)$ de 
$\overline{\bf F}_p$-repr\'esentations semi-simples de dimension $2$ de $\G_{\Q_p}$, de d\'eterminant
$\delta_M\epsilon$, o\`u $\epsilon$ d\'esigne le caract\`ere cyclotomique.
On note $R_{\cal B}^{{\rm ps},\delta_M}$ l'anneau des d\'eformations universelles
de d\'eterminant $\delta_M\epsilon$
du pseudo-caract\`ere ${\rm Tr}\circ\rho_{\cal B}$.

Si ${\cal B}$ est un bloc, on d\'efinit:
\begin{align*}
&P_{\cal B}=\oplus_{\pi\in{\cal B}} P_{\pi};&&{\text{$P_{\pi}$, enveloppe projective
de $\pi^\vee$}};\\
&E_{\cal B}={\rm End}_GP_{\cal B};&&{\text{$Z_{\cal B}$, centre de $E_{\cal B}$}}.
\end{align*}
On dispose du r\'esultat fondamental suivant ({\og th\'eor\`eme $R=T$\fg} local de \paskunas~\cite{Paskext} -- 
et~\cite{PT} pour les cas exceptionnels).
\begin{theo}\label{Intro6} {\rm (\paskunas, \paskunas-Tung; \cite{Paskext,PT})}
Il existe une fl\`eche naturelle $R_{\cal B}^{{\rm ps},\delta_M}\to Z_{\cal B}$, 
et cette fl\`eche induit
un isomorphisme
$$R_{\cal B}^{{\rm ps},\delta_M}[\tfrac{1}{p}]\simeq Z_{\cal B}[\tfrac{1}{p}]$$
\end{theo}

\noindent $\bullet$ 
Si $M={\rm Sp}\otimes\eta$ est sp\'ecial (note~\ref{special}), 
on pose $R_{M,{\cal B}}=L$.
On pose aussi
$\rho_{{\cal B},M}=0$ et $\bPi^\dual(\rho_{{\cal B},M})=0$, sauf si
${\cal B}$ est le bloc de
${\rm St}\otimes\bar\eta$, o\`u $\bar\eta$ est la r\'eduction modulo~$p$ de
$\eta$, auquel cas 
on pose
$\rho_{{\cal B},M}=\eta$ 
(vu comme caract\`ere de $\G_{\Q_p}$
via la th\'eorie locale du corps de classes), 
et $\bPi^\dual(\rho_{{\cal B},M})=({\rm St}^{\rm cont}\otimes\eta)^\dual$.

\vskip.1cm
\noindent $\bullet$ Si $M$ est supercuspidal,
on note:

$\diamond$ $R_{{\cal B},M}$ le quotient de $R_{\cal B}^{{\rm ps},\delta_M}[\tfrac{1}{p}]$ param\'etrant
les repr\'esentations de type $M$ (i.e., potentiellement semi-stables, \`a poids $0$ et $1$, dont le $D_{\rm pst}$
est isomorphe \`a $M$), 

$\diamond$ $\rho_{{\cal B},M}$ la $R_{{\cal B},M}$-repr\'esentation de $\G_{\Q_p}$
de dimension~$2$ interpolant les repr\'esentations de type $M$.

(L'existence
de $R_{{\cal B},M}$ et $\rho_{{\cal B},M}$ est due \`a Kisin~\cite{Kis} dans la plupart des cas; voir \S\ref{fami1} pour des compl\'ements.
Notons que, via l'isomorphisme du th.\,\ref{Intro6}, 
$R_{{\cal B},M}$ est aussi un quotient de $Z_{\cal B}[\frac{1}{p}]$.)

$\diamond$ $\bPi^\dual(\rho_{{\cal B},M})$ la repr\'esentation de $G$ interpolant les $\bPi(\rho_x)^\dual$
pour $x\in{\rm Spm}(R_{{\cal B},M})$: si $\rho^\diamond_{{\cal B},M}$ est la
$R_{{\cal B},M}$-duale de $\rho_{{\cal B},M}$, on a (cf.~\cite[\S\S\,II.2.4 et II.3.1]{gl2})
$$\bPi^\dual(\rho_{{\cal B},M})=D^\natural(\rho^\diamond_{{\cal B},M}\otimes\epsilon)\boxtimes\piqp$$ 

Dans les deux cas, 
$$\check{\bf V}(\bPi^\dual(\rho_{{\cal B},M}))\simeq\rho_{{\cal B},M}$$

\subsubsection{R\'esultats en famille}\label{noti8}
On pose
$$H^1_{\eet}({\cal M}^p_{n,\Qbar_p},L(1)):=L\otimes_{\O_L}\big(
{\varprojlim}_k\big({\varinjlim}_{[K:\Q_p]<\infty}
H^1_{\eet}({\cal M}^p_{n,K},(\O_L/p^k)(1))\big)\big).$$
C'est un sous-objet de $H^1_{\eet}({\cal M}^p_{n,\C_p},L(1))$, 
stable par $G$, $\check G$ et $\G_{\Q_p}$,
nettement plus petit que le module initial.

On d\'eduit du th\'eor\`eme de finitude (th.\,\ref{Intro1}) et de la th\'eorie de Gabriel~\cite{gaby}, 
une d\'ecomposition en blocs
des $H^1_{\eet}({\cal M}^p_{n,K},(\O_L/p^n)(1))$. Par passage \`a la limite, cela fournit
une d\'ecomposition
$$H^1_{\eet}({\cal M}^p_{n,\Qbar_p},L(1))\simeq 
\oplus_M\big(\widehat\oplus_{\cal B}(P_{\cal B}\otimes_{E_{\cal B}}
{\mathbbm m}_{{\cal B},M})\big)\otimes_L{\rm JL}(M)$$
o\`u:
$${\mathbbm m}_{{\cal B},M}={\rm Hom}_{L[G\times\check G]}(P_{\cal B}\otimes{\rm JL}(M),
H^1_{\eet}({\cal M}^p_{n,\Qbar_p},L(1)))$$

Le th.\,\ref{intro10} est alors une cons\'equence directe du r\'esultat suivant (th.\,\ref{factor5}):
\begin{theo}\label{Intro7} 
{\rm (i)} 
Le dual de ${\mathbbm m}_{{\cal B},M}$ est un $Z_{\cal B}[\frac{1}{p}]$-module de type fini
sur lequel $Z_{\cal B}[\frac{1}{p}]$ agit par son quotient $R_{{\cal B},M}$.

{\rm (ii)}  Supposons\footnote{Si $p=3$, la preuve s'\'etend aux blocs ne contenant pas un twist
de la steinberg.}
 que $p>3$. On a un isomorphisme de $E_{\cal B}^{\delta_M}[\G_{\Q_p}]$-modules
$${\mathbbm m}_{{\cal B},M}\simeq {\rm Hom}_{L[G]}(P_{\cal B},\bPi^\dual(\rho_{{\cal B},M}))\otimes\rho_{{\cal B},M}\otimes\check R_{{\cal B},M},$$
o\`u les produits tensoriels sont au-dessus de $R_{{\cal B},M}$.
\end{theo}
Le premier \'enonc\'e du (i) r\'esulte formellement du (i) du th.\,\ref{finit2}, le second
\'enonc\'e est une cons\'equence du th.\,\ref{icdn1.2}. La d\'ecomposition du (ii) 
est une interpolation de la d\'ecomposition en tout point ferm\'e de ${\rm Spec}\,R_{{\cal B},M}$
fournie par le th.\,\ref{icdn1.2} 
(comme $R_{{\cal B},M}$ est un produit d'anneaux principaux, cette identification
en tout point implique une identification globale).

 \section{Repr\'esentations de $\gl_2(F)$}
Soient $F$ une extension finie de $\Q_p$, $\O_F$ l'anneau de ses entiers, $k_F$
son corps r\'esiduel et $\varpi$ une uniformisante.
On note:

$\bullet$ $G$ le groupe $\gl_2(F)$,

$\bullet$ $Z=\big\{\matrice{a}{0}{0}{a},\ a\in F^\dual\big\}$ le centre de $G$,

$\bullet$ $K=\gl_2(\O_F)$ le sous-groupe compact maximal standard de $G$,

$\bullet$ $B=\matrice{F^\dual}{F}{0}{F^\dual}$ le borel sup\'erieur,

$\bullet$ $G'$ le groupe quotient $G/\matrice{\varpi}{0}{0}{\varpi}^\Z$.
\subsection{G\'en\'eralit\'es}
Soient $L$ une extension finie de $\Q_p$, $\O_L$ l'anneau de ses entiers, ${\goth m}_L$
l'id\'eal maximal de $\O_L$, et $k_L=\O_L/{\goth m}_L$ son corps r\'esiduel.
\subsubsection{Modules lisses de torsion}
Soit $H$ un groupe de Lie $p$-adique (dans les applications, $H$ sera $G$ ou $G'$).
Un $\O_L[H]$-module $\pi$ est dit:

\quad $\bullet$ {\it de torsion} si tout vecteur $v\in \pi$ est tu\'e par $p^N$, pour $N$ assez grand;

\quad $\bullet$ {\it lisse} si $\pi$ est de torsion et si le stabilisateur dans $H$ de tout $v\in\pi$ est ouvert.

On note ${\rm Rep}^{\rm lisse}\,H$ la cat\'egorie des $\O_L[H]$-modules lisses.
On remarquera que les objets irr\'eductibles de ${\rm Rep}^{\rm lisse}\,H$ 
sont tu\'es par ${\goth m}_L$
et donc sont des $k_L[H]$-modules. Il n'est pas clair si une extension (dans la cat\'egorie des 
$\O_L[H]$-modules abstraits) de $\O_L[H]$-modules lisses est encore un $\O_L[H]$-module lisse. Le r\'esultat beaucoup plus faible suivant nous sera cependant tr\`es utile par la suite.

\begin{lemm}\label{extension}
Soient $U$ un sous-groupe ouvert de $H$ et
$0\to A\to B\to C\to 0$ une suite exacte de $\O_L[U]$-modules.
On suppose que 
 $p^a$ tue $A$, et que $U$ agit trivialement sur $A$ et sur $C$. Alors 
 $H$ poss\`ede un sous-groupe ouvert agissant trivialement sur $B$.
\end{lemm}
\begin{proof} Quitte \`a remplacer $U$ par un sous-groupe ouvert, on peut supposer que $U$ est un pro-$p$ groupe uniforme, et alors $U_n=\{h^{p^n}|\,\, h\in U\}$ est un sous-groupe ouvert de $U$ pour tout $n\geq 1$. Nous allons montrer que 
$U_{a}$ agit trivialement sur $B$, ce qui permettra de conclure. 
Notons que 
$(h-1)^2\cdot v=0$ pour tout $h\in U$ et $v\in B$. 
Il s'ensuit que $h^{p^a}\cdot v=v+p^a(h-1)\cdot v$ pour 
$h\in U$ et $v\in B$, et comme $(h-1)\cdot v$ est tu\'e par $p^a$
puisqu'il appartient \`a $A$, on a $h^{p^a}\cdot v=v$, d'o\`u le r\'esultat.
\end{proof}

 Soit $K$ un sous-groupe ouvert compact de $H$. Le $\O_L[H]$-module 
$$\O_L\langle H\rangle:=\O_L[H]\otimes_{\O_L[K]} \O_L[[K]]$$
 est muni \cite[prop.\,3.2]{Sho} de l'unique structure de 
   $\O_L$-alg\`ebre    
   telle que les morphismes naturels $\O_L[H]\to \O_L\langle H\rangle $ et $\O_L[[K]]\to \O_L\langle H\rangle$ soient des morphismes de $\O_L$-alg\`ebres, et cette structure 
   est ind\'ependante (\`a isomorphisme canonique pr\`es) du choix du sous-groupe ouvert compact $K$ de $H$ utilis\'e pour la d\'efinir. De plus, par \cite[lemme 3.5]{Sho} la structure de $\O_L[H]$-module de tout objet 
   de ${\rm Rep}^{\rm lisse}\,H$ s'\'etend de mani\`ere unique en une structure de 
   $\O_L\langle H\rangle$-module.

\subsubsection{Admissibilit\'e, finitude} \label{Shotton}
On dit que $\pi\in{\rm Rep}^{\rm lisse}\,H$ est:

$\bullet$ {\it admissible} si $\pi^K[{\goth m}_L^j]$ est de longueur finie sur $\O_L$, 
pour tout sous-groupe ouvert compact $K$ de $H$ et tout $j\geq 1$.

$\bullet$ {\it localement admissible} si $\O_L[H]\cdot v\subset \pi$ est admissible pour tout 
$v\in \pi$, ou, de mani\`ere \'equivalente, si $\pi$ peut s'\'ecrire comme une limite inductive de repr\'esentations lisses admissibles. 

$\bullet$ {\it de longueur finie} si $\pi$ est de longueur finie comme $\O_L[H]$-module.

$\bullet$ {\it de type fini} si $\pi$ est de type fini comme $\O_L[H]$-module.
Les conditions suivantes sont \'equivalentes~\cite[lemmas\,2.4, 3.6]{Sho}: 

\quad (i) $\pi$ est de type fini;

\quad (ii)  $\pi$ est quotient d'une induite compacte
   c-${\rm Ind}_K^H(\sigma)$ pour une $\O_L$-repr\'e\-sentation 
lisse $\sigma$, de longueur finie,
d'un sous-groupe ouvert compact $K$ de~$H$;

\quad (iii) $\pi$ est un $\O_L\langle H\rangle$-module de type fini.

$\bullet$ {\it de pr\'esentation finie}
s'il existe une suite exacte de $\O_L[H]$-modules 
$$\text{c-}{\rm Ind}_{K_1}^H(\sigma_1)\to \text{c-}{\rm Ind}_{K_2}^H (\sigma_2)\to \pi\to 0$$
o\`u $K_1$, $K_2$ sont des sous-groupes ouverts compacts de $H$ et 
$\sigma_i$ est une $\O_L$-repr\'esentation lisse de $K_i$, 
de longueur finie sur $\O_L$.
Cela n'est pas \'equivalent \`a ce que $\pi$ soit un $\O_L[H]$-module de pr\'esentation finie, mais
plut\^ot \cite[prop.\,3.8]{Sho} \`a ce que $\pi$ soit un $\O_L\langle H\rangle$-module de pr\'esentation finie. 

On note 
$${\rm Rep}^{\rm adm}\,H,\quad {\rm Rep}^{\rm ladm}\,H,\quad {\rm Rep}^{\rm \ell f}\,H,\quad {\rm Rep}^{\rm tf}\,H,\quad 
{\rm Rep}^{\rm pf}\,H$$ les sous-cat\'egories pleines de ${\rm Rep}^{\rm lisse}\,H$
des repr\'esentations admissibles, localement admissibles, de longueur finie, de type fini et de pr\'esentation finie,
    respectivement. Si $\delta: H\to \O_L^{\dual}$ est un caract\`ere lisse, on rajoute 
    $\delta$ en exposant pour indiquer la sous-cat\'egorie pleine des repr\'esentations de caract\`ere central 
    $\delta$.

     Le r\'esultat suivant d\'ecoule directement de la discussion ci-dessus et des propri\'et\'es usuelles des modules de type fini et de pr\'esentation finie sur un anneau, mais il peut aussi \^etre d\'emontr\'e directement (et pour des groupes plus g\'en\'eraux que les groupes de Lie $p$-adiques), voir \cite[lemmes 2.6, 2.7]{Sho}: 

    \begin{prop}{\rm (Shotton)}\label{Shot25}
    Soit $0\to \pi_1\to \pi\to \pi_2\to 0$ une suite exacte dans ${\rm Rep}^{\rm lisse}\,H$.

{\rm a)} Si $\pi_1, \pi_2$ sont de type fini {\rm(}resp. de pr\'esentation finie{\rm)}, 
	il en est de m\^eme de $\pi$.

{\rm b)} Si $\pi$ est de pr\'esentation finie et $\pi_1$ est de type fini, $\pi_2$ est de pr\'esentation finie. 
\end{prop} 

Un des r\'esultats principaux de l'article \cite{Sho} de Shotton est le suivant:
\begin{prop}\label{Shot27}
\cite[cor.\,4.4]{Sho}
Si $H$ est un groupe de Lie $p$-adique qui est un produit amalgam\'e de deux sous-groupes ouverts compacts, 
   alors $\O_L\langle H\rangle $ est un anneau coh\'erent, et donc ${\rm Rep}^{\rm pf}\,H$ 
   est une sous-cat\'egorie ab\'elienne de ${\rm Rep}^{\rm lisse}\,H$ 
{\rm (i.e. le noyau et le conoyau d'un morphisme entre 
des repr\'esentations de pr\'esentation finie sont encore de pr\'esentation finie)}.
\end{prop}
(C'est imm\'ediat pour le conoyau, mais pas du tout clair pour le noyau!)

\begin{rema}
(i) $\SL_2(F)$ et 
$\{g\in \gl_2(F),\  \det g\in\O_F^\dual\}$ 
sont des produits amalgam\'es de deux sous-groupes ouverts compacts, 
     d'apr\`es le th\'eor\`eme d'Ihara; le r\'esultat ci-dessus s'applique donc \`a ces groupes.

     (ii)  Si 
     $H$ est un sous-groupe ouvert de $G'$, d'indice fini, alors une 
     repr\'esentation lisse de $G'$ est de type fini (resp. de pr\'esentation finie) si et seulement si 
     sa restriction \`a $H$ l'est~\cite[lemme\,2.8]{Sho}. 
     \end{rema}

     Nous allons utiliser cette remarque pour le groupe 
     $H=\{g\in \gl_2(F),\  \det g\in\O_F^\dual\}$ vu comme sous-groupe ouvert d'indice fini de 
     $G'=\gl_2(F)/\varpi^{\mathbb{Z}}$.

     \subsection{Repr\'esentations irr\'eductibles de $\gl_2(F)$}
Les objets irr\'eductibles de ${\rm Rep}^{\rm lisse}\,G$ sont tu\'es par ${\goth m}_L$,
et donc sont des $k_L$-modules.
     \subsubsection{Repr\'esentations absolument irr\'eductibles}
     Appelons
{\em poids de Serre} un $KZ$-module lisse irr\'eductible sur $\overline{\bf F}_p$, le
centre $Z$ agissant par un caract\`ere (la restriction \`a $K$
		se factorise alors
		par le quotient $ \gl_2(k_F)$).

Si $\sigma$ est un poids de Serre, on note $$I(\sigma)=\text{c-}{\rm Ind}_{KZ}^G(\sigma)$$
l'induite compacte de $\sigma$. 
On dispose d'un isomorphisme \cite[prop.\,8]{BL} de $\overline{\bf F}_p$-alg\`ebres
\begin{equation}\label{hecke}
{\rm End}_G (I(\sigma))\simeq \overline{\bf F}_p[T],
	\end{equation}
	pour un certain op\'erateur de Hecke $T$. De plus, 
	$I(\sigma)$ est un module libre sur $\overline{\bf F}_p[T]$ d'apr\`es \cite[th.\,19]{BL}. 

	\smallskip
	Le th\'eor\`eme de classification de Barthel-Livn\'e \cite{BL0,BL} 
	montre que les repr\'esentations lisses irr\'eductibles de 
	$G$ sur $\overline{\mathbf{F}}_p$, avec un caract\`ere central,
	sont les suivantes:

	$\bullet$ les $\chi\circ\det$ (not\'ees simplement $\chi$), 
	avec $\chi: G\to \overline{\mathbf{F}}_p^\dual$ un caract\`ere lisse.

	$\bullet$ les ${\rm Ind}_B^G(\chi_1\otimes \chi_2)$, 
	o\`u $\chi_1,\chi_2$ sont des caract\`eres lisses distincts de $F^\dual$. 

	$\bullet$ les ${\rm St}\otimes \chi$, 
	o\`u ${\rm St}$ est la steinberg, quotient de 
	${\rm Ind}_B^G(1\otimes 1)$ par la repr\'esentation triviale. 

	$\bullet$ les {\it supersinguli\`eres}, i.e.~les
	quotients irr\'eductibles des $I(\sigma)/(T)$, avec $\sigma$ un poids de Serre.    
	\smallskip

	Les repr\'esentations irr\'eductibles non supersinguli\`eres sont dites {\it ordinaires};
les repr\'esentations ordinaires sont donc les composantes de Jordan-H\"older des repr\'esentations
de la s\'erie principale.

\subsubsection{Questions de rationalit\'e}
Le lemme suivant de {\paskunas} (\cite[lemma 5.1]{Paskext}) nous sera tr\`es utile.
\begin{lemm}\label{51}
Soit $H$ un groupe, $L/K$ une extension de corps et soient $V,W$ de $K[H]$-modules, $V$ \'etant de type fini. Alors 
$${\rm Hom}_{K[H]}(V, W)\otimes_K L\simeq {\rm Hom}_{L[H]}(V\otimes_K L, W\otimes_K L).$$
\end{lemm}

\begin{lemm}\label{irred}
Soit $\pi\in {\rm Rep}^{\rm adm}\,G$, irr\'eductible. Si $k={\rm End}_{k_L[G]}\pi$,
     alors:

{\rm (i)} $k$ est une extension finie de $k_L$.

{\rm (ii)} $\pi$ est absolument irr\'eductible vue comme $k[G]$-module.

{\rm (iii)} $k\otimes_{k_L}\pi=\oplus_{\sigma\in {\rm Hom}(k,k)}k\otimes_{k,\sigma}\pi$,
et les $k\otimes_{k,\sigma}\pi$ sont deux \`a deux non isomorphes comme $k[G]$-modules
mais sont toutes isomorphes comme $k_L[G]$-modules.
\end{lemm}
\begin{proof}
$k$ est une alg\`ebre \`a division, de dimension finie 
sur le corps fini $k_L$ (car $\pi$ est irr\'eductible et admissible, donc si l'on fixe $v\in \pi^{K_1}$ non nul, la fl\`eche 
		$f\mapsto f(v)$ induit une injection de $k$
		dans le $k_L$-espace de dimension finie $\pi^{K_1}$).
Il s'ensuit que $k$ est une extension finie de $k_L$, ce qui prouve le (i).

Si $k'$ est une extension de $k$ et si $\pi':=k'\otimes_k\pi$, alors
 d'apr\`es le lemme~\ref{51}, 
 $${\rm End}_{k'[G]}(\pi')= k'\otimes_k{\rm End}_{k[G]}(\pi)=k'$$
Si $\pi$ n'est pas absolument irr\'eductible vue comme $k[G]$-module,
on peut trouver une extension finie $k'$ de $k$ et un sous
$k'[G]$-module propre $\tau\subset \pi'$. 
Comme $\pi'$ est admissible, on peut supposer que $\tau$ est 
irr\'eductible\footnote{Sinon on fabrique une suite d\'ecroissante stricte de $G$-modules dans
        $\tau$, contredisant son admissibilit\'e.}. Le groupe $\Gamma={\rm Gal}(k'/k)$ agit sur
        $\pi'$ et $X:=\sum_{h\in \Gamma} h\cdot\tau$ est $G\times \Gamma$-stable dans $\pi'$. 
Par descente galoisienne $X^{\Gamma}$ est non nul et $G$-stable dans $\pi$,
donc $X^{\Gamma}=\pi$ (car $\pi$ est irr\'eductible
comme $k_L[G]$-module et donc, a fortiori, comme $k[G]$-module), 
ce qui exhibe $\pi'$ comme un quotient d'une somme directe 
finie $\oplus_{h\in \Gamma} h\cdot\tau$ de repr\'esentations irr\'eductibles. Donc
        $\pi'$ est semi-simple et comme ses endomorphismes sont $k'$, elle
est irr\'eductible contrairement \`a l'hypoth\`ese. Ceci prouve le (ii).

Enfin la d\'ecomposition 
$k\otimes_{k_L}k=\prod_{\sigma\in {\rm Hom}(k,k)} k$ fournit celle du (iii); le fait que les
$k\otimes_{k,\sigma}\pi$ sont deux \`a deux non isomorphes comme $k[G]$-modules
r\'esulte de ce que ${\rm End}_{k[G]}(k\otimes_{k_L}\pi)=k\otimes_{k_L}k$ (d'apr\`es
le lemme~\ref{51}) est un produit de corps; qu'elles soient toutes isomorphes comme $k_L[G]$-modules
est \'evident.
\end{proof}
\begin{rema}
L'irr\'eductibilit\'e sans l'admissibilit\'e
n'est pas suffisante pour impliquer les conclusions du lemme~\ref{irred}:
Le~\cite[th.\,1.2]{Le} a construit une ${\bf F}_{p^3}$-repr\'esentation 
supersinguli\`ere $\pi$ de $\gl_2(\Q_{p^3})$, telle que ${\rm End}_{G}\,\pi=\overline{\bf F}_{p}$,
ce qui implique que ${\bf F}_{p^{3n}}\otimes\pi$ a
$n$ composantes irr\'eductibles, pour tout $n$, et aucune n'est absolument irr\'eductible.
\end{rema} 

\subsubsection{Le cas $F=\Q_p$} \label{hecke2}
La classification des repr\'esentations irr\'eductibles
 est nettement mieux comprise dans le cas $F=\Q_p$:

$\bullet$  Breuil \cite{Breuil} a montr\'e que
les $I(\sigma)/(T)$ sont irr\'eductibles et donc les supersinguli\`eres sont
les $I(\sigma)/(T)$ (il a aussi d\'etermin\'e dans quels cas $I(\sigma_1)/(T)\simeq I(\sigma_2)/(T)$).

$\bullet$
Berger \cite{Berger} a montr\'e que l'existence d'un caract\`ere central 
est en fait une cons\'equence de l'irr\'eductibilit\'e.

$\bullet$ Il est facile de v\'erifier que les repr\'esentations
 ordinaires sont admissibles et de pr\'esentation finie (cela est vrai m\^eme pour $F\ne \Q_p$).
La preuve de Breuil permet de v\'erifier que les supersinguli\`eres le sont aussi
(on obtient des r\'esultats plus fins en combinant les th.\,IV.2.1 et~IV.4.7 de \cite{gl2}).
Il en r\'esulte
que tout $\pi\in{\rm Rep}^{\rm lisse}\,G$ irr\'eductible est admissible et de pr\'esentation finie, et donc (en utilisant la stabilit\'e 
de ces deux propri\'et\'es par extensions) que 
$${\rm Rep}^{\rm\ell f}\,G\subset {\rm Rep}^{\rm adm}\,G\cap {\rm Rep}^{\rm pf} \, G$$
(Voir le th.\,\ref{ord-ter} pour une inclusion dans l'autre sens.)

$\bullet$ Il r\'esulte de la classification de Barthel-Livn\'e et Breuil
des $\overline{\bf F}_p$-repr\'esentations irr\'eductibles de $G$ que, si $\pi$ est une telle repr\'esentation,
il existe un corps fini $k$ tel que $\pi$ soit obtenue par extension des scalaires d'une repr\'esentation
d\'efinie sur $k$. Il s'ensuit que l'orbite de $\pi$ sous l'action de ${\rm Gal}(\overline{\bf F}_p/k_L)$
est finie\footnote{
Pour d\'efinir cette action, on choisit une base $(e_i)_{i\in I}$ de $\pi$ sur $\overline{\bf F}_p$;
si $U_g\in\gl_I(\overline{\bf F}_p)$ est la matrice de $g\in G$ dans cette base, on a $U_{gh}=U_gU_h$
et donc, si $\sigma\in {\rm Gal}(\overline{\bf F}_p/k_L)$, on a $\sigma(U_{gh})=\sigma(U_g)\sigma(U_h)$;
il s'ensuit que $\sigma(U_g)$ est la matrice de $g$ sur une repr\'esentation $\pi^\sigma$ de $G$;
cette repr\'esentation ne d\'epend pas du choix de la base car changer de base change $U_g$ en $M^{-1}U_gM$
et $\sigma(U_g)$ en $\sigma(M)^{-1}\sigma(U_g)\sigma(M)$.}.  Si $H$ est le fixateur
de $\pi$ dans ${\rm Gal}(\overline{\bf F}_p/k_L)$, et si $k=\overline{\bf F}_p^H$, alors $\pi$ est l'extension
des scalaires d'une repr\'esentation $\pi_k$ d\'efinie sur $k$, telle que ${\rm End}_G\pi_k=k$.
Les $\pi_k^\sigma$, vues comme $k_L$-repr\'esentations de $G$, sont alors toutes isomorphes, 
et on note $\pi_L$ la classe d'isomorphisme de $K_L$-repr\'esentations ainsi d\'efinie.

L'application $\pi\mapsto\pi_L$ induit
une bijection naturelle
entre les orbites
sous ${\rm Gal}(\overline{\bf F}_p/k_L)$ de l'ensemble des $\overline{\bf F}_p$-repr\'esentations
irr\'eductibles de $G$ et les $k_L$-repr\'esentations irr\'eductibles de $G$, la bijection
r\'eciproque envoyant $\pi$ sur l'ensemble des composantes de Jordan-H\"older
de $\overline{\bf F}_p\otimes_{k_L}\pi$ (que celles-ci forment une orbite sous
${\rm Gal}(\overline{\bf F}_p/k_L)$ se d\'eduit du (iii) du lemme~\ref{irred}).

\subsection{Repr\'esentations de $\gl_2(\Q_p)$ et repr\'esentations de $\G_{\Q_p}$}\label{qc3}
On suppose $F=\Q_p$ dans ce paragraphe.
Si $\delta:\Q_p^\dual\to \O_L^\dual$ est un caract\`ere continu, on note ${\rm Rep}^\delta\,G$
la sous-cat\'egorie de ${\rm Rep}^{\rm\ell f}\,G$ des objets de caract\`ere central~$\delta$.
\subsubsection{La correspondance de Langlands locale $p$-adique.}
Rappelons que l'on dispose d'un foncteur covariant exact $\Pi\mapsto{\bf V}(\Pi)$
de ${\rm Rep}^{\rm \ell f}\,G$ (plus pr\'ecis\'ement de la sous-cat\'egorie des objets \`a caract\`ere central) dans la cat\'egorie des $\O_L$-repr\'esentations continues de $\G_{\Q_p}$,
de longueur finie sur $\O_L$. 

On \'etend ce foncteur par limite projective et tensorisation
par $L$ aux $L$-repr\'esentations de Banach unitaires de $G$, r\'esiduellement de longueur finie 
(i.e. dont la r\'eduction modulo~${\goth m}_L$ est de longueur finie; 
une telle repr\'esentation est automatiquement admissible et tout $L$-banach admissible, 
de longueur finie est r\'esiduellement de longueur finie -- un des r\'esultats principaux 
de \cite{CDP}) et avec un caract\`ere central. 
Si $\Pi$ est une telle $L$-repr\'esentation,
${\bf V}(\Pi)$ est une $L$-repr\'esentation de $\G_{\Q_p}$, de dimension finie.

R\'eciproquement, si $V$ est une $L$-repr\'esentation de $\G_{\Q_p}$, de dimension~$2$,
il existe \cite{gl2,CDP} une plus grande repr\'esentation $\bPi(V)$ de $G$, unitaire, r\'esiduellement de longueur finie  telle que
${\bf V}(\bPi(V))=V$ et $\bPi(V)^{\SL_2(\Q_p)}=0$ (cette derni\`ere condition
est n\'ecessaire pour assurer l'unicit\'e car ${\bf V}$ tue $\Pi^{\SL_2(\Q_p)}$).
La correspondance $V\mapsto\bPi(V)$ est la correspondance de Langlands locale $p$-adique.

La repr\'esentation ${\bf V}(\Pi)$ n'est pas toujours de dimension~$2$, mais est
toujours {\og de type $\gl_2$\fg} (cf.~th.\,\ref{pasku6} pour un \'enonc\'e pr\'ecis: l'expression
$g+\delta\epsilon(g)g^{-1}$ qui appara\^{\i}t dans
l'\'enonc\'e de ce th\'eor\`eme est la trace de $g$ si $g$ agit sur un module de rang $2$ et le d\'eterminant
de $g$ est $\delta\epsilon(g)$).

\subsubsection{Les repr\'esentations $I(\chi_1,\chi_2)$}\label{BL11}
On voit les caract\`eres continus de $\Q_p^\dual$
comme des caract\`eres de $\mathcal{G}_{\qp}$ par la th\'eorie locale du corps de classes
\footnote{Normalis\'ee de telle sorte que les uniformisantes correspondent aux frobenius g\'eom\'etriques.} et aussi
comme des caract\`eres de $G$ en composant avec le d\'eterminant. 
\vskip1mm
Soit $\epsilon$ la r\'eduction modulo $p$ du caract\`ere 
$x\to x|x|$ de $\qp^{\dual}$; vu comme caract\`ere de~$\sg_{\Q_p}$, c'est la r\'eduction
du caract\`ere cyclotomique.
On dit qu'un couple de caract\`eres $\chi_1, \chi_2: \qp^{\dual}\to \overline{\bf F}_p^{\dual}$ 
continus
(et donc lisses) est {\it g\'en\'erique} si $\chi_1\chi_2^{-1}\ne 1,\epsilon^{\pm 1}$.
\vskip1mm
Soit $(\chi_1,\chi_2)$ g\'en\'erique, et soient
$I(\chi_1, \chi_2)$ et $I(\chi_2,\chi_1)$ les
repr\'esentations absolument irr\'eductibles de $G$
   $$I(\chi_1, \chi_2):={\rm Ind}_B^G(\chi_2\otimes \chi_1\epsilon^{-1}),
\quad I(\chi_2,\chi_1)={\rm Ind}_B^G(\chi_1\otimes \chi_2\epsilon^{-1}).$$
Leur caract\`ere central est $\chi_1\chi_2\epsilon^{-1}$, et il existe une unique extension non triviale
$$0\to I(\chi_1,\chi_2)\to\Pi_{\chi_1,\chi_2}\to I(\chi_2,\chi_1)\to 0$$
Les repr\'esentations $I(\chi_1,\chi_2)$, 
$I(\chi_2,\chi_1)$ et $\Pi_{\chi_1,\chi_2}$ 
ont des mod\`eles sur $k_L(\chi_1,\chi_2)$ qui, comme dans le \no\ref{hecke2},
 peuvent \^etre consid\'er\'ees comme des
repr\'esentations sur $k_L$.
On a alors
$${\bf V}(I(\chi_1,\chi_2))=\chi_1,\quad {\bf V}(I(\chi_2,\chi_1))=\chi_2,
\quad
{\bf V}(\Pi_{\chi_1,\chi_2})=\bar{r}_{\chi_1,\chi_2}$$
o\`u $\chi_1$ et $\chi_2$ sont les $k_L(\chi_1,\chi_2)$-repr\'esentations
de $\G_{\Q_p}$ de dimension~$1$ associ\'ees \`a $\chi_1$ et $\chi_2$, vues comme
des $k_L$-repr\'esentations (de dimension $[k_L(\chi_1,\chi_2):k_L]$) et
$\bar{r}_{\chi_1,\chi_2}$ l'unique $k_L(\chi_1,\chi_2)$-repr\'esentation 
de $\mathcal{G}_{\qp}$ extension nontriviale de $\chi_2$ par $\chi_1$ (vue comme une
$k_L$-repr\'esentation)
   $$0\to \chi_1\to \bar{r}_{\chi_1,\chi_2}\to \chi_2\to 0$$

\begin{rema}\label{BL113}
Soit $\delta: Z\to \overline{\mathbf{F}}_p^{\dual}$ un caract\`ere lisse. 
La classification de Barthel-Livn\'e et Breuil montre 
qu'il n'y a qu'un nombre fini de $\pi\in{\rm Rep}^\delta\,G$, irr\'eductibles,
qui ne sont pas 
de la forme $I(\chi_1,\chi_2)$, avec $(\chi_1,\chi_2)$ g\'en\'erique.
\end{rema}

\subsubsection{G\'en\'eralit\'es sur les blocs}\label{qc4}
On met une relation d'\'equivalence sur les repr\'esentations irr\'eductibles
d\'efinie par
$\pi\sim\pi'$ si et seulement si il existe une suite $\pi=\pi_0,\pi_1,\dots,\pi_r=\pi'$ telle que
$\pi_{i+1}=\pi_i$ ou ${\rm Ext}^1(\pi_i,\pi_{i+1})\neq 0$ ou ${\rm Ext}^1(\pi_{i+1},\pi_i)\neq 0$
(ces conditions ne sont pas exclusives).
Une classe d'\'equivalence est {\it un bloc}. 
Si ${\cal B}$ est un bloc, on note ${\rm Rep}_{\cal B}\,G$ la sous-cat\'egorie de ${\rm Rep}^{\rm \ell f}\, G$
des objets dont toutes les composantes de Jordan-H\"older appartiennent \`a ${\cal B}$.

\begin{prop}\label{pasku0}
{\rm (\paskunas)}
Les blocs de ${\rm Rep}_{\overline{\bf F}_p}\,G$ sont les:

$\bullet$ $\{\pi\}$, o\`u $\pi$ est supersinguli\`ere.

$\bullet$ $\{I(\chi_1,\chi_2),I(\chi_2,\chi_1)\}$, o\`u
$\chi_1\chi_2^{-1}\neq \epsilon^{\pm}, 1$.

$\bullet$ $\{I(\chi,\chi)\}$, si $p\neq 2$.
{\rm(}Si $p=2$, ce bloc devient $\{\chi, {\rm St}\otimes\chi\}${\rm .)}

$\bullet$ $\{\chi, {\rm St}\otimes\chi, I(\chi,\chi\epsilon)\}$, si $p\neq 3$.
{\rm(}Si $p=3$, ce bloc devient
$\{\chi, {\rm St}\otimes\chi, \chi\epsilon, {\rm St}\otimes\chi\epsilon\}${\rm .)}
\end{prop}

\begin{rema}\label{pasku00}
On peut utiliser le lemme~\ref{51} et le \no\ref{hecke2} 
pour en d\'eduire une description
des blocs de ${\rm Rep}^{\rm\ell f}\,G$: ceux-ci sont en bijection naturelle avec les orbites
des blocs de ${\rm Rep}_{\overline{\bf F}_p}\,G$ sous l'action de
${\rm Gal}(\overline{\bf F}_p/k_L)$.
\end{rema}

Si $\pi\in {\rm Rep}^{\rm \ell f}\,G$ 
est irr\'eductible, on note $P_\pi$ le dual de Pontryagin d'une enveloppe injective de 
$\pi$ dans ${\rm Rep}^{\rm ladm} G$.
Alors $P_\pi$ est un $\O_L[G]$-module compact sans $p$-torsion, limite projective de $\O_L[G]$-modules compacts et
de longueur finie.
Si ${\cal B}$ est un bloc, on pose
$$P_{\cal B}:=\oplus_{\pi\in {\cal B}} P_\pi\quad
E_{\cal B}:={\rm End}_G(P_{\cal B}),\quad Z_{\cal B}:={\text{centre de $E_{\cal B}$}}.$$
La th\'eorie g\'en\'erale de Gabriel \cite{gaby} fournit le r\'esultat suivant:
\begin{prop}\label{pasku2}
{\rm (i)} Les foncteurs
$$\Pi\mapsto {\mathbbm m}_{\cal B}(\Pi):={\rm Hom}_G(P_{\cal B},\Pi^\vee)
\quad{\rm et}\quad
{\mathbbm m}\mapsto ({\mathbbm m}\otimes_{E_{\cal B}}P_{\cal B})^\vee$$
 sont inverses l'un de l'autre
et fournissent une \'equivalence de cat\'egories entre ${\rm Rep}_{\cal B}\,G$ et
la cat\'egorie des modules de longueur finie sur $E_{\cal B}$.

{\rm (ii)}
Si $\Pi\in {\rm Rep}^{\rm\ell f}\,G$, alors
$$\Pi\simeq \oplus_{\cal B}\big({\mathbbm m}_{\cal B}(\Pi)\otimes_{E_{\cal B}}P_{\cal B})^\vee.$$
\end{prop}

\begin{rema}\label{pasku1}
{\rm (i)} On a les m\^emes r\'esultats pour ${\rm Rep}^\delta\,G$; les blocs sont les m\^emes que
pour ${\rm Rep}^{\rm\ell f}\,G$ en rajoutant la condition que le caract\`ere central soit $\delta$
(ce qui se traduit par $\chi_1\chi_2\epsilon^{-1}=\delta$ pour $I(\chi_1,\chi_2)$, et par
$\chi^2=\delta$ pour le bloc contenant~$\chi$).
Si ${\cal B}$ est un bloc, on d\'efinit
$P_{\cal B}^\delta$, $E_{\cal B}^\delta$, $Z_{\cal B}^\delta$, comme ci-dessus, mais
en prenant les enveloppes projectives dans la cat\'egorie des repr\'esentations
de caract\`ere central $\delta^{-1}$ (dualiser change le caract\`ere central en son inverse).

{\rm (ii)}
Si $\pi\in {\rm Rep}^\delta_{\cal B}\, G$, alors $Z_{\cal B}^\delta$ 
agit sur $\pi$ ainsi que, par fonctorialit\'e,
sur ${\bf V}(\pi)$.
\end{rema}

\begin{rema}\label{ratio12}
Si ${\cal B}$ est un bloc de ${\rm Rep}^{\rm\ell f}\,G$ et si $\pi\in{\cal B}$, alors ${\rm End}_{k_L[G]}\pi$ ne d\'epend pas de $\pi$; notons le $k({\cal B})$ et
notons $L({\cal B})$ l'extension non ramifi\'ee de $L$ de corps r\'esiduel $k({\cal B})$.
Les composantes irr\'eductibles des $k({\cal B})\otimes_k\pi$, pour $\pi\in{\cal B}$, 
se r\'epartissent dans des types ${\cal B}^\sigma$ de 
${\rm Rep}^{\rm\ell f}_{\O_{L({\cal B})}}\,G$ formant une orbite sous l'action
de ${\rm Gal}(k({\cal B})/k_L)$ et les ${\cal B}^\sigma$ sont form\'es
de repr\'esentations absolument irr\'eductibles. 
L'oubli de l'action de $\O_{L({\cal B})}$ induit une
\'equivalence entre ${\rm Rep}_{\O_{L({\cal B})},{\cal B}^\sigma}\,G$
et ${\rm Rep}_{\cal B}\,G$. On en d\'eduit, pour tout $\sigma$, des isomorphismes:
$$E_{\cal B}\simeq E_{\O_{L({\cal B})},{\cal B}^\sigma},\quad
Z_{\cal B}\simeq Z_{\O_{L({\cal B})},{\cal B}^\sigma}$$
et de m\^eme pour $E_{\cal B}^\delta,Z_{\cal B}^\delta$.
\end{rema}

\Subsubsection{Blocs et repr\'esentations galoisiennes modulo~$p$}\label{qc10}
\begin{rema}\label{pasku3}
(i)
Il y a \cite{Pas1} une bijection ${\cal B}\leftrightarrow\rho_{\cal B}$ entre $\overline{\bf F}_p$-blocs et 
$\overline{\bf F}_p$-repr\'esentations de $G_{\Q_p}$, semi-simples, de dimension $2$.
Cette bijection envoie un bloc ${\cal B}$ sur la repr\'esentation semi-simple dont
l'ensemble des sous-repr\'esentations irr\'eductibles est $\{{\bf V}(\pi),\ \pi\in{\cal B}\}$.
Donc:

$\bullet$ $\rho_{\cal B}$ est irr\'eductible, de dimension~$2$, si ${\cal B}=\{\pi\}$, avec $\pi$ supersinguli\`ere.

$\bullet$ $\rho_{\cal B}=\chi_1\oplus\chi_2$, 
si ${\cal B}=\{I(\chi_1,\chi_2),I(\chi_2,\chi_1)\}$ et $\chi_1 \chi_2^{-1}\neq 1, \epsilon^{\pm 1}$.

$\bullet$ $\rho_{\cal B}=\chi\oplus\chi$, si ${\cal B}=\{I(\chi,\chi)\}$. 

$\bullet$ $\rho_{\cal B}=\chi\oplus\chi\epsilon$, 
si ${\cal B}=\{\chi,{\rm St}\otimes\chi,I(\chi,\chi\epsilon)\}$.

\noindent Dans tous les cas, $(\det\rho_{\cal B})\epsilon^{-1}$ est \'egal au caract\`ere central des
\'el\'ements de ${\cal B}$.

(ii) On d\'eduit de ce r\'esultat une bijection naturelle entre les blocs de
${\rm Rep}^{\rm\ell f}\,G$ (resp.~${\rm Rep}^\delta\,G$)
et les orbites sous l'action de ${\rm Gal}(\overline{\bf F}_p/k_L)$
des $\overline{\bf F}_p$-repr\'esentations de~$\G_{\Q_p}$,
semi-simples, de dimension $2$ (resp.~et dont le d\'eterminant est $\delta\epsilon$).

Plus pr\'ecis\'ement, si ${\cal B}$ est un bloc de ${\rm Rep}^{\rm\ell f}\,G$ et si $\pi\in{\cal B}$, alors ${\rm End}_{k_L[G]}\pi$ ne d\'epend pas de $\pi$; notons le $k({\cal B})$.
Alors ${\rm Tr}\rho_{\cal B}$ est \`a valeurs dans $k({\cal B})$ et donc
$\rho_{\cal B}$ est d\'efinie sur $k({\cal B})$. Les $\rho_{\cal B}^\sigma$, pour
$\sigma\in{\rm Gal}(k({\cal B})/k_L)$, sont deux \`a deux non isomorphes comme
$k({\cal B})$-repr\'esentations de $\G_{\Q_p}$ mais sont toutes isomorphes comme
$k_L$-repr\'esentations (de dimension~$[k({\cal B}):k_L]$).
\end{rema}

Soit ${\cal B}$ un bloc de ${\rm Rep}^\delta\,G$. Notons
$R_{\cal B}^{{\rm ps},\delta}$ l'anneau des d\'eformations universelles
du pseudo-caract\`ere ${\rm Tr}\circ\rho_{\cal B}$ (vu comme pseudo-caract\`ere de dimension~$2$ 
\`a valeurs dans~$k({\cal B})$, cf.~(ii) de la rem.\,\ref{pasku3}), 
de d\'eterminant $\delta\epsilon$, et 
$$T_{\cal B}^\delta:\G_{\Q_p}\to R_{\cal B}^{{\rm ps},\delta}$$
 le pseudo-caract\`ere universel.
On a alors le {\og th\'eor\`eme $R=T$\fg} local suivant\footnote{\paskunas-Tung demandent
que ${\cal B}$ soit constitu\'e de repr\'esentations absolument irr\'eductibles, le cas g\'en\'eral
s'en d\'eduit en utilisant les rem.\,\ref{ratio12} et~\ref{pasku3}.} 
(cf.~(ii) de la rem.\,\ref{pasku1} pour l'action
de $Z_{\cal B}^\delta$).
\begin{theo}\label{pasku6}
{\rm (\paskunas~\cite{Paskext}, {\paskunas}-Tung~\cite{PT})}
Il existe un unique morphisme
$$\iota_{\cal B}^\delta:R_{\cal B}^{{\rm ps},\delta}\to Z_{\cal B}^\delta$$
 tel que, pour tout $\pi\in {\rm Rep}^\delta_{\cal B} G$ et tout $g\in{\rm Gal}_{\Q_p}$,
on ait
$$\iota_{\cal B}^\delta(T_{\cal B}^\delta(g))=g+\delta\epsilon(g)g^{-1}
\quad{\text{dans ${\rm End}({\bf V}(\pi))$}}.$$
De plus:

$\bullet$ Si $p\geq 5$, $\iota_{\cal B}^\delta$ est un isomorphisme.

$\bullet$ Dans le cas g\'en\'eral $L\otimes_{\O_L}\iota_{\cal B}^\delta$ est un isomorphisme,
$Z_{\cal B}^\delta=R_{\cal B}^{{\rm ps},\delta}/\O_L{\text{-{\rm torsion}}}$ si $p=3$ et le conoyau
de $R_{\cal B}^{{\rm ps},\delta}/\O_L{\text{-{\rm torsion}}}\to Z_{\cal B}^\delta$ est tu\'e par $2$ si $p=2$.
\end{theo}

\subsection{Crit\`eres de finitude}
On suppose $F=\Q_p$ dans ce paragraphe.
Le th.\,\ref{pierre1} ci-dessous
joue un r\^ole important dans la preuve du cor.\,\ref{Intro3}.

\subsubsection{Quotients de $I(\sigma)$}
   Si $\sigma$ est un poids de Serre,
la repr\'esentation $I(\sigma)$ n'est pas de longueur finie
    mais on dispose du r\'esultat de rigidit\'e suivant, sp\'ecifique \`a notre groupe $G$ et d\'ecoulant des r\'esultats de Barthel-Livn\'e et Breuil.
   
   \begin{prop}\label{BL}
    Soit $\sigma$ un poids de Serre et soit $\pi$ un quotient nontrivial de $I(\sigma)$. Alors 
   $\pi$ est de longueur finie.
   \end{prop}
   
   \begin{proof} Ce r\'esultat est "standard", mais 
   nous allons en donner la preuve pour le confort du lecteur. Soit $\pi'=\ker(I(\sigma)\to \pi)$ et soit $I_1$ le pro-$p$ Iwahori de~$G$. La prop.\,18 de \cite{BL} montre que 
   $(\pi')^{I_1}$ est de co-dimension finie dans $I(\sigma)^{I_1}$. 
De plus ${\rm End}_{k_L[G]}(I(\sigma))\simeq k_L[T]$ (cf.~relation~(\ref{hecke})). 
Soit $f\in I(\sigma)$ une fonction de support $KZ$, telle que $f(1)$ engendre la droite $\sigma^{I_1}$. Puisque $\dim I(\sigma)^{I_1}/(\pi')^{I_1}<\infty$, il existe un polyn\^ome non nul $P$ tel que $P(T)f\in \pi'$. Mais alors la surjection $I(\sigma)\to \pi$ se factorise par 
   $I(\sigma)/(P(T))$, et cette derni\`ere repr\'esentation est de longueur finie d'apr\`es les r\'esultats de Barthel-Livn\'e et Breuil. 
Voir aussi le cor.\,2.1.4 de \cite{DEG} pour un autre argument, utilisant la classification 
(due \`a Barthel-Livn\'e) des morphismes $G$-\'equivariants $I(\sigma)\to I(\sigma')$, 
pour $\sigma, \sigma'$ poids de Serre. 
   \end{proof}
   
   \begin{rema}\label{em-vi}
Emerton conjecture \cite[conj. 2.3.7]{Ordparts1}
 que la proposition admet la version suivante pour $\gl_2(F)$ (avec $F/\qp$ une extension finie arbitraire): un quotient nontrivial {\it admissible} de $I(\sigma)$
est de longueur finie. 
Plus g\'en\'eralement, il conjecture que, pour un groupe r\'eductif $p$-adique quelconque, une repr\'esentation admissible et de type fini est de longueur finie.
   \end{rema}

\subsubsection{Finitude et finitude du cosocle}
Soit $G':=G/p^{\mathbf{Z}}$, $p$ \'etant vu comme l'\'el\'ement $\matrice{p}{0}{0}{p}$ de $Z$. 
On identifie $K$ \`a un sous-groupe de 
      $G'$. Si $\sigma$ est un poids de Serre, on pose
 $$I_{G'}(\sigma):={\rm c\text{-}Ind}_{K}^{G'}(\sigma).$$ 
Si l'on munit 
$\sigma$ de sa structure de $KZ$-module obtenue en faisant agir $p$ trivialement, alors $I_{G'}(\sigma)$ s'identifie \`a 
$I(\sigma)$.

\begin{prop}\label{filtration}
Les $k_L[G']$-modules lisses de type fini 
sont pr\'ecis\'ement       
les $k_L[G']$-modules lisses $\pi$ tels qu'il existe une filtration 
$$0=\pi_0\subset...\subset \pi_r=\pi$$
et des surjections $G'$-\'equivariantes $I(\sigma_i)\to \pi_i/\pi_{i-1}$ pour $1\leq i\leq r$ et pour certains poids de Serre $\sigma_i$.
\end{prop}
      
\begin{proof}
 Une implication d\'ecoule du fait que $I(\sigma_i)$ (et donc tout quotient de $I(\sigma_i)$) est de type fini et 
 la stabilit\'e des 
 repr\'esentations de type fini 
par extensions.
       
R\'eciproquement,
supposons que $\pi$ est une repr\'esentation lisse de type fini de $G'$.
 Soit $v_1,\dots,v_n$ une famille de vecteurs qui engendre $\pi$ en tant que $k_L[G']$-module.
 Soit $W$ le $k_L[K]$ sous-module de $\pi$ engendr\'e par $v_1,\dots,v_n$. Alors 
$\dim_{k_L}(W)<\infty$ et l'on dispose d'une surjection naturelle ${\rm c\text{-}Ind}_{K}^{G'}(W)\to \pi$.
 Prenons une filtration 
$0=W_0\subset\dots\subset W_r=W$ avec $\sigma_i:=W_i/W_{i-1}$ irr\'eductible, et notons 
$\pi_i$ l'image de ${\rm c\text{-}Ind}_{K}^{G'}(W_i)$ dans $\pi$ par la surjection ci-dessus. Alors 
$\pi_i/\pi_{i-1}$ est un quotient de 
${\rm c\text{-}Ind}_{K}^{G'}(W_i)/{\rm c\text{-}Ind}_{K}^{G'}(W_{i-1})\simeq {\rm c\text{-}Ind}_{K}^{G'} (\sigma_i)$, ce que l'on voulait.
\end{proof}

 Le r\'esultat ci-dessus reste valable, avec la m\^eme preuve, pour $F\ne \Q_p$, mais la proposition ci-dessous est tr\`es sp\'ecifique au cas $F=\Q_p$.
      
     \begin{prop}
     Tout $k_L[G]$-module lisse de type fini de $G'$ est de pr\'esentation finie.
     \end{prop}
     
     \begin{proof}
     Soit $0=\pi_n\subset \pi_{n-1}\subset...\subset \pi_0=\pi$ une filtration de $\pi$ telle qu'il existe des poids de Serre $\sigma_i$ et des surjections $I_{G'}(\sigma_i)\to \pi_{i-1}/\pi_i$. Comme les repr\'esentations de pr\'esentation finie sont stables par extensions, il suffit de voir que chaque $\pi_{i-1}/\pi_i$ est de pr\'esentation finie. Or $\pi_{i-1}/\pi_i$ est soit de longueur finie 
(prop.\,\ref{BL}), auquel cas on peut conclure par le \no\ref{hecke2}, ou 
     isomorphe \`a $I_{G'}(\sigma_i)$, qui est clairement de pr\'esentation finie.
     \end{proof}

\begin{theo}\label{pierre1}
Soit $\pi$ un $k_L[G']$-module lisse de type fini, dont le 
cosocle est admissible {\rm (ou, de mani\`ere \'equivalente, de longueur finie)}. 
Alors $\pi$ est de longueur finie.
\end{theo}

\begin{proof}
Soit $0=\pi_n\subset \pi_{n-1}\subset...\subset \pi_0=\pi$ une filtration de $\pi$ telle qu'il existe des poids de Serre $\sigma_i$ et des surjections $I_{G'}(\sigma_i)\to \pi_{i-1}/\pi_i$. Nous allons montrer que $\pi_{i-1}/\pi_i$ est de longueur finie pour tout $i$, ce qui permettra de conclure. 
   
Puisque le cosocle de $I_{G'}(\sigma_1)$ est de longueur infinie (cons\'equence de \cite{BL}) et celui de $\pi$ est de longueur finie, la surjection 
$I_{G'}(\sigma_1)\to \pi_0/\pi_1$ n'est pas un isomorphisme, et la prop.\,\ref{BL} montre que $\pi_0/\pi_1$ est de longueur finie. Il suffit de montrer que le cosocle de $\pi_1$ est de longueur finie, car on peut alors conclure par r\'ecurrence sur la longueur de la filtration. 
   
   On veut donc montrer que ${\rm Hom}_{G'}(\pi_1, \pi')={\rm Hom}_{G}(\pi_1, \pi')$ est de dimension finie sur $k_L$ pour tout $k_L[G']$-module lisse irr\'eductible $\pi'$, et nul pour presque tout $\pi'$ (i.e. tous sauf un nombre fini). 
   La suite exacte 
   $$0\to \pi_1\to \pi_0\to \pi_0/\pi_1\to 0$$
   en induit une autre 
   $$0\to {\rm Hom}_G(\pi_0/\pi_1, \pi')\to {\rm Hom}_G(\pi_0, \pi')\to {\rm Hom}_G(\pi_1, \pi')\to 
   {\rm Ext}^1_G(\pi_0/\pi_1, \pi').$$
   Puisque $\pi_0/\pi_1$ est de longueur finie, les lemmes~\ref{Pas} et~\ref{cc} 
ci-dessous  montrent que les termes extr\^emes de la suite exacte ci-dessus sont de dimension finie, et nuls pour presque tout $\pi'$. Comme il en est de m\^eme du terme ${\rm Hom}_G(\pi_0, \pi')$ 
(puisque le cosocle de $\pi_0=\pi$ est de longueur finie), cela permet de conclure.
  \end{proof}

  Le r\'esultat suivant est une adaptation directe d'un r\'esultat analogue dans \cite{Berger}:

\begin{lemm}\label{cc}
Tout $k_L[G']$-module lisse irr\'eductible $\pi$ poss\`ede un caract\`ere central.
\end{lemm}
\begin{proof}
Comme $\pi$ est lisse et irr\'eductible, le pro-$p$ sous-groupe central $1+p\zp$ de $G'$ agit trivialement sur 
$\pi$, donc l'action du centre de $G'$ se factorise par $\mathbf{F}_p^\dual$.
Si $x$ est un g\'en\'erateur de ce dernier groupe, alors $0=x^{p-1}-1=\prod_{i=1}^{p-1}(x-i)$ 
en tant qu'endomorphismes de $\pi$, donc par irr\'eductibilit\'e il existe $i$ tel que $x=i$ sur $\pi$.
Cela permet de conclure.
\end{proof}

  \begin{lemm}\label{Pas}
    Soit $\pi$ un $k_L[G]$-module lisse, de longueur finie, avec un caract\`ere central.
 Alors ${\rm Ext}^1_G(\pi, \pi')$ est de dimension finie sur $k_L$ pour tout $k_L[G]$-module 
lisse irr\'eductible $\pi'$, avec un caract\`ere central, et nul pour presque tout $\pi'$. 
  \end{lemm}
  
  \begin{proof} Soit $\delta$ (resp. $\delta'$) le caract\`ere central de $\pi$ (resp. $\pi'$). On peut supposer que $\pi$ est irr\'eductible. Soit $\overline{\pi}=\overline{\mathbf{F}}_p\otimes_{k_L} \pi$ et $\overline{\pi'}=\overline{\mathbf{F}}_p\otimes_{k_L} \pi'$.  \'Ecrivons $\overline{\pi}=\oplus_{i=1}^n \pi_i$ et $\overline{\pi'}=\oplus_{j=1}^m \pi'_j$, les 
  repr\'esentations $\pi_i$ et $\pi'_j$ \'etant irr\'eductibles 
(cf.~(ii) et (iii) du lemme\,\ref{irred}). 
  
  Le lemme~\ref{51} montre que la fl\`eche $${\rm Ext}^1_G(\pi,\pi')\otimes_{k_L} \overline{\mathbf{F}}_p\to 
  {\rm Ext}^1_G(\overline{\pi}, \overline{\pi'})\simeq \bigoplus_{1\leq i\leq n, 1\leq j\leq m} {\rm Ext}^1_G(\pi_i, \pi'_j)$$ est injective, il suffit donc de s'assurer que ce dernier espace est de dimension finie, et nul pour presque tout $\pi'$. 
  
       La prop.\,8.1 de \cite{Pas0} montre que ${\rm Ext}^1_G(\pi_i, \pi'_j)=0$ 
pour tout $i,j$ si $\delta\neq\delta'$, et si 
  $\delta=\delta'$, on a une suite exacte 
  $$0\to {\rm Ext}^1_{G,\delta}(\pi_i, \pi'_j)\to {\rm Ext}^1_G(\pi_i, \pi'_j)\to X_{ij}\to 0,$$
  avec $X_{ij}=0$ si $\pi_i$ n'est pas isomorphe \`a $\pi'_j$ 
et $X_{ij}={\rm Hom}(Z, \overline{\mathbf{F}}_p)$ dans le cas contraire.
Notons que si $\pi_i\simeq \pi'_j$,
 alors ${\rm Hom}_G(\overline{\pi}, \overline{\pi'})\ne 0$ et le lemme~\ref{51} montre que 
  ${\rm Hom}_G(\pi,\pi')\ne 0$, donc $\pi'\simeq \pi$. Comme ${\rm Hom}(Z, \overline{\mathbf{F}}_p)$ est de dimension finie, il suffit donc de montrer que 
les ${\rm Ext}^1_{G,\delta}(\pi_i, \pi'_j)$ sont de dimension finie, et nuls pour presque tout~$\pi'$. 
  Cela est classique (cf. \cite{Pas0,Pas1}). 
      \end{proof}

\subsubsection{Finitude des vecteurs lisses de $\pi^\vee$}      
      Pour les applications \`a la cohomologie de la tour de Lubin-Tate nous aurons aussi besoin du r\'esultat suivant: 
      
 \begin{theo}\label{LTappli}
 Soit $\pi$ un $k_L[G']$-module lisse de type fini. Alors 
 $\pi$ est de longueur finie si et seulement si l'espace des vecteurs lisses 
 de $\pi^{\vee}$ est de dimension finie sur $k_L$.
 \end{theo} 
 
 \begin{proof} Supposons que $\pi$ est de longueur finie et montrons que 
 $\dim (\pi^{\vee})^{\rm lisse}<\infty$. Par d\'evissage on peut supposer que 
 $\pi$ est irr\'eductible, auquel cas le r\'esultat est bien connu (on a m\^eme 
 $(\pi^{\vee})^{\rm lisse}=0$ sauf si $\pi$ est de dimension finie, 
cf.~\cite[prop.\,3.9]{Kohldual}). 
 
  L'autre implication est plus d\'elicate. Nous allons commencer par montrer:
  
   \begin{lemm}
  Si $\pi$ est un $k_L[G']$-module lisse de type fini et de longueur infinie, alors il existe un poids de Serre
  $\sigma$ tel que ${\rm Hom}_{G}(\pi, I_{G'}(\sigma))\ne 0$.
  \end{lemm}
  
  \begin{proof}
  Soit $0=\pi_n\subset \pi_{n-1}\subset...\subset \pi_0=\pi$ une filtration de $\pi$ telle qu'il existe des poids de Serre $\sigma_i$ et des surjections $I_{G'}(\sigma_i)\to \pi_{i-1}/\pi_i$. Si 
 la surjection $I_{G'}(\sigma_1)\to \pi/\pi_1$ est un isomorphisme, on peut conclure, supposons donc que ce n'est pas le cas, donc (prop.\,\ref{BL}) $\pi/\pi_1$ est de longueur finie. Alors 
 $\pi_1$ est de type fini et de longueur infinie, donc par r\'ecurrence sur la longueur de la filtration on dispose d'un poids de Serre $\sigma$ tel que ${\rm Hom}_G(\pi_1, I_{G'}(\sigma))\ne 0$. En appliquant le foncteur 
 ${\rm Hom}_{G}(-, I_{G'}(\sigma))$ \`a la suite exacte $0\to \pi_1\to \pi\to \pi/\pi_1\to 0$ on obtient une suite exacte 
 $${\rm Hom}_G(\pi, I_{G'}(\sigma))\to {\rm Hom}_G(\pi_1, I_{G'}(\sigma))\to {\rm Ext}^1(\pi/\pi_1, I_{G'}(\sigma)).$$
Comme $\pi/\pi_1$ est de longueur finie, un r\'esultat de Dotto-Emerton-Gee \cite{DEG} montre que
${\rm Ext}^1(\pi/\pi_1, I_{G'}(\sigma))$ est de dimension finie.\footnote{Dans \cite{DEG} on travaille avec un caract\`ere central fix\'e, mais on peut s'y ramener 
comme suit (nous remercions le rapporteur pour cette observation et pour l'argument ci-dessous): par d\'evissage il suffit de voir que $\dim {\rm Ext}^1(\tau, I_{G'}(\sigma))<\infty$ pour 
$\tau$ irr\'eductible, donc \`a caract\`ere central (lemme \ref{cc}), disons $\zeta$. On peut supposer que 
$\zeta$ est aussi le caract\`ere central de $I_{G'}(\sigma)$, sinon ${\rm Ext}^1(\tau, I_{G'}(\sigma))=0$. Il suffit alors d'observer que toute extension $0\to I_{G'}(\sigma)\to E\to \tau\to 0$ a aussi $\zeta$ pour caract\`ere central. En effet, si $z\in Z({\rm GL}_2(\qp))$ alors $z-\zeta(z)\in {\rm End}(E)$ se factorise par un morphisme 
$\tau\to I_{G'}(\sigma)$, qui est forc\'ement nul par rigidit\'e (prop.\,\ref{BL}).}
 Supposons que ${\rm Hom}_G(\pi, I_{G'}(\sigma))$, donc ${\rm Hom}_G(\pi_1, I_{G'}(\sigma))$ est de dimension finie sur $k_L$. D'autre part ${\rm Hom}_G(\pi_1, I_{G'}(\sigma))$ est un 
 $k_L[T]={\rm End}_{G}(I_{G'}(\sigma))$-module, donc il existe $P\in k_L[T]$ non nul tel qui tue 
 ${\rm Hom}_G(\pi_1, I_{G'}(\sigma))$ mais, comme $I_{G'}(\sigma)$ est libre sur $k_L[T]$, 
la multiplication par $P(T)$ est injective sur 
  ${\rm Hom}_G(\pi_1, I_{G'}(\sigma))$, une contradiction. 
  \end{proof}

Revenons \`a la preuve du th.\,\ref{LTappli}.
  Supposons que $(\pi^{\vee})^{\rm lisse}$ est de dimension finie sur $k_L$ et montrons que 
  $\pi$ est de longueur finie. Si ce n'est pas le cas, le lemme fournit un poids de Serre 
  $\sigma$ et un morphisme $G$-\'equivariant continu non nul $f: I_{G'}(\sigma)^{\vee}\to \pi^{\vee}$. Comme 
  $(\pi^{\vee})^{\rm lisse}$ est de dimension finie sur $k_L$, il est ferm\'e dans $\pi^{\vee}$, donc 
  $f^{-1}((\pi^{\vee})^{\rm lisse})$ est ferm\'e dans $I_{G'}(\sigma)^{\vee}$. D'autre part, 
   $f^{-1}((\pi^{\vee})^{\rm lisse})$ contient $(I_{G'}(\sigma)^{\vee})^{\rm lisse}$, qui est dense dans 
   $I_{G'}(\sigma)^{\vee}$, donc $f^{-1}((\pi^{\vee})^{\rm lisse})= I_{G'}(\sigma)^{\vee}$; ainsi  
   l'image de $f$ est incluse dans $(\pi^{\vee})^{\rm lisse}$
et donc est de dimension finie, ce qui fournit  
une sous-repr\'esentation $({\rm Im}(f))^{\vee}$
de dimension finie de $I_{G'}(\sigma)$, non nulle puisque $f\neq 0$. Cela est impossible. 
  \end{proof}

\subsection{Repr\'esentations admissibles de pr\'esentation finie de ${\rm GL}_2(F)$}
Soit 
$$\mathcal{C}={\rm Rep}^{\rm adm}\,G\cap {\rm Rep}^{\rm pf}\,G$$
la sous-cat\'egorie de ${\rm Rep}^{\rm lisse}\,G$ 
des repr\'esentations admissibles et de pr\'esentation finie. 
Cette cat\'egorie a \'et\'e \'etudi\'ee par Hu~\cite{YHu} et Vign\'eras~\cite{vigne}.
\begin{theo}\label{ord-ter}
{\rm (i)} $\mathcal{C}$ est une sous-cat\'egorie de ${\rm Rep}^{\rm \ell f}\,G$.

{\rm (ii)} ${\cal C}$ est stable par sous-quotients et extensions.
\end{theo}
\begin{proof}
 Le point (i) d\'ecoule directement de la proposition $5.6\, (i)$ de \cite{vigne}. La stabilit\'e de 
 ${\cal C}$ par extensions d\'ecoule de la proposition \ref{Shot25} (une extension lisse de repr\'esentations lisses admissibles est trivialement admissible). Soit $V\in {\cal C}$ et soit $V_1$ un sous-objet de $V$. Par la proposition $5.5\, (i)$ de \cite{vigne} (noter que l'hypoth\`ese est satisfaite gr\^ace \`a la proposition $5.6$ de loc.cit.)
 on a $V_1\in \mathcal{C}$. Si $V_2$ est un quotient de $V_1$, alors $V_2$ est admissible (tout quotient d'une repr\'esentation lisse admissible est encore lisse admissible) et $V_2$ est de pr\'esentation finie par la proposition \ref{Shot25} (noter que le noyau de la surjection $V_1\to V_2$ est dans $\mathcal{C}$, en particulier de type fini, d'apr\`es ce que l'on vient de faire). 
\end{proof}

Si $F=\Q_p$, 
les repr\'esentations irr\'eductibles de $G$ sont admissibles et de pr\'esentation finie,
 donc 
$${\cal C}={\rm Rep}^{\rm \ell f}\,G,\quad{\text{si $F=\Q_p$}}.$$ 
Ce n'est plus du tout le cas pour 
$F\ne \Q_p$, comme le montre le r\'esultat suivant:

\begin{theo}\label{SchWu} {\rm (Schraen~\cite{schraen}, Wu~\cite{Wu})} \label{em-vi2}
Si $F\neq\Q_p$,
les repr\'esentations supersinguli\`eres de $\gl_2(F)$ sur $\overline{\bf F}_p$ 
ne sont pas de pr\'esentation finie.
\end{theo}

\begin{rema}\label{em-vi22}
Le~\cite{Le} a construit des $k_L$-repr\'esentations de $\gl_2(\Q_{p^3})$
qui sont lisses,
irr\'eductibles, mais pas admissibles.
\end{rema}

Nous allons d\'ecrire ${\cal C}$ dans le cas $F\neq\Q_p$ via le foncteur {\og partie ordinaire\fg}
d'Emerton.  
\begin{center}
{\it Nous supposons $F\neq\Q_p$ dans le reste de ce paragraphe.} 
\end{center}

\subsubsection{Composantes de Jordan-H\"older}
Si $k$ est une extension finie de $k_L$, on note 
$\mathcal{C}_k^{\rm ord}$ la cat\'egorie des $k[G]$-modules lisses, de longueur finie, dont tous les sous-quotients irr\'eductibles sont de la forme $\chi\circ \det, {\rm St}\otimes \chi\circ \det$ ou ${\rm Ind}_B^G (\chi_1\otimes \chi_2)$, avec $\chi, \chi_1,\chi_2: F^{\dual}\to k^{\dual}$ des caract\`eres lisses.
Ainsi $\mathcal{C}_{k_L}^{\rm ord}$ est une sous-cat\'egorie de $\mathcal{C}$.

\begin{lemm}\label{ord}
Pour tout $\pi\in \mathcal{C}$ tu\'e par ${\goth m}_L$ il existe une extension finie 
$k$ de $k_L$ telle que $\pi\otimes_{k_L} k\in \mathcal{C}_k^{\rm ord}$.
\end{lemm}
\begin{proof} 
On peut supposer que $\pi$ est irr\'eductible. Prenons une extension $k$ de $k_L$ et des 
$\pi_i$ comme dans le (iii) du lemme \ref{irred}.
Notons que chaque $\pi_i$ est aussi de pr\'esentation finie, donc $\pi_i\otimes_k \overline{\mathbf{F}}_p$ aussi. De plus, par construction
les endomorphismes de $\pi_i\otimes_k \overline{\mathbf{F}}_p$
sont scalaires, donc cette repr\`esentation poss\`ede un caract\`ere central. 
Le th.\,\ref{SchWu} montre alors que $\pi_i\otimes_k \overline{\mathbf{F}}_p$ n'est pas supersinguli\`ere,
 et la classification de Barthel-Livn\'e \cite{BL0} permet de conclure que $\pi_i\otimes_k \overline{\mathbf{F}}_p$ est un caract\`ere lisse,
 une tordue de la Steinberg ou une s\'erie principale.
 Dans tous les cas on trouve une extension finie 
$k'$ de $k$ et une repr\'esentation $\sigma$ d\'efinie sur $k'$, qui est un caract\`ere, une tordue de la Steinberg ou une s\'erie principale et telle que $\pi_i\otimes_k \overline{\mathbf{F}}_p\simeq \sigma\otimes_{k'} \overline{\mathbf{F}}_p$.
 Le lemme \ref{51} montre alors que $\pi_i\otimes_{k} k'$ est isomorphe \`a $\sigma$, ce qui permet de conclure.   
\end{proof}

\subsubsection{Le foncteur {\og partie ordinaire\fg}}
Soit $\bar{B}$ le Borel inf\'erieur de $G$ et soit $A$ le tore diagonal de $G$.
Le foncteur $U\mapsto {\rm Ind}_{\bar{B}}^G(U)$ de
${\rm Rep}(A)$ dans ${\rm Rep}(G)$ est
exact et pr\'eserve l'admissibilit\'e \cite[prop.\,4.1.5, 4.1.7]{Ordparts1}, donc induit un foncteur 
$${\rm Ind}_{\bar{B}}^G: {\rm Rep}(A)^{\rm adm}\to {\rm Rep}(G)^{\rm adm}$$
Emerton a montr\'e \cite[th. 4.4.6]{Ordparts1} que ce dernier foncteur poss\`ede un adjoint \`a droite, le {\it{foncteur partie ordinaire}}
$${\rm Ord}={\rm Ord}_B: {\rm Rep}(G)^{\rm adm}\to {\rm Rep}(A)^{\rm adm}$$
Ainsi le passage \`a la partie ordinaire induit une 
identification 
$${\rm Hom}_G({\rm Ind}_{\bar{B}}^G(U), V)\simeq {\rm Hom}_{A}(U, {\rm Ord}(V))$$
pour des repr\'esentations lisses admissibles $U$ et 
$V$ de $A$, respectivement $G$. En particulier l'identit\'e de 
${\rm Ord}(V)$ induit un morphisme canonique 
$$\iota_V: {\rm Ind}_{\bar{B}}^G( {\rm Ord}(V))\to V,$$
tel que ${\rm Ord}(\ker \iota_V)=0$ et ${\rm Ord}(V^{\rm ord})={\rm Ord}(V)$ si l'on note 
$V^{\rm ord}$ l'image de $\iota_V$. On se propose de montrer (prop.\,\ref{finito1}) que,
 si $V\in \mathcal{C}$, alors 
$\iota_V$ est presqu'un isomorphisme, i.e. son noyau et son conoyau sont de type fini sur $\mathcal{O}_L$. 
L'argument qui suit est inspir\'e des lemmes 5.4 et 5.11 de \cite{HuHa}, et demande quelques pr\'eliminaires. 

\begin{prop}\label{calcul ord}
{\rm    a)} Si $U\in {\rm Rep}^{\rm adm}(A)$ alors 
$${\rm Ord}({\rm Ind}_{\bar{B}}^G(U))\simeq U,\,\, R^1{\rm Ord}({\rm Ind}_{\bar{B}}^G(U))=0.$$

{\rm     b)} On a ${\rm Ord}({\rm St}\otimes \chi\circ \det)\simeq \chi\otimes \chi$ et 
$ R^1{\rm Ord}({\rm St}\otimes \chi\circ \det)=0$.

{\rm     c)} Si $\pi=\chi\circ \det$ alors ${\rm Ord}(\pi)=R^1{\rm Ord}(\pi)=0$.
\end{prop}

\begin{proof}
a) Le premier est le contenu de \cite[lemma 2.3.4, prop.\,4.3.4]{Ordparts1}, 
le second est d\'emontr\'e dans \cite[cor.\,4.2.4]{Hauseux}.

b) Cela d\'ecoule de \cite[th. 4.2.12]{Ordparts2}.

c) Il suffit d'utiliser a) et b) et la suite exacte \'evidente.
\end{proof}

\Subsubsection{Description de ${\cal C}$}
\begin{lemm}\label{exact}
Le foncteur ${\rm Ord}$ est exact sur $\mathcal{C}$ et 
${\rm Ord}(\pi)$ est un $\mathcal{O}_L$-module de type fini pour $\pi\in \mathcal{C}$.        
\end{lemm}

\begin{proof} 
Il suffit de montrer que ${\rm R^1Ord}$ est nul sur $\mathcal{C}$, et par d\'evissage il suffit de voir que ${\rm R^1Ord}(\pi)=0$ pour un objet irr\'eductible $\pi$ de $\mathcal{C}$.
 Un tel $\pi$ est tu\'e par 
${\goth m}_L$, et il existe une extension finie $k$ de $k_L$ telle que $\pi\otimes_k k_L\in \mathcal{C}_k^{\rm ord}$ (lemme~\ref{ord}).
 Comme le foncteur ${\rm Ind}_{\bar B}^G$
est compatible avec les extensions finies du corps des coefficients, il en est de m\^eme du foncteur ${\rm Ord}$
et de ses foncteurs d\'eriv\'es.
 Il suffit donc de voir que ${\rm R^1Ord}$ est nul sur $\mathcal{C}_k^{\rm ord}$, et encore par d\'evissage on se ram\`ene \`a l'annulation de ${\rm R^1Ord}$ sur $\chi\circ \det, {\rm St}\otimes \chi\circ \det$ ou ${\rm Ind}_B^G (\chi_1\otimes \chi_2)$, qui d\'ecoule de la prop.\,\ref{calcul ord}.

Pour le deuxi\`eme point, on sait que ${\rm Ord}(\pi)$ est un $\mathcal{O}_L[A]$-module de torsion, lisse, admissible, et que ${\rm Ord}(\pi)\otimes_{\mathcal{O}_L} k_L={\rm Ord}(\pi\otimes_{\mathcal{O}_L} k_L)$ d'apr\`es ce que l'on vient de d\'emontrer.
 Il suffit de montrer que ${\rm Ord}(\pi\otimes_{\mathcal{O}_L} k_L)$ est de dimension finie sur $k_L$.
 Cela se fait encore par d\'evissage
en utilisant le lemme \ref{ord} et la compatibilit\'e du foncteur partie ordinaire avec le changement du corps des coefficients, et le calcul de 
${\rm Ord}$ sur les objets irr\'eductibles
 de $\mathcal{C}_k^{\rm ord}$ fourni par la prop.\,\ref{calcul ord}.  
\end{proof}

\begin{lemm}\label{finito}
Soit $\pi\in \mathcal{C}$ tel que ${\rm Ord}(\pi)=0$. Alors 
$\pi$ est un $\mathcal{O}_L$-module de type fini.
\end{lemm}

\begin{proof}
Soit $\pi'$ un sous-objet irr\'eductible de $\pi$, alors ${\rm Ord}(\pi')=0$ et donc
${\rm Ord}(\pi'\otimes_{k_L} k)=0$ pour toute extension finie $k$ de $k_L$. Comme 
${\rm Ord}$ est exact sur $\mathcal{C}_k^{\rm ord}$ et prend des valeurs non nulles sur 
les tordus de la Steinberg et les s\'eries principales, le lemme \ref{ord} montre que 
$\pi'$ est de dimension finie sur $k_L$. Comme 
${\rm Ord}$ est exact sur $\mathcal{C}$, on a ${\rm Ord}(\pi/\pi')=0$ et une r\'ecurrence sur la longueur montre que tous les sous-quotients irr\'eductibles de $\pi$ sont de dimension finie sur $k_L$, ce qui permet de conclure.   
\end{proof}

\begin{prop}\label{finito1}
Pour tout $\pi\in \mathcal{C}$ le noyau et le conoyau de la fl\`eche
naturelle $\iota_V: {\rm Ind}_{\bar{B}}^G({\rm Ord}(\pi))\to \pi$ sont de type fini sur 
$\mathcal{O}_L$.          
\end{prop}

\begin{proof}
Soit $\pi^{\rm ord}$ l'image de $\iota_V$, alors par construction 
${\rm Ord}(\pi^{\rm ord})={\rm Ord}(\pi)$ et $\pi/\pi^{\rm ord}, \pi^{\rm ord}\in \mathcal{C}$
donc par le lemme \ref{exact} 
on a ${\rm Ord}(\pi/\pi^{\rm ord})=0$ et donc le conoyau de $\iota_V$ est de type fini sur 
$\mathcal{O}_L$ par le lemme \ref{finito}.
Soit $X=\ker(\iota_V)$. Alors ${\rm Ord}(X)=0$, donc par le lemme \ref{finito} il suffit de montrer que 
$X\in \mathcal{C}$. Comme $X$ est un sous-objet de ${\rm Ind}_{\bar{B}}^G ({\rm Ord}(\pi))$, il suffit de voir que ce dernier est dans $\mathcal{C}$, ce qui d\'ecoule facilement du fait que 
${\rm Ord}(\pi)$ est de type fini sur $\mathcal{O}_L$ (lemme \ref{exact}).
\end{proof}

\begin{coro}\label{finito2}
Soit $\Theta$ un r\'eseau dans un $L$-Banach unitaire admissible $\Pi$ de $G$ tel que 
$\Theta/\varpi_L\Theta\in \mathcal{C}$. Il existe alors une repr\'esentation $\sigma$ de $A$, de type fini 
sur $\mathcal{O}_L$, et un morphisme 
${\rm Ind}_{\bar{B}}^G(\sigma)\to \Theta$
dont le noyau et le conoyau sont de type fini sur $\mathcal{O}_L$.
\end{coro}
\begin{proof}
Il suffit d'appliquer le r\'esultat ci-dessus avec $\pi_n=\Theta/\varpi_L^n\Theta$, en remarquant que 
${\rm Ord}(\pi_{n+1})\otimes_{\mathcal{O}_L/\varpi_L^{n+1}} \mathcal{O}_L/\varpi_L^n={\rm Ord}(\pi_n)$ par exactitude de 
${\rm Ord}$ sur~$\mathcal{C}$, et un r\'esultat semblable pour l'induite parabolique. 
\end{proof}

 \section{Copr\'esentation finie}\label{copr1}
Soit $F$ une extension finie de $\Q_p$.  
Le but de ce chapitre est de d\'emontrer le r\'esultat suivant:   
\begin{theo}\label{smooth} 
Soient $K$ une extension finie de $F$ et $k\geq 1$, $q\geq 0$. 
Le $G^{\prime}$-module $H^q_{\eet}(\mathcal{M}_{n,K}^{\varpi}, \O_L/p^k)$ est le dual de Pontryagin d'une repr\'esentation lisse, de 
pr\'esentation finie de $G^{\prime}$, en particulier c'est un $\O_L$-module profini.
\end{theo}
\begin{rema}\label{copr2}
Si $F=\Q_p$, nous montrons au chap. \ref{lg1} que $H^1_{\eet}(\mathcal{M}_{n,K}^{\varpi}, k_L)^{\vee}$
est de longueur finie et donc admissible. Par contraste, si $F\neq \Q_p$, 
le th.\,\ref{smooth} implique que $H^1_{\eet}(\mathcal{M}_{n,K}^{\varpi}, k_L)^{\vee}$ 
n'est pas admissible si $K$ est assez grand.
(cf.~th.\,\ref{main10bis}). 
\end{rema}
   
   \begin{rema}\label{copr3}
    Il est fort probable qu'en dimension arbitraire on dispose encore du r\'esultat suivant: $H^i_{\eet}(\mathcal{M}_{n,K}^{\varpi}, k_L)^{\vee}$ est un $G^{\prime}$-module lisse, de type fini. 
    \end{rema}
      \subsection{Pr\'eliminaires}\label{copr4}
      La preuve du th\'eor\`eme est purement g\'eom\'etrique et demande quelques pr\'eliminaires.  
Dans beaucoup des \'enonc\'es qui suivent, il est plus naturel de travailler avec $\Z_p/p^k$ plut\^ot que
$\O_L/p^k$; on obtient le r\'esultat pour $\O_L/p^k$ en appliquant $\otimes_{\Z_p}\O_L$.
    \subsubsection{La filtration de Bloch--Kato-Hyodo sur les cycles proches}  \label{copr5}
Nous allons \'etendre aux sch\'emas formels la filtration de
Bloch-Kato-Hyodo sur les faisceaux de cycles proches pour les sch\'emas et l'identification
du gradu\'e; 
que ce soit possible est {\og bien connu\fg} des experts mais
nous n'avons pas trouv\'e de r\'ef\'erence o\`u ceci est fait.
   
  Soient $K$ une extension finie de $\qp$, $\varpi$ une uniformisante, et $k$ le corps r\'esiduel de~$\so_K$.
  Soit $X$ un sch\'ema formel semistable sur $\so_K$ ou le changement de base d'un tel sch\'ema
sur l'anneau des entiers d'un sous-corps de $K$. 
On munit $X$ de la 
log-structure induite par la fibre sp\'eciale.
Soit $M$ le faisceau en mono\"{\i}des sur $X$ d\'efinissant la log-structure, $M^{\gp}$ 
le faisceau en groupes associ\'e. 
Cette log-structure est canonique, dans la terminologie de Berkovich \cite[2.3]{BerL}, 
i.e., $M(U)=\{x\in\so_X(U)| x_K\in\so^{\dual}_{X_K}(U_K)\}$, cf.~\cite[Theorem 2.3.1]{BerL}.
Il s'ensuit que  $M^{\gp}(U)=\so(U_K)^\dual$.

 Soit $i$ l'inclusion de la fibre sp\'eciale $X_0$ dans $X$.
Soit $\R^i\Psi\mathbf{Z}/p^n(j)$ le faisceau sur $X_{0,\eet}$ associ\'e au
pr\'efaisceau ${U}_0\to H^i_{\eet}({U}_K,\mathbf{Z}/p^n(j))$, o\`u $U$ d\'ecrit les
sch\'emas formels, \'etales au-dessus de $X$.
C'est le $i$-i\`eme foncteur d\'eriv\'e du foncteur $\Psi$ des cycles proches
de~\cite[Prop. 4.1]{BV} (o\`u il est not\'e $\Theta$) mais
nous n'utiliserons pas $\Psi$ lui-m\^eme.
 
Pour $U$ comme ci-dessus, la suite exacte de Kummer fournit une fl\`eche
   $$
   i^\dual M^{\rm gp}(U_0)=M^{\rm gp}(U)=\so(U_K)^\dual\stackrel{\partial}{\to} \mathbf{Z}/p^n(1)(U_K)[1],
   $$
qui se faisceautise en une application ``symbole''
   $$
   i^\dual M^{\rm gp}\to \R^1\Psi\mathbf{Z}/p^n(1).
   $$
 En utilisant le cup-produit, cela induit une fl\`eche symbole 
 $$i^\dual(M^{\rm gp})^{\otimes q}\to \R^q\Psi\mathbf{Z}/p^n(q).$$ 

      Soit $(X_2, M_2)$ la r\'eduction modulo $p^2$ de $(X,M)$. 
   On munit $(M_2^{\rm gp})^{\otimes q}$ d'une filtration 
   $$...\subset U^2\subset V^1\subset U^1\subset V^0\subset U^0=(M_2^{\rm gp})^{\otimes q},$$
   d\'efinie comme suit. Si $q=0$ on pose 
   $$U^{m+1}=V^m=0 \,\, \text{si}\,\, m\geq 0;$$
si $q=1$ on pose 
$$V^0=(1+\varpi \mathcal{O}_{X_2}) \varpi^{\mathbf{Z}}, U^m=1+\varpi^m \mathcal{O}_{X_2}, V^m=U^{m+1}, \, m\geq 1;$$
et, enfin, si $q\geq 2$ on note $U^m$ l'image de $U^m(M_2^{\rm gp})\otimes (M_2^{\rm gp})^{\otimes (q-1)}$ et 
$V^m$ l'image de $$U^m(M_2^{\rm gp})\otimes (M_2^{\rm gp})^{\otimes (q-2)}\otimes \varpi^{\mathbf{Z}}+
U^{m+1}.$$ En utilisant la fl\`eche symbole, cela permet de d\'efinir une filtration sur $ \R^q\Psi \mathbf{Z}/p(q)$. Si $a_i\in i^\dual(M^{\rm gp}_2)$, on note $\{a_1,...,a_q\}$ l'image de 
$a_1\otimes...\otimes a_q\in (i^\dual(M^{\rm gp}_2))^{\otimes q}$ dans  $\R^q\Psi \mathbf{Z}/p(q)$
par l'application symbole.

    Soit $Y:=X_0$ avec la log-structure induite,
et soit $\Omega^{q}_{Y/k}$ le faisceau des $q$-diff\'erentielles logarithmiques sur $Y$
 relativement \`a la log-structure sur $k$ induite par $1\mapsto 0$. Posons
    $$B^q_{Y/k}=d \Omega^{q-1}_{Y/k}\subset \Omega^q_{Y/k},\,\, Z^{q}_{Y/k}=\ker(d: \Omega^q_{Y/k}\to \Omega^{q+1}_{Y/k}).$$
    Enfin, notons $\Omega^q_{Y/k,\rm log}$ le sous-faisceau ab\'elien de $\Omega^q_{Y/k}$ engendr\'e par 
    les $\wedge_{i=1}^q d\log a_i$ pour $a_i\in M_Y$. 
    
Le th\'eor\`eme suivant est un simple corollaire de l'analogue alg\'ebrique prouv\'e
par Bloch--Kato-Hyodo dans \cite{BK}, \cite{Hy} (voir aussi \cite[Th. 3.3.1]{Ts})
\begin{theo} \label{generalized}
    Si $e$ est l'indice de ramification absolu de $K$, les gradu\'es de la filtration sur 
    $ \R^q\Psi \mathbf{Z}/p(q)$ sont d\'ecrits comme suit: 
    
   \begin{enumerate}[label={\rm(\arabic*)}]
   
   \item On a des isomorphismes 
   $$U^0/V^0\simeq \Omega^q_{Y/k, \rm log}, \quad V^0/U^1\simeq \Omega^{q-1}_{Y/k, \rm log},$$
   induits respectivement par 
$$\{a_1,...,a_q\}\mapsto \wedge_{i=1}^q \dlog \overline{a_i},\quad
   \{a_1,...,a_{q-1},\varpi\}\mapsto \wedge_{i=1}^{q-1} \dlog \overline{a_i}$$
   
   \item Si $m\in ]0, pe/(p-1)[$ est premier \`a $p$, on a des isomorphismes 
   $$U^m/V^m\simeq \Omega^{q-1}_{Y/k}/B_{Y/k}^{q-1}, \quad V^m/U^{m+1}\simeq \Omega_{Y/k}^{q-2}/Z_{Y/k}^{q-2}$$
   induits respectivement par 
$$\{1{+}\varpi^m x, a_1,...,a_{q-1}\}\mapsto \bar{x}\wedge_{i=1}^{q-1} \dlog \overline{a_i},\quad
   \{1{+}\varpi^m x, a_1,...,a_{q-2}, \varpi\}\mapsto \bar{x}\wedge_{i=1}^{q-2} \dlog \overline{a_i}$$ 
   
   \item Si $m\in ]0, pe/(p-1)[$ est un multiple de $p$, on a des isomorphismes 
   $$U^m/V^m\simeq \Omega^{q-1}_{Y/k}/Z^{q-1}_{Y/k}, \quad V^m/U^{m+1}\simeq \Omega_{Y/k}^{q-2}/Z_{Y/k}^{q-2}$$
   induits respectivement par 
$$\{1{+}\varpi^m x, a_1,...,a_{q-1}\}\mapsto \bar{x}\wedge_{i=1}^{q-1} \dlog \overline{a_i},
  \ \   \{1{+}\varpi^m x, a_1,...,a_{q-2}, \varpi\}\mapsto \bar{x}\wedge_{i=1}^{q-2} \dlog \overline{a_i}$$ 
    
    \item Si $m\geq pe/(p-1)$ alors $U^m=0$.
   \end{enumerate}
    
    \end{theo}
    \begin{proof}
   Commen\c{c}ons par l'isomorphisme $U^0/V^0\simeq \Omega^q_{Y/k,\log}$. 
        Consid\'erons les fl\`eches:
     $$
     \xymatrix{
     U^0_M/V^0_M\ar@{->>}[r]
       & U^0_X/V^0_X\ar[r]^{f_X} & \Omega^q_{Y/k,\log},
     }
     $$
o\`u nous avons ajout\'e des indices pour indiquer o\`u se trouvent les filtrations.
Il s'agit de prouver que $f_X$ est un isomorphisme.
     
Le probl\`eme est local et nous pouvons supposer que $X$ est alg\'ebrisable,
i.e., que $X$ est la compl\'etion $X=\wh{T}$ d'un sch\'ema semistable $T$ over $\so_K$. 
Par \cite[Th. 5.1]{BV}, la compl\'etion induit un isomorphisme naturel:
     \begin{equation}
     \label{alg}
     i^\dual\R^qj_* \mathbf{Z}/p(q) \stackrel{\sim}{\to }\R^q\Psi \mathbf{Z}/p(q), \quad j: T_K\hookrightarrow T.
     \end{equation}
Le faisceau des cycles proches alg\'ebrique
$ i^\dual\R^qj_* \mathbf{Z}/p(q) $ poss\`ede une  filtration $U^{\bullet}_T,V^{\bullet}_T$ 
analogue \`a celle de son analogue formel et on a un diagramme commutatif: 
       $$
     \xymatrix@R=2mm{ &   U^0_X/V^0_X\ar[rd]^{f_X}\\
     U^0_M/V^0_M\ar@{->>}[rd] \ar@{->>}[ru]  & & \Omega^q_{Y/k,\log}\\
      & U^0_T/V^0_T\ar[ru]^{f_T}_{\sim}\ar@{->>}[uu]
     }
     $$
La fl\`eche $f_T$ est  un isomorphisme d'apr\`es le th\'eor\`eme de Bloch--Kato--Hyodo.  
Une chasse au diagramme prouve que $f_X$ est un isomorphisme. 
Cela signifie aussi que la fl\`eche verticale est un isomorphisme qui,  combin\'e
avec l'isomorphisme $U^0_T\stackrel{\sim}{\to} U^0_X$ de la formule (\ref{alg}), montre
que la fl\`eche naturelle $V^0_T\to V^0_X$ est aussi un isomorphisme.  
En r\'esum\'e, nous avons \'etabli les isomorphismes
     $$
     U^0_X/V^0_X\simeq \Omega^q_{Y/k,\log}, \quad V^0_T\stackrel{\sim}{\to} V^0_X. 
     $$
     
Maintenant, partant du dernier isomorphisme, on prouve de mani\`ere analogue les
isomorphismes 
     \begin{align*}
     & V^0_X/U^1_X\simeq  \Omega^{q-1}_{Y/k,\log}, \quad U^1_T\stackrel{\sim}{\to }U^1_X, \\
      & U^1_X/V^1_X\simeq \Omega^{q-1}_{Y/k}/B^{q-1}_{Y/k},\quad V^1_T\stackrel{\sim}{\to }V^1_X,
      \end{align*}
      etc. Puisque, pour $m\geq pe/(p-1)$, on a $U_T^m=0$, cela suffit \`a prouver le th\'eor\`eme.
    \end{proof}
\begin{rema}\label{qequals1} En particulier, pour $q=1$, on  obtient les r\'esultats suivants pour $\R^1\Psi\Z/p(1)$.
\begin{enumerate}
\item
On construit une extension
$$0\to \mathbf{F}_p\to U^0/U^1\to \Omega^{1}_{Y/k, \rm log}\to 0,$$
\`a partir des isomorphismes
$$U^0/V^0\simeq \Omega^{1}_{Y/k, \rm log},\ 
a\mapsto \dlog \bar{a}, \quad V^0/U^1\simeq \mathbf{F}_p, \ \varpi^n \mapsto n\ ({\rm mod}\, p).$$
\item 
        De plus $U^m=0$ si $m\geq pe/(p-1)$, et pour $m\in ]0, pe/(p-1)[$ 
        $$U^m/U^{m+1}\simeq \begin{cases}  \mathcal{O}_Y,  & m \notin p\mathbf{Z},\\ 
        \mathcal{O}_Y/Z^0_{Y/k}, & m\in p\mathbf{Z},
        \end{cases}$$
        l'isomorphisme \'etant induit par $1+\varpi^m a\mapsto \bar{a}$.
        \item 
Comme $Y$ est log-lisse et de type Cartier, on a un isomorphisme 
$$\mathcal{O}_Y\simeq Z^0_{Y/k}, \,\, a\mapsto a^p.$$
\end{enumerate}
\end{rema}

  \subsubsection{Mod\`eles semistables}\label{app1}
Si $K$ est une extension finie de $\Q_p$, un sch\'ema formel sur $\so_K$ est dit
 {\em semistable} si, localement pour la topologie de Zariski,  
il admet une fl\`eche \'etale vers le spectre formel
    $$\Spf(\so_K\{X_1,\ldots,X_n\}/(X_1\cdots X_r-\varpi_K)), \quad 1\leq r\leq n,
    $$ 
o\`u  $\varpi_K$ est une uniformisante de $K$.
  
  Les courbes $\mathcal{M}_n^{\varpi}$ ne sont pas quasi-compactes, donc l'existence d'un mod\`ele semi-stable d\'efini sur une extension finie de $\qp$ n'est pas automatique. 
  Cependant, on dispose du r\'esultat suivant, 
essentiellement prouv\'e dans l'appendice de \cite{CDN1}. 
Il est valable pour $\gl_2(F)$, mais pas pour $\gl_n(F)$.

      \begin{prop}\label{appendix}
      L'espace $\mathcal{M}_n^{\varpi}$ poss\`ede un mod\`ele formel semi-stable 
$G\times \check{G}\times \mathcal{G}_F$-\'equivariant, d\'efini sur une extension finie de $F$. Les composantes irr\'eductibles de la fibre sp\'eciale de ce mod\`ele forment un nombre fini d'orbites sous l'action de $G$, 
et chacune de ces composantes irr\'eductibles est fix\'ee
par un sous-groupe ouvert de~$G$.  
      \end{prop} 
      
      \begin{proof}
Le premier point est \'etabli dans la prop.\,A.1 de \cite{CDN1}. 

  Pour le second point, consid\'erons les diagrammes commutatifs $G$-\'equivariants
 suivants de morphismes de courbes $E$-analytiques et de sch\'emas formels, 
respectivement, qui apparaissent dans la preuve de \cite[prop.\,A.1]{CDN1}.
  $$
  \xymatrix@R=.6cm{
  \mv_{n,E}\ar[r]^{f_1}\ar[d]^{p_2} & \Omega_{\rm Dr,E}\ar[d]^{p_1}\\
  X_{n,E}\ar[r]^{f_2} & X_E,
  }\hskip2cm
  \xymatrix@R=.5cm{ (\mv_{n,E})^{\circ}\ar[r]\ar[d]^-{p^{\circ}_2} & (\mv_{n,E})^{+} \ar[r]^{f_1^+}\ar[d]^-{p^+_2}& \Omega_{\rm Dr,E}^{+}\ar[d]^-{p_1}\\
  X_{n,E}^{\circ}\ar[r] & X_{n,E}^{+}\ar[r]^{f_2^+} & X_E^{+}}
  $$
\begin{itemize}
\item $\Omega_{\rm Dr}$ est le demi-plan de Drinfeld $\mathbb{P}^1\moins\mathbb{P}^1(F)$ sur $F$, 
$\Omega_{\rm Dr}^{+}$ est son mod\`ele semistable standard;
\item $\Gamma\subset G/\varpi^{\Z}$ est un sous-groupe cocompact assez petit pour
que l'action de $\Gamma$ sur l'arbre de Bruhat-Tits soit libre;
\item $X:=\Gamma\backslash\Omega_{\rm Dr}$ est une courbe propre et lisse; 
$X^+:=\Gamma\backslash\Omega^+_{\rm Dr}$ -- le mod\`ele stable (et aussi le mod\`ele semistable minimal) de $X$;
\item $X_n:=\Gamma\backslash\mathcal{M}_n^{\varpi}$;  $E$ est une extension finie de $F$ 
telle que $X_n$ poss\`ede un mod\`ele stable $X_{n,E}^+$;
\item $X_{n,E}^{\circ}$ est le mod\`ele minimal semistable de $X_{n,E}$ obtenu \`a partir de
 $X_{n,E}^+$ en commen\c{c}ant par \'eclater les points singuliers puis en \'eclatant les points d'auto-intersection.
\end{itemize}
  
Les diagrammes ci-dessus sont cart\'esiens et
$(\mv_{n,E})^{\circ}$ est un mod\`ele semistable de $\mv_{n,E}$. De plus, la
triangulation $S(X_{n,E}^{+})$ associ\'ee est le quotient de la triangulation $S((\mv_{n,E})^{+} )$ 
par $\Gamma$. Comme
    $S(X_{n,E}^{+})$ et $S((\mv_{n,E})^{+} )$ sont toutes les deux
les ensembles des pr\'eimages (par sp\'ecialisation), dans $X_{n,E}$ et $\mv_{n,E}$ respectivement, 
des points g\'en\'eriques des composantes irr\'eductibles de leurs fibres sp\'eciales respectives,
la fibre sp\'eciale de $(\mv_{n,E})^{+} $ a un nombre fini de $\Gamma$-orbites (et donc de $G$-orbites) de
composantes irreductibles. 
Il en est donc de m\^eme, pour le mod\`ele
semistable $(\mv_{n,E})^{\circ} $ (essentiellement par sa d\'efinition).  

De plus, la fibre sp\'eciale est munie de la topologie discr\`ete; il en r\'esulte que l'orbite
sous l'action de $\gl_2(\O_F)$ d'une composante irr\'eductible est finie (car $\gl_2(\O_F)$ est
profini) et donc que le stabilisateur d'une composante est d'indice fini et ouvert
(car ferm\'e) et agit \`a travers un quotient fini
sur la composante (qui est donc fix\'ee par un sous-groupe ouvert).
Ceci permet de conclure.
\end{proof}

\subsubsection{Faisceaux cohomologiquement profinis de copr\'esentation finie}\label{copr6}

 Ce num\'ero contient quelques pr\'eliminaires de nature topologique, qui seront utilis\'es dans la preuve du th\'eor\`eme principal de ce chapitre. 

     Rappelons quelques sorites sur les pro-objets d'une cat\'egorie, en suivant \cite{Porter}. Soit 
     $\mathcal{A}$ une cat\'egorie ab\'elienne. La cat\'egorie ${\rm Pro}(\mathcal{A})$ des syst\`emes projectifs (ou prosyst\`emes) d'objets de $\mathcal{A}$ est aussi ab\'elienne. Soit $\mathcal{A}_1$ une sous-cat\'egorie \'epaisse de $\mathcal{A}$, i.e. pour toute suite exacte $0\to A\to B\to C\to 0$ dans 
     $\mathcal{A}$, $B$ est un objet de $\mathcal{A}_1$ si et seulement si $A$ et $C$ sont des objets de $\mathcal{A}_1$. Notons 
     $E(\mathcal{A}_1)$ la sous-cat\'egorie pleine de ${\rm Pro}(\mathcal{A})$ form\'ee des prosyst\`emes isomorphes dans ${\rm Pro}(\mathcal{A})$ \`a un prosyst\`eme d'objets de $\mathcal{A}_1$. Alors $E(\mathcal{A}_1)$ est aussi une sous-cat\'egorie \'epaisse de la cat\'egorie ${\rm Pro}(\mathcal{A})$ (prop.\,2.9 de \cite{Porter}).   

    Nous allons travailler avec la cat\'egorie 
$\mathcal{A}={\rm Mod}_{\O_L}$ des $\O_L$-modules (discrets) et 
avec sa sous-cat\'egorie \'epaisse 
    $\mathcal{A}_1={\rm Mod}_{\O_L}^{\rm \ell f}$ des $\O_L$-modules de longueur finie. 
    Les objets de la sous-cat\'egorie \'epaisse $E({\rm Mod}_{\O_L}^{\rm \ell f})$ de ${\rm Pro}({\rm Mod}_{\O_L})$ seront appel\'es des {\rm prosyst\`emes profinis}. 
     Le crit\`ere de Mittag-Leffler combin\'e avec la prop.\,2.3 de \cite{Porter} montrent que le foncteur 
   limite projective de $E({\rm Mod}_{\O_L}^{\rm \ell f})$ dans la cat\'egorie des $\O_L$-modules topologiques profinis est exact. 
   
   Soit maintenant $Y$ un $\O_L$-sch\'ema et soit $\{U_s\}_{s\geq 1}$ un recouvrement croissant de $Y$ par des ouverts quasi-compacts. Soit  
      ${\rm Mod}_{\O_L}(Y_{\eet})$ la cat\'egorie des faisceaux de $\O_L$-modules sur $Y_{\eet}$. 
On dispose de foncteurs
   $${\rm SP}^i: {\rm Mod}_{\O_L}(Y_{\eet})\to {\rm Pro}({\rm Mod}_{\O_L}), 
\quad {\rm SP}^i(\mathcal{F}):=\{H^i_{\eet}(U_s, \mathcal{F}|_{U_s})\}_s.$$
   On dit que $\mathcal{F}\in {\rm Mod}_{\O_L}(Y_{\eet})$ est 
  {\em cohomologiquement profini} si 
  ${\rm SP}^i(\mathcal{F})\in E({\rm Mod}_{\O_L}^{\rm \ell f})$ pour tout $i\geq 0$. La discussion ci-dessus montre que si 
  $\mathcal{F}$ est cohomologiquement profini, alors 
  $H^i_{\eet}(Y, \mathcal{F})$ est isomorphe \`a $\varprojlim_{s} H^i_{\eet}(U_s, \mathcal{F}|_{U_s})$, et donc est un 
  $\O_L$-module profini, pour tout $i\geq 0$. En effet, dans la suite exacte de Milnor
  $$0\to {\rm R^1}\varprojlim_{s} H^{i-1}_{\eet}(U_s, \mathcal{F}|_{U_s})\to H^i_{\eet}(Y, \mathcal{F})\to \varprojlim_{s} H^i_{\eet}(U_s, \mathcal{F}|_{U_s})\to 0$$
  le terme \`a gauche est nul par le crit\`ere de Mittag-Leffler.

   Soit $G'$ un groupe de Lie $p$-adique pour lequel la cat\'egorie des $\O_L$-repr\'esentations lisses de pr\'esentation finie est ab\'elienne. 
   On suppose que $G'$ agit sur $Y$ de telle sorte que chaque $U_s$ est fix\'e par un sous-groupe ouvert de $G'$. 
   Soit 
      ${\rm Mod}_{\O_L}^{G'}(Y_{\eet})$ la cat\'egorie des faisceaux $G'$-\'equivariants de $\O_L$-modules sur 
      $Y_{\eet}$. On dit que $\mathcal{F}\in {\rm Mod}^{G'}_{\O_L}(Y_{\eet})$
  est 
 {\em $G'$-continu} si pour tous $i\geq 0$ et $s\geq 1$
 le $\O_L$-module $H^i_{\eet}(U_s, \mathcal{F}):=H^i_{\eet}(U_s, \mathcal{F}|_{U_s})$ est tu\'e par une puissance de $p$ et fix\'e par un sous-groupe ouvert compact de $G'$.
    Si $\mathcal{F}$ est cohomologiquement profini et 
 $G'$-continu, l'action de $G'$ sur le $\O_L$-module profini $H^i_{\eet}(Y,\mathcal{F})$ 
est continue pour tout $i\geq 0$, et donc 
 $H^i_{\eet}(Y,\mathcal{F})^{\vee}$ (dual de Pontryagin) est une repr\'esentation lisse de $G'$ sur $\O_L$. 
On dit que 
$\mathcal{F}$ est {\it cohomologiquement profini
de copr\'esentation finie} (abr\'eg\'e en {\it cpcf}) si $\mathcal{F}$ est cohomologiquement profini, 
$G'$-continu, et si les $H^i_{\eet}(Y, \mathcal{F})^{\vee}$, $i\geq 0$,
sont des repr\'esentations lisses de pr\'esentation finie de $G'$. 
      
 \begin{lemm}\label{two-out-of-three}
    Soit $0\to \sff_1\to \sff\to \sff_2\to 0$ une suite exacte dans ${\rm Mod}_{\O_L}^{G'}(Y_{\eet})$.

{\rm (i)}
Si deux des trois faisceaux sont cohomologiquement profinis, il en est de m\^eme du troisi\`eme.

{\rm (ii)}
Si deux des trois faisceaux sont $G'$-continus, il en est de m\^eme du troisi\`eme.

{\rm (iii)} Si les trois faisceaux sont cohomologiquement profinis et $G'$-continus, et si 
deux des trois faisceaux sont cpcf, il en est de m\^eme du troisi\`eme.
 \end{lemm}
   
   \begin{proof} 
Le (i) r\'esulte de la suite exacte longue dans ${\rm Pro}({\rm Mod}_{\O_L})$
$$\cdots\to {\rm SP}^{i-1}(\mathcal{F}_1)\to {\rm SP}^{i-1}(\mathcal{F})\to {\rm SP}^{i-1}(\mathcal{F}_2)\to {\rm SP}^{i}(\mathcal{F}_1)\to {\rm SP}^{i}(\mathcal{F})\to 
{\rm SP}^{i}(\mathcal{F}_2)\to\cdots$$
et du fait que $E({\rm Mod}^{\rm \ell f}_{\O_L})$ est \'epaisse dans ${\rm Pro}({\rm Mod}_{\O_L})$.

Passons au (ii), et supposons que $\mathcal{F}_1$ et $\mathcal{F}_2$ sont $G'$-continus (la preuve des deux autres cas est identique). Fixons 
$s\geq 1$. Par hypoth\`ese il existe un sous-groupe ouvert $K$ de $G'$ qui agit trivialement sur $U_s$ et sur 
$H^i_{\eet}(U_s, \mathcal{F}_1)$ et $H^i_{\eet}(U_s, \mathcal{F}_2)$. On peut supposer que $K$ est pro-$p$ et uniforme. Soit 
$N\geq 1$ un entier tel que $p^N$ tue $H^i_{\eet}(U_s, \mathcal{F}_1)$ et $H^i_{\eet}(U_s, \mathcal{F}_2)$. 
La suite longue de cohomologie et le lemme \ref{extension} montrent qu'un sous-groupe ouvert de 
$K$ (donc de $G'$) agit trivialement sur $H^i_{\eet}(U_s, \mathcal{F})$, donc 
$\mathcal{F}$ est $G'$-continu.
 
Pour le (iii), supposons que  $\sff_1, \sff_2$ sont cpcf,
et montrons qu'il en est de m\^eme de~$\sff$
(la preuve des autres cas est identique). On a une suite exacte 
   $$ H^{i-1}_{\eet}(Y, \sff_2)\to H^i_{\eet}(Y, \sff_1)\to H^i_{\eet}(Y,\sff)\to H^i_{\eet}(Y, \sff_2)\to H^{i+1}_{\eet}(Y, \sff_1).$$
   Par (i) et (ii) $\mathcal{F}$ est cohomologiquement profini et $G'$-continu, on a donc 
une suite exacte 
$$ H^{i+1}_{\eet}(Y, \sff_1)^{\vee}\to H^i_{\eet}(Y, \sff_2)^{\vee}\to  H^i_{\eet}(Y,\sff)^{\vee}\to H^i_{\eet}(Y, \sff_1)^{\vee}\to  H^{i-1}_{\eet}(Y, \sff_2)^{\vee}$$
de $\O_L[G']$-modules lisses. 
Par hypoth\`ese tous les termes sauf celui du milieu sont de pr\'esentation finie. 
On conclut gr\^ace au lemme suivant.
\end{proof}

\begin{lemm}\label{extpf}
Si 
$A\to B\to C\to D\to E$ est une suite exacte de $\O_L[G']$-modules lisses, 
avec $A,B,D,E$ de pr\'esentation finie, alors 
$C$ l'est aussi.
\end{lemm}

\begin{proof}
Soit $\delta: D\to E$ et $\alpha: A\to B$, alors on a une suite exacte 
$$0\to \coker(\alpha)\to C\to \ker(\delta)\to 0.$$
   D'apr\`es le \S\,\ref{Shotton},
 $\coker(\alpha)$ et $\ker(\delta)$ sont lisses de pr\'esentation finie.  Comme les repr\'esentations lisses de pr\'esentation finie sont stables par extension, $C$ est lisse de pr\'esentation finie, ce qui permet de conclure.
\end{proof}

\subsection{Preuve du Th\'eor\`eme~\ref{smooth}}\label{smoothie} 
Passons \`a la preuve du Th\'eor\`eme~\ref{smooth}.

\vskip2mm
\noindent
$\bullet$ {\it Pr\'eliminaires}.---
  Soit $X'$ un mod\`ele formel semi-stable de $\mathcal{M}_n^{\varpi}$, 
d\'efini sur une extension finie $F'$ de $F$ 
(un tel mod\`ele existe d'apr\`es le \S\,\ref{app1}). Soit $K'$ une extension finie de $\Q_p$
contenant $F'$, $K$ et les racines $p$-i\`emes de l'unit\'e, et telle que $K'/K$ soit galoisienne. Soit $k'$ le corps r\'esiduel de $K'$. 
Soient $X:=X'\otimes_{\so_{F^{\prime}}} \mathcal{O}_{K'}$ et $Y$ sa fibre sp\'eciale. 
 Il existe une suite de sous-sch\'emas ferm\'es
(resp.~ouverts)
$Y_s, s\in\N,$ (resp. $U_s, s\in \N$) de $Y$ telle que:

(i) $Y_s$ est une r\'eunion finie de composantes irr\'eductibles,

(ii) $Y_s\subset U_s\subset Y_{s+1}$ et leur r\'eunion est $Y,$

(iii) les tubes  $\{U_{s,\eta}:=]U_{s}[_{X}\},s\in\N,$ forment un recouvrement Stein de $X_{K'}$.

\noindent
(Il suffit de prendre  pour $Y_0$ une composante irr\'eductible, de d\'efinir $Y_{s+1}$
comme la r\'eunion de $Y_s$ et des composantes irr\'eductibles d'intersection non vide avec $Y_s$, et
de prendre pour $U_{s-1}$ le compl\'ementaire dans $Y_s$ des composantes de $Y_{s+1}$ non incluses dans $Y_s$.)

On voit $X$ comme un log-sch\'ema formel muni de la log-structure venant de $Y$.
On munit $Y, U_s, Y_s$ de la log-structure induite.

\vskip2mm
\noindent $\bullet$
{\it \'Etape 1}.---   Montrons que le faisceau $\Omega^t_{Y/k'}$ est cpcf pour tout $t\geq 0$. 
Il est  
$G'$-continu, car il existe un sous-groupe ouvert de $G'$ qui fixe $U_s$, et ce
sous-groupe agit trivialement sur $H^i_{\eet}(U_s, \Omega^t_{Y/k'}|_{U_s})$ pour tout $i$. 
Il est cohomologiquement profini car le prosyst\`eme
 $\{H^i_{\eet}(U_s,  \Omega^t_{Y/k'}|U_s)\}_s$ est isomorphe au prosyst\`eme des 
$H^i(Y_s, \Omega^t_{Y/k'}\otimes_{\so_Y}\so_{Y_s})$ et ces groupes sont finis
car $Y_s$ est propre et le faisceau $\Omega^t_{Y/k'}\otimes_{\so_Y}\so_{Y_s}$ 
est localement libre de rang fini.

    Il nous reste \`a montrer que le $\O_L[G']$-module 
$H^s_{\eet}(Y, \Omega^t_{Y/k'})^{\vee}$ 
(lisse, d'apr\`es ce que l'on vient de d\'emontrer) 
est de pr\'esentation finie pour tout $s$.  
Soit $\{C_j\}_{j\in J}$ l'ensemble des composantes irr\'eductibles de $Y$. Par Mayer-Vietoris ferm\'e,
on a une suite exacte, avec $C_{i,j}:=C_i\cap C_j$ 
     \begin{equation}
     \label{maps1}
\xymatrix@R=3mm@C=4mm{
       \prod_{j}H^{s-1}_{\eet}(C_{j},\Omega^t_{C_{j}})\ar[r]&
 \prod_{i<j}H^{s-1}_{\eet}(C_{i,j},\Omega^t_{C_{i,j}})\ar[d]\\
& H^{s}_{\eet}(Y,\Omega^t_{Y/k'})\ar[d]\\ 
& \prod_jH^{s}_{\eet}(C_j,\Omega^t_{C_j})\ar[r]&  \prod_{i<j}H^{s}_{\eet}(C_{i,j},\Omega^t_{C_{i,j}})}
     \end{equation}
o\`u l'on a pos\'e
 $\Omega^t_{C_{j}}:=\Omega^t_{Y/k'}\otimes\so_{C_j}$ et $\Omega^t_{C_{i,j}}:=\Omega^t_{Y/k'}\otimes\so_{C_{i,j}}$. 

Il suffit de montrer que les $G^{\prime}$-modules 
$\prod_{j}H^*_{\eet}(C_{j},\Omega^t_{C_{j}})$ and $\prod_{i<j}H^*_{\eet}(C_{i,j},\Omega^t_{C_{i,j}})$
sont isomorphes \`a  $I(W)^{\vee}\simeq {\rm Ind}_{H}^{G^{\prime}}(W^{\vee})$ 
pour un sous-groupe ouvert compact $H$ de $G'$ et une repr\'esentation de dimension finie $W$ de $H$. En effet, si c'est le cas la suite exacte ci-dessus est automatiquement stricte (tous les objets intervenant sont profinis), on peut donc la dualiser et conclure en utilisant le lemme \ref{extpf}.

 Dans le cas de $\prod_{j}H^*_{\eet}(C_{j},\Omega^t_{C_{j}})$, chaque sch\'ema $C_j$ 
est propre et lisse et les faisceaux 
$\Omega^t_{C_j}$ sont des $\so_{C_j}$-modules
localement libres de rang fini, donc les $H^*_{\eet}(C_j, \Omega^t_{C_j})$ 
sont de dimension finie sur $k'$. Comme les 
$C_j$ ne forment qu'un nombre fini d'orbites sous l'action de $G^{\prime}$ (on choisit un syst\`eme $J$
de repr\'esentants), 
et comme le stabilisateur $K_j$ de chaque $C_j$ est ouvert dans $G^{\prime}$ (prop.\,\ref{appendix}) 
et donc est conjugu\'e \`a un sous-groupe d'indice fini dans $\gl_2(\O_F)$, on peut supposer, en rempla\c{c}ant $C_j$
par un translat\'e, que $K_j$ est d'indice fini dans $\gl_2(\O_F)$.
Alors 
$$\prod_{j\in I} H^*_{\eet}(C_j, \Omega^t_{C_j})\simeq {\rm Ind}_{\gl_2(\O_F)}^{G^{\prime}}W,\quad
{\text{o\`u }} W=\bigoplus_{j\in J}{\rm Ind}_{K_j}^{\gl_2(\O_F)}H^*_{\eet}(C_j, \Omega^t_{C_j})$$
est de dimension finie sur $k'$.

Pour traiter le cas de $\prod_{i<j}H^*_{\eet}(C_{i,j},\Omega^t_{C_{i,j}})$, on peut raisonner
de m\^eme en rempla\c{c}ant $\gl_2(\O_F)$ par le sous-groupe d'Iwahori $I$ (qui est le sous-groupe de $\gl_2(\O_F)$ stabilisant
l'ar\^ete \'evidente en niveau $0$).

\vskip2mm
\noindent $\bullet$
{\it \'Etape 2}.--- Montrons que le faisceau $\R^j\Psi \mathbf{F}_p$, ainsi que ceux intervenant dans la filtration de Bloch-Kato-Hyodo sont cpcf. 

Puisque $K'$ contient les racines $p$-i\`emes de l'unit\'e, on a 
$\mathbf{F}_p(j)\simeq \mathbf{F}_p$. 
D'apr\`es le th.\,\ref{generalized}, 
 le faisceau \'etale  $\mathcal{F}^j:=\R^j\Psi \mathbf{F}_p(j)$ sur $Y$ 
poss\`ede une filtration finie $$0=\mathcal{F}_d\subset...\subset \mathcal{F}_1\subset \mathcal{F}_0=\mathcal{F}^j,$$
      dont les quotients successifs $\mathcal{F}_r/\mathcal{F}_{r+1}$ sont de la 
forme\footnote{Dans le cas que nous consid\'erons, $Y$ est une courbe, et seuls $m=0,1$ interviennent.}
      $$\Omega^m_{Y/k',\log}, \quad \Omega^m_{Y/k'}/B^m_{Y/k},\quad \Omega^m_{Y/k'}/Z^m_{Y/k'}.
      $$
   
     Consid\'erons les suites exactes de faisceaux
 \begin{align*}
 & 0\to \Omega^m_{Y/k',{\rm log}}\to Z^m_{Y/k'}\stackrel{C-1}{\longrightarrow}\Omega^m_{Y/k'}\to 0,\\
 &   0\to B^m_{Y/k'}\to Z^m_{Y/k'}\stackrel{C}{\to} \Omega^m_{Y/k'}\to 0,\\
& 0\to Z^m_{Y/k'}\to \Omega^m_{Y/k'}\stackrel{d}{\to} B^{m+1}_{Y/k'}\to 0.
\end{align*}
Tous les faisceaux intervenant dans ces suites exactes sont nuls pour $m$ assez grand (en fait pour $m>1$ puisque $Y$ est une courbe). 
En utilisant cette observation, l'\'etape $1$ et le lemme \ref{two-out-of-three}, on montre par r\'ecurrence descendante sur $m$ que 
$B^m_{Y/k'}, Z^m_{Y/k'}, \Omega^m_{Y/k',{\rm log}}$ sont cpcf, et deux nouvelles applications du lemme \ref{two-out-of-three} montrent que 
$\Omega^m_{Y/k'}/B^m_{Y/k}$, $\Omega^m_{Y/k'}/Z^m_{Y/k'}$ et enfin chaque $\mathcal{F}_k$ est aussi cpcf. En particulier $\R^j\Psi \mathbf{F}_p$ est cpcf.

\vskip2mm
\noindent $\bullet$
{\it \'Etape 3}.--- Montrons que le prosyst\`eme $\{H^{q}_{\eet}(U_{s,\eta}, \Z/p^k)\}_s$ est profini pour tous $q$ et $k$, et que l'action de $G'$ y est continue, i.e. chaque $H^{q}_{\eet}(U_{s,\eta}, \Z/p^k)$ est fix\'e par un sous-groupe ouvert compact de $G'$. On a une suite exacte de prosyst\`emes
$$\{H^{q}_{\eet}(U_{s,\eta}, \Z/p^{k-1})\}_s\to \{H^{q}_{\eet}(U_{s,\eta}, \Z/p^k)\}_s\to \{H^{q}_{\eet}(U_{s,\eta}, \Z/p)\}_s$$
Si le r\'esultat est connu pour les deux termes extr\^emes, il s'obtient pour le terme au milieu en utilisant le fait que les prosyst\`emes profinis forment une sous-cat\'egorie \'epaisse de 
${\rm Pro}({\rm Mod}_{\O_L})$ et le lemme \ref{extension}. On peut donc supposer que $k=1$.

D'apr\`es~\cite[cor.\,4.2]{BV}, on dispose, pour tout $s$, d'une suite spectrale 
$$E_2^{i,j}=H^i_{\eet}(U_s, \R^j\Psi \mathbf{F}_p)\Longrightarrow H^{i+j}_{\eet}(U_{s,\eta}, \mathbf{F}_p).$$
Comme les prosyst\`emes $\{H^i_{\eet}(U_s, \R^j\Psi \mathbf{F}_p)\}_s$ (pour $i,j\geq 0$) sont profinis (\'etape $2$), 
il en r\'esulte que le prosyst\`eme
 $\{H^{q}_{\eet}(U_{s,\eta}, \mathbf{F}_p)\}_s$ est finiment filtr\'e par des prosyst\`emes 
 dont les gradu\'es associ\'es
sont dans $E({\rm Mod}^{\rm \ell f}_{\O_L})$, et donc est profini car 
$E({\rm Mod}^{\rm \ell f}_{\O_L})$ est une sous-cat\'egorie \'epaisse de 
${\rm Pro}({\rm Mod}_{\O_L})$. Le m\^eme argument coupl\'e au lemme \ref{extension} fournit la continuit\'e de l'action de $G'$.

\vskip2mm
\noindent $\bullet$
{\it \'Etape 4}.--- Montrons que $H^{q}_{\eet}(\mathcal{M}_{n,K'}^{\varpi}, \Z/p^k)$ est profini, dual d'une repr\'esentation lisse de pr\'esentation finie de 
$G'$. Le caract\`ere profini et la continuit\'e de l'action de $G'$ sur $H^{q}_{\eet}(\mathcal{M}_{n,K'}^{\varpi}, \Z/p^k)$ d\'ecoulent directement de l'\'etape $3$, qui montre que 
$$H^{q}_{\eet}(\mathcal{M}_{n,K'}^{\varpi}, \Z/p^k)\simeq \varprojlim_{s} H^{q}_{\eet}(U_{s,\eta}, \Z/p^k).$$
En particulier $H^{q}_{\eet}(\mathcal{M}_{n,K'}^{\varpi}, \Z/p^k)^{\vee}$ est bien un $\O_L[G']$-module lisse. Il reste \`a voir qu'il est de pr\'esentation finie. Les suites exactes
$$\xymatrix@R=3mm@C=3mm{
H^{q-1}_{\eet}(\mathcal{M}_{n,K'}^{\varpi}, \Z/p)\ar[d]\\
 H^{q}_{\eet}(\mathcal{M}_{n,K'}^{\varpi}, \Z/p^{k-1})\ar[r]
&H^{q}_{\eet}(\mathcal{M}_{n,K'}^{\varpi}, \Z/p^{k})\ar[r]
& H^{q}_{\eet}(\mathcal{M}_{n,K'}^{\varpi}, \Z/p)\ar[d]\\
&& H^{q+1}_{\eet}(\mathcal{M}_{n,K'}^{\varpi}, \Z/p^{k-1})}$$
r\'eduisent le probl\`eme au cas $k=1$ (cela utilise implicitement le fait que tous les objets sont profinis, comme on vient de le voir, donc on peut dualiser ces suites exactes et ensuite appliquer le lemme \ref{extpf}). Mais ce cas d\'ecoule du fait que les $\R^j\Psi \mathbf{F}_p$  sont cpcf (\'etape $2$), de la suite spectrale 

$$E_2^{i,j}=H^i_{\eet}(Y, \R^j\Psi \mathbf{F}_p)\Longrightarrow H^{i+j}_{\eet}(\mathcal{M}_{n,K'}^{\varpi}, \mathbf{F}_p)$$
et d'un raisonnement comme dans l'\'etape $3$ (remplacer l'usage du lemme \ref{extension} par celui de la prop.\,\ref{Shot25}).

\begin{rema}
Les arguments ci-dessus prouvent en fait le r\'esultat suivant:
\begin{quote}
{\it Si $X$ est un sch\'ema formel semistable sur $\so_K$ dont la fibre sp\'eciale a
une stratification comme celle de la preuve du th.\,\ref{smooth}, alors les
$H^q_{\eet}(X_K,{\mathbf F}_p)$, pour $q\geq 0$,  sont profinis.} 
\end{quote}
\end{rema}

\vskip2mm
\noindent $\bullet$
{\it \'Etape 5}.---Nous venons de d\'emontrer l'existence d'une extension finie galoisienne $K'$ de 
   $K$ telle que $H^i_{\eet}(\mathcal{M}_{n,K'}^{\varpi},\O_L/p^k)$, $i,k\geq 0$, 
est co-lisse, de copr\'esentation finie. Il reste \`a expliquer comment descendre \`a $K$.
       
    Soit $H:={\rm Gal}(K^{\prime}/K)$ et soit $\{V_s\}_{s}$ un recouvrement Stein 
    de $X:=\mathcal{M}_{n,K}^{\varpi}$, donc les $V'_s:=V_s\otimes_K K'$ fournissent un recouvrement Stein de $\mathcal{M}_{n,K'}^{\varpi}$. On a des suites spectrales de
     Hochschild-Serre \cite[prop.\,2.6.12]{Hu}
  $$
  E^{i,j}_2=H^i(H, H^j_{\eet}(V'_s,\O_L/p^k))\Rightarrow H^{i+j}_{\eet}(V_s,\O_L/p^k).
  $$
Comme les prosyst\`emes $\{H^j_{\eet}(V'_s,\O_L/p^k)\}_s$ sont profinis (\'etape $3$)
et $H$ est fini, le complexe calculant la cohomologie de $H$ est form\'e de prosyst\`emes profinis.
Donc les prosyst\`emes 
$\{H^i(H, H^j_{\eet}(V'_s,\O_L/p^k))\}_s$ sont profinis. 
A partir de l\`a, on peut conclure comme dans l'\'etape 1 que les prosyst\`emes
$\{H^{q}_{\eet}(V_s,\O_L/p^k)\}_s$, pour $q\geq 0$, sont profinis. De plus, nous avons vu que l'action de 
$G'$ sur $\{H^j_{\eet}(V'_s,\O_L/p^k)\}_s$ est continue, donc il en est de m\^eme de celle sur $E^{i,j}_2$, et le lemme 
\ref{extension} coupl\'e \`a la suite spectrale ci-dessus permettent de d\'eduire la continuit\'e de l'action de 
$G'$ sur $\{H^{q}_{\eet}(V_s,\O_L/p^k)\}_s$. 
  
On a donc une suite spectrale
   $$
     E^{i,j}_2=H^i(H, H^j_{\eet}(X_{K'},\O_L/p^k))\Rightarrow H^{i+j}_{\eet}(X,\O_L/p^k)
   $$
avec des objets profinis sur toutes les pages et aboutissement profini. De plus, tous les objets sont munis d'une action continue de 
$G'$. 
Dualiser ne pose donc pas de probl\`eme donc on peut argumenter comme ci-dessus: en partant du fait
 que les $H^j_{\eet}(X_{K'},\O_L/p^k)$ sont co-lisses de copr\'esentation finie, on en d\'eduit 
 que l'aboutissement $H^{q}_{\eet}(X,\O_L/p^k)$ est aussi co-lisse, de copr\'esentation finie,
comme on le voulait.

  \begin{coro}\label{pointwise1}
   Pour toute repr\'esentation continue $\bar{\rho}: \mathcal{G}_{F}\to {\rm GL}_d(k_L)$ le 
   $\gl_2(F)$-module ${\rm Hom}_{k_L[\mathcal{G}_{F}]}(\bar{\rho}, H^1_{\eet}(\mathcal{M}_{n,\C_p}^{\varpi},k_L))^{\vee}$ est lisse, de pr\'esentation finie.
 \end{coro}
 
 \begin{proof}
  Soit $F'$ une extension galoisienne finie de $F$ telle que $\bar{\rho}$ se factorise par $H:={\rm Gal}(F'/F)$. 
  Alors 
  $${\rm Hom}_{k_L[\mathcal{G}_{F}]}(\bar{\rho}, H^1_{\eet}(\mathcal{M}_{n,\C_p}^{\varpi},k_L))^{\vee}\simeq{\rm Hom}_H(\bar{\rho}, H^1_{\eet}(\mathcal{M}_{n,\C_p}^{\varpi},k_L)^{\mathcal{G}_{F'}}).$$
  Si l'on montre que $H^1_{\eet}(\mathcal{M}_{n,\C_p}^{\varpi},k_L)^{\mathcal{G}_{F'}}\simeq\pi^{\vee}$ pour une repr\'esentation lisse de pr\'esentation finie $\pi$, alors 
  $${\rm Hom}_{H} (\bar{\rho}, \pi^{\vee})\simeq(\bar{\rho}^{\vee}\otimes_{k_L} \pi^{\vee})^H\simeq((\bar{\rho}\otimes_{k_L} \pi)_H)^{\vee},$$
  et $(\bar{\rho}\otimes_{k_L}\pi)_H$ est lisse, de pr\'esentation finie 
d'apr\`es la prop.\,\ref{Shot27} (car $\bar{\rho}$ est de dimension finie et 
  $H$ est fini), ce qui permet de conclure. 
  
  Pour montrer l'existence d'une repr\'esentation $\pi$ comme ci-dessus,
prenons un recouvrement Stein $\{V_s\}$ de 
  ${\mathcal M}^{\varpi}_{n,F^{\prime}}$. 
D'apr\`es \cite[cor.\, 3.7.5]{dJvdp}, on dispose de suites spectrales de
Hochschild-Serre pour les rev\^etements $V_{s,\C_p}/V_s$, 
ce qui fournit des suites exactes
  $$\xymatrix@R=3mm@C=5mm{
  0\ar[r]& H^1(\mathcal{G}_{F'}, H^0(V_{s,\C_p},k_L))\ar[r]& H^1(V_s,k_L)\ar[r]& 
  H^1(V_{s,\C_p},k_L)^{\mathcal{G}_{F'}} \ar[d]\\
&&& H^2(\mathcal{G}_{F'}, H^0(V_{s,\C_p},k_L))}$$
Un passage \`a la limite sur $s$ fournit la suite exacte
$$\xymatrix@R=3mm@C=5mm{
  0\ar[r]& H^1(\mathcal{G}_{F'}, H^0(\mathcal{M}_{n,\C_p}^{\varpi},k_L))\ar[r]& H^1(\mathcal{M}_{n, F'}^{\varpi},k_L)\ar[r]& 
  H^1(\mathcal{M}_{n,\C_p}^{\varpi},k_L)^{\mathcal{G}_{F'}} \ar[d]\\
&&& H^2(\mathcal{G}_{F'}, H^0(\mathcal{M}_{n,\C_p}^{\varpi},k_L))}$$
  Notons que $H^0(\mathcal{M}_{n,\C_p}^{\varpi},k_L)$ est de dimension finie, donc les termes extr\`emes de cette suite exacte sont de dimension finie.
 Un raisonnement comme dans la preuve du th\'eor\`eme ci-dessus montre que tous les termes sont profinis et munis d'une action continue de $G'$.
 Comme tous les termes \`a l'exception de $  H^1(\mathcal{M}_{n,\C_p}^{\varpi},k_L)^{\mathcal{G}_{F'}}$ sont des duaux de repr\'esentations lisses de pr\'esentation finie, il en est de m\^eme de $  H^1(\mathcal{M}_{n,\C_p}^{\varpi},k_L)^{\mathcal{G}_{F'}}$, ce qui permet de conclure.
 \end{proof}

\subsection{Non admissibilit\'e pour $F\ne \Q_p$}\label{copr7}
On suppose $F\neq\Q_p$ dans ce qui suit.  
On a alors
le r\'esultat suivant qui contraste avec la situation dans le cas $F=\Q_p$.
   \begin{theo}\label{main10bis}
    Si $F\ne \Q_p$ et $n>0$, alors 
    $H^1_{\eet}({\cal M}_{n,K}^{\varpi}, k_L)^{\vee}$ n'est pas admissible si $K$ est
une extension finie, assez grande, de $F$.   
   \end{theo}

Nous allons prouver le r\'esultat plus pr\'ecis du th.~\ref{main10} ci-dessous (le th.\,\ref{main10bis}
est une cons\'equence de la non-admissibilit\'e
de $\Pi_M(V^+_1)$ dans les notations du th.\,\ref{main10}).

Soient $M$ un $(\varphi,N,\G_F)$-module supercuspidal et ${\rm JL}(M)^+$ un r\'eseau de ${\rm JL}(M)$
stable par $\check G$.
Si $V$ une $L$-repr\'esentation de $\G_F$, de dimension finie, et $V^+$ est un r\'eseau de $V$ stable
par $\G_F$, on pose $V_k^+=(\O_L/{\goth m}_L^k)\otimes V^+$ et:
\begin{align*}
\Pi_M(V)&:= {\rm Hom}_{\G_F\times\check G}
\big(V\otimes {\rm JL}(M),H^1_{\eet}({\cal M}^\varpi_{n,\C_p},L(1))\big)^\dual\\
\Pi_M(V^+)&:={\rm Hom}_{\G_F\times\check G}
\big(V^+\otimes {\rm JL}(M)^+,H^1_{\eet}({\cal M}^\varpi_{n,\C_p},\O_L(1))\big)^\dual\\
\Pi_M(V^+_k)&:={\rm Hom}_{\G_F\times\check G}
\big(V^+\otimes {\rm JL}(M)^+,H^1_{\eet}({\cal M}^\varpi_{n,\C_p},(\O_L/{\goth m}_L^k)(1))\big)^\dual
\end{align*}
($L$-dual pour le premier, $\O_L$-dual pour le second et $(\O_L/{\goth m}_L^k)$-dual
pour le dernier).
\begin{theo}\label{main10}
Si $\Pi_M(V)\neq 0$, alors pour tout $k\geq 1$, $\Pi_M(V^+_k)$ 
est une repr\'esentation lisse de $G$, non admissible.
\end{theo}
\begin{proof}
$\Pi_M(V)$ est un $G$-banach unitaire (lemme~\ref{bana1})
 et, s'il existe $k\geq 1$ tel que
$\Pi_M(V^+_k)$ est admissible, 
alors $\Pi_M(V)$ poss\`ede un r\'eseau dont la r\'eduction modulo~${\goth m}_L$
est un objet de ${\cal C}$ (lemme~\ref{bana2}). La prop.\,\ref{Ban-ord} implique alors que
${\rm Hom}_G(\widehat{\rm LL}(M),\Pi_M(V))=0$.

Par ailleurs, en passant aux vecteurs $G$-born\'es dans la premi\`ere ligne
du diagramme de~\cite[th.\,0.8]{CDN1} (apr\`es avoir appliqu\'e ${\rm Hom}_{\G_F}(V,-)$)
on obtient une injection de $\Pi_M(V)^\dual$ 
dans $H(V)\otimes\widehat{\rm LL}(M)^\dual$, o\`u
$H(V):={\rm Hom}_{\G_F}(V,X_{\rm st}(M))$ est un $L$-module de rang fini, avec action triviale de $G$. 
Par dualit\'e, cela fournit une fl\`eche de
$H(V)\otimes\widehat{\rm LL}(M)$ dans $\Pi_M(V)$, d'image dense. Comme cette image est nulle d'apr\`es
ce qui pr\'ec\`ede,
on a $\Pi_M(V)=0$,
ce qui prouve, par l'absurde, le th.~\ref{main10}.
\end{proof}

\begin{rema}\label{copr8}
(i) Si $F=\Q_p$, alors $\Pi_M(V)$ et $\Pi_M(V^+_k)$, pour tout $k$, sont admissibles.

(ii) 
Pour tout $M$ supercuspidal tel que $\varpi$ agit trivialement sur ${\rm JL}(M)$ on peut trouver $V$ telle que $\Pi_M(V)\neq 0$.
On prend $\Gamma\subset G'$, cocompact et on d\'efinit la courbe analytique $X_n:={\cal M}_{n,F}^\varpi/\Gamma$.
Alors $X_n$ est l'analytifi\'ee d'une courbe alg\'ebrique d\'efinie sur $F$. On suppose $\Gamma$ assez petit
pour que $X_0$ soit de genre~$\geq 2$. Alors, d'apr\`es Nakajima~\cite[th.\,4]{Naka}, toute repr\'esentation
de ${\rm Gal}(X_n/X_0)=\check G_0/\check G_n$ a une multiplicit\'e non nulle dans $\Omega^1(X_n)$,
et donc aussi dans $H^1_{\rm dR}(X_n)$ et dans $H^1_{\eet}(X_{n,\C_p},L(1))$ par le th\'eor\`eme de comparaison
\'etale--de Rham. Autrement dit,
$V_M:=H^1_{\eet}(X_{n,\C_p},L(1))[M]$ est non nul, pour tout $M$ de niveau~$\leq n$.
On peut alors prendre pour $V$ n'importe quel sous-$L$-module de $V_M$ stable par $\G_F$ (par exemple $V_M$).

(iii) Il ne semble pas possible de d\'eduire du th\'eor\`eme que 
$\Pi_M(V)$ est
non admissible.  La question de son admissibilit\'e reste donc en suspens.

(iv) On peut d\'efinir un $L[G]$-module $$\Pi_M^{\rm an}(V):={\rm Hom}_{\G_F\times\check G}
\big(V\otimes {\rm JL}(M),H^1_{\proet}({\cal M}^\varpi_{n,\C_p},L(1))\big)^\dual$$
En combinant les r\'esultats de \cite{CDN1} avec ceux de Patel-Schmidt-Strauch~\cite{PSS} 
et Ardakov-Wadsley \cite{AWD}, on peut montrer que, si $n=1$, alors $\Pi_M^{\rm an}(V)$ 
est une repr\'esentation localement analytique coadmissible, de longueur finie de $G$. 
Il est probable que cela reste vrai pour tout $n$. 
\end{rema}

\begin{lemm}\label{bana1}
$\Pi_M(V)$ est une repr\'esentation unitaire de $G$ dont $\Pi_M(V)^+$ est un r\'eseau stable par $G$,
et $(\O_L/{\goth m}_L^k)\otimes\Pi_M(V^+)$ est un quotient de $\Pi_M({V}^+_k)$.
\end{lemm}
  \begin{proof}
Soient $X$, $X^+$ et $X_k^+$ les duaux respectifs de $\Pi_M(V)$, $\Pi_M(V^+)$ et $\Pi_M(V^+_k)$.
Alors $X^+$ est sans $p$-torsion (puisque $H^1_{\eet}({\cal M}_{n,\C_p}, \mathcal{O}_L(1))$ l'est).
On va montrer que $X^+$ est profini, ce qui prouve que $\Pi_M(V^+)$ est un r\'eseau dans un banach,
et que $\Pi_M(V)$ est un banach.
On a 
   $$X^+=
{\varprojlim}_{k} 
   {\rm Hom}_{\G_F}\big(V^+\otimes{\rm JL}(M), H^1_{\eet}({\cal M}_{n,\C_p}^{\varpi}, \mathcal{O}_L/{\goth m}_L^k(1))\big).$$
  Il suffit de voir que chaque 
$ {\rm Hom}_{\G_F}\big(V^+\otimes{\rm JL}(M), H^1_{\eet}({\cal M}_{n,\C_p}^{\varpi}, \mathcal{O}_L/{\goth m}_L^k(1))\big)$ 
est profini, et par l'argument usuel de descente (via la suite spectrale de Hochschild-Serre) 
il suffit de voir que pour toute extension finie 
  $K$ de $F$ le groupe $H^1_{\eet}({\cal M}_{n,K}^{\varpi}, \mathcal{O}_L/{\goth m}_L^k(1))$
 est profini pour tout $k\geq 1$, or cela est garanti par le th.\,\ref{smooth}.

La fl\`eche naturelle $(\O_L/{\goth m}_L^k)\otimes X^+\rightarrow X^+_k$ est injective.
Comme $(\O_L/{\goth m}_L^k)\otimes X^+$ et $X^+_k$ sont compacts, cette injection 
est un hom\'eomorphisme de $(\O_L/{\goth m}_L^k)\otimes X^+$ sur son image qui est ferm\'ee dans $X^+_k$.
Par dualit\'e, cela fournit une surjection de $\Pi_M({V}_k^+)$ sur $(\O_L/{\goth m}_L^k)\otimes\Pi_M(V^+)$.
\end{proof}

\begin{lemm}\label{bana2}
S'il existe $k\geq 1$ tel que $\Pi_M(V^+_k)$ soit admissible, alors
$k_L\otimes \Pi_M(V^+)$ est un objet de ${\cal C}$.
\end{lemm}
\begin{proof}
Il r\'esulte du th.~\ref{smooth} que $\Pi_M(V^+_k)$ est un objet de ${\cal C}$. Il en est de m\^eme
de $(\O_L/{\goth m}_L^k)\otimes\Pi_M(V^+)$ (et donc aussi de $k_L\otimes \Pi_M(V^+)$)
puisque c'est un quotient de $\Pi_M(V^+_k)$.
\end{proof}
\begin{prop}\label{Ban-ord}
Soient $\pi$ une $L$-repr\'esentation lisse supercuspidale de $G$ et $\widehat\pi$ son
compl\'et\'e universel, et soit 
$\Pi$ une 
$L$-repr\'esentation de Banach unitaire de $G$ qui poss\`ede un r\'eseau $G$-invariant dont la r\'eduction modulo ${\goth m}_L$ est dans $\mathcal{C}$. Alors 
${\rm Hom}_G(\widehat{\pi}, \Pi)=0$.
\end{prop}
          
\begin{proof}
D'apr\`es le cor.\,\ref{finito2},
 il existe une $L$-repr\'esentation unitaire $W$ de~$A$, de dimension finie, et un morphisme 
$G$-\'equivariant $\iota=\iota_{\Pi}: {\rm Ind}_{\bar{B}}^G(W)\to 
\Pi$ dont le noyau et le conoyau sont de dimension finie sur $L$. Supposons qu'il existe une application 
$G$-\'equivariante non nulle $\alpha: \hat{\pi}\to \Pi$. La compos\'ee 
$\hat{\pi}\to \Pi\to {\rm coker}(\iota)$  est forc\'ement nulle car $\pi$ est dense dans $\hat{\pi}$ et n'a pas de quotient de dimension finie. Donc on peut supposer que
$\Pi\simeq {\rm Ind}_{\bar{B}}^G(W)/W'$ pour une repr\'esentation de dimension finie $W'$. 
Quitte \`a remplacer $L$ par une extension finie, on peut supposer que $W$ et $W'$ 
sont des extensions successives de caract\`eres et donc que ${\rm Ind}_{\bar{B}}^G(W)/W'$ 
est une extension successive finie de repr\'esentations 
de la forme ${\rm Ind}_{\bar{B}}^G (\chi_1\otimes\chi_2)$, 
${\rm St}\otimes \chi\circ \det$ ou $\chi\circ \det$ 
avec $\chi, \chi_1,\chi_2: F^{\dual}\to L^{\dual}$ des caract\`eres continus.
Mais $\hat{\pi}$ n'a pas de morphisme $G$-\'equivariant non nul vers une de ces repr\'esentations: 
il suffit en effet de voir que $\pi$ s'envoie sur $0$, 
ce qui est clair car $\pi$ est cuspidale
et l'espace des vecteurs lisses d'une de ces repr\'esentations est de la s\'erie principale.
Donc tout morphisme $G$-\'equivariant de $\hat{\pi}$ vers une extension successive finie 
de repr\'esentations de la forme ci-dessus est nul, ce qui permet de conclure.
                        \end{proof}

 \section{Une suite spectrale pour le foncteur de Scholze}
 
  Le but de ce chapitre est de d\'emontrer l'existence d'une suite spectrale reliant les foncteurs de Scholze \cite{SLT} 
  et la cohomologie des espaces de Drinfeld. L'\'etude de cette suite spectrale dans le cas particulier du groupe $\gl_2(\qp)$ est un ingr\'edient important de la preuve du th\'eor\`eme de finitude
(th.\,\ref{Intro1}), mais cette suite spectrale a sans doute un int\'er\^et en soi, et donc nous l'\'etablirons
  dans la g\'en\'eralit\'e de \cite{SLT}. 

\subsection{Rappels concernant le foncteur de Scholze}

   Soit $F$ une extension finie de $\qp$ et soit $\breve{F}\subset{\C_p}$ 
le compl\'et\'e de l'extension maximale non ramifi\'ee de $F$. Soit 
 $\mathcal{G}_F={\rm Gal}(\bar{F}/F)$.
   
On fixe un entier $n\geq 1$ et une alg\`ebre centrale \`a division $D_p$ sur $F$, d'invariant $1/n$. On note 
$(\mathcal{M}_{\rm LT, K})_{K}$ la tour de Lubin-Tate, une famille d'espaces rigides analytiques sur $\breve{F}$, index\'ee par les sous-groupes
ouverts compacts $K$ de $\gl_n(F)$. Le groupe $D_p^{\dual}$ agit sur chaque \'etage de la tour et le groupe $\gl_n(F)$ agit sur la tour, un \'el\'ement 
$g$ de $\gl_n(F)$ transformant $\mathcal{M}_{\rm LT, K}$ en $\mathcal{M}_{\rm LT, g^{-1}Kg}$. On dispose d'applications de p\'eriodes 
$$\pi_{\rm GH}: \mathcal{M}_{\rm LT, K}\to \mathbb{P}^{n-1}_{\breve{F}},$$
qui sont \'etales surjectives (de fibres $\gl_n(F)/K$), $D_p^{\dual}$-\'equivariantes ($D_p^{\dual}$ agissant sur 
$\mathbb{P}^{n-1}_{\breve{F}}$ via l'identification de cet espace avec la vari\'et\'e de Brauer-Severi associ\'ee \`a $D_p$). 
De plus, tous ces espaces sont munis de donn\'ees de descente \`a la Weil et $\pi_{\rm GH}$ respectent ces donn\'ees. 

    La tour $(\mathcal{M}_{\rm LT, K})_{K}$ poss\`ede une limite, qui est un espace perfecto\"{\i}de $\mathcal{M}_{\rm LT, \infty,\breve{F}}$ sur 
    $\breve{F}$ (cf. \cite{SW}) et les applications des p\'eriodes $\pi_{\rm GH}$ ci-dessus induisent un morphisme 
    $$\pi_{\rm GH}: \mathcal{M}_{\rm LT, \infty,\breve{F}}\to \mathbb{P}^{n-1}_{\breve{F}},$$
    faisant de $\mathcal{M}_{\rm LT, \infty,\breve{F}}$ un $\gl_n(F)$-torseur pro\'etale. Si 
    $\pi$ est un $\gl_n(F)$-module muni de la topologie discr\`ete, on note $\underline{\pi}$ le faisceau constant
   sur le site \'etale de $\mathcal{M}_{\rm LT, \infty,\breve{F}}$ et 
   $$\mathcal{F}_{\pi}:=(\pi_{\rm GH, *}(\underline{\pi}))^{\gl_n(F)},$$
   un faisceau \'etale $D_p^{\dual}$ et Weil-\'equivariant sur $\mathbb{P}^{n-1}_{\breve{F}}$. Concr\`etement, 
   pour tout morphisme \'etale $U\to \mathbb{P}^{n-1}_{\breve{F}}$ on a $\mathcal{F}_{\pi}(U)=
   \mathcal{C}^0_{{\rm GL}_n(F)}(|U_{\infty}|, \pi)$, 
   o\`u $U_{\infty}=U\times_{\mathbb{P}^{n-1}_{\breve{F}}} \mathcal{M}_{\rm LT, \infty,\breve{F}}$ est le pullback de $U$ en niveau infini.
   Si $\pi$ est lisse (i.e. le stabilisateur de tout vecteur de $\pi$ est ouvert), alors les fibres g\'eom\'etriques de 
   $\mathcal{F}_{\pi}$ s'identifient \`a $\pi$ (cf. la preuve de la prop 3.1 de \cite{SLT}). Cette identification d\'epend du choix d'un rel\`evement du point g\'eom\'etrique de $\mathbb{P}^{n-1}_{\breve{F}}$ \`a $\mathcal{M}_{\rm LT, \infty,\breve{F}}$, mais une fois un tel choix fait elle est fonctorielle par rapport \`a $\pi$. On posera 
   $$S^i(\pi):=H^i_{\eet}(\mathbb{P}^{n-1}_{{\C_p}}, \mathcal{F}_{\pi}).$$

   Soit $(A, \mathfrak{m})$ un anneau local, noeth\'erien, complet, de corps r\'esiduel une extension finie de $\mathbf{F}_p$ et soit 
   $H$ un groupe de Lie $p$-adique. On dit qu'un $A[H]$-module $\pi$ est {\it admissible} si 
   $$\pi=\bigcup_{K, n} \pi^{K}[\frak{m}^n],$$
   la r\'eunion \'etant prise sur les sous-groupes ouverts compacts $K$ de $H$ et sur les entiers $n\geq 1$, et si 
 les $\pi^K[\mathfrak{m}^n]$ sont des $A/\mathfrak{m}^n$-modules de type fini (et donc finis tous court). 

Le r\'esultat suivant r\'esume les propri\'et\'es locales 
des foncteurs $\pi\mapsto S^i(\pi)$, \'etablies dans \cite{SLT}. 

\begin{theo}\label{Good}{\rm (Scholze)}
 Soit $(A,\mathfrak{m})$ comme ci-dessus et soit $\pi$ un $A[\gl_n(F)]$-module admissible. 
 
 \begin{enumerate}[label={\rm(\arabic*)}]
 
 \item $S^i(\pi)$ est un 
 $A[D_p^{\dual}]$-module admissible et le morphisme naturel 
 $$S^i(\pi)\otimes_{\zp} \mathcal{O}_{\C_p}\to H^i_{\eet}(\mathbb{P}^{n-1}_{{\C_p}}, \mathcal{F}_{\pi}\otimes_{\zp} \mathcal{O}^+)$$
 est un presqu'isomorphisme. 
 
 \item On a $S^i(\pi)=0$ pour $i>2(n-1)$.
 
 \item L'action\footnote{Induite par l'action de $I_F$ sur ${\C_p}$ et par la donn\'ee de descente.} du groupe de Weil $W_F$ sur $S^i(\pi)$ s'\'etend par continuit\'e en une action de $\mathcal{G}_F$.
 
 \item On a $S^0(\pi^{ {\rm SL}_n(F)})=S^0(\pi)$, l'action de $\gl_n(F)$ s'y factorise par $\det: \gl_n(F)\to F^{\dual}$ et celle de 
 $W_F\times D_p^{\dual}$ par le morphisme $W_F\times D^{\dual}\to F^{\dual}$, inverse du produit de l'application d'Artin (envoyant des Frobenius g\'eom\'etriques sur des uniformisantes) et de la norme r\'eduite.
 \end{enumerate}
\end{theo}

\begin{proof} Les deux premiers points (et les plus d\'elicats) sont 
le th.\,4.4 de \cite{SLT}, le troisi\`eme et le quatri\`eme 
sont les prop.\,4.6 et 4.7 de loc. cit.
\end{proof}

\subsection{Foncteur de Scholze et cohomologie de la tour de Drinfeld}

 Le but de ce paragraphe est de d\'emontrer le th\'eor\`eme de comparaison \ref{compare} ci-dessous. On reprend les notations du paragraphe ci-dessus, en particulier 
 soit $\pi$ un $A[\gl_n(F)]$-module lisse admissible. Supposons aussi qu'il existe $k\geq 1$ tel que $p^k\pi=0$. 
 Alors $\pi^{\vee}$ (dual de Pontryagin de $\pi$) est un $\zp$-module profini, limite inverse des $\zp$-modules de longueur finie $(\pi^{K}[\mathfrak{m}^i])^{\vee}$ (pour $i\geq 1$ et $K$ des sous-groupes ouverts compacts de $\gl_n(F)$). 
 
    L'espace $\mathcal{M}_{\infty}:=\mathcal{M}_{\rm LT, \infty,{\C_p}}$ poss\`ede un recouvrement affino\"{\i}de perfecto\"{\i}de Stein $\{U_i\}_{i\geq 1}$ et on a des suites exactes  $$0\to \R^1\lim_{i} H^{q-1}_{\eet}(U_i, \Z/p^k)\to H^q_{\eet}(\mathcal{M}_{\infty}, \Z/p^k)\to \lim_{i} H^q_{\eet}(U_i, \Z/p^k)\to 0.$$
 On verra ci-dessous que $\R^1\lim_{i} H^{q-1}_{\eet}(U_i, \Z/p^k)=0$, donc $H^q_{\eet}(\mathcal{M}_{\infty}, \Z/p^k)$ est aussi un $\zp$-module prodiscret. Si $U=\lim_i U_i$ et $V=\lim_j V_j$ sont des $\zp$-modules prodiscrets, alors 
 $${\rm Hom}_{\zp}^{\rm cont}(U, V)=\lim_{j}(\colim_{i} {\rm Hom}_{\zp}(U_i, V_j))$$ 
 en est aussi un, en munissant les ${\rm Hom}_{\zp}(U_i, V_j)$ de la topologie discr\`ete.
 
 Soit $G=\gl_n(F)$. La cohomologie (continue) de $G$ \`a coefficients dans un module prodiscret $M$ est par d\'efinition la cohomologie du complexe des cocha\^ines continues sur $G$ \`a valeurs dans $M$.
 
\begin{theo} \label{compare} Soit $\pi$ un $A[G]$-module lisse admissible, tu\'e par $p^k$. Il existe une suite spectrale de $D_p^{\dual}\times W_F$-modules 
 \begin{equation}
 \label{cwir}
  E_2^{i,j}=H^i(G,H^j_{\eet}(\mathcal{M}_{\infty}, \underline{\pi}))\Longrightarrow 
 S^{i+j}(\pi).
 \end{equation}
 De plus $H^j_{\eet}(\mathcal{M}_{\infty}, \underline{\pi})=0$ pour $j\geq 2$.
 
 \end{theo}
 \begin{proof}  Consid\'erons le morphisme des p\'eriodes $\pi_{\rm GH}: \mathcal{M}_{\infty}\to \mathbb{P}^{n-1}_{\C_p}$. C'est un $G$-torseur pro\'etale.  
La suite spectrale (\ref{cwir}) \`a \'etablir est une suite de type Cartan-Leray pour 
  le recouvrement pro\'etale $\mathcal{U}=\{\mathcal{M}_{\infty}\to \mathbb{P}^{n-1}_{\C_p}\}$. 
Pour la construire, consid\'erons la suite spectrale de \v{C}ech pour la projection
$\nu:  \mathbb{P}^{n-1}_{{\C_p},\proeet}\to \mathbb{P}^{n-1}_{{\C_p},\eet}$,
 \begin{equation}
 \label{cwir1}
 E_2^{i,j}=\check{H}^i_{\proeet}(\mathcal{U}, \underline{H}^j(\wt{\mathcal{F}}_{\pi}))\Longrightarrow H^{i+j}_{\proeet}(\mathbb{P}^{n-1}_{\C_p}, \wt{\mathcal{F}}_{\pi})=S^{i+j}(\pi),
 \end{equation}
 o\`u $\wt{\mathcal{F}}_{\pi}:=\nu^\dual{\mathcal{F}}_{\pi}$.
 La derni\`ere \'egalit\'e  dans (\ref{cwir1}) est une cons\'equence de \cite[prop.\,14.8]{Sz-diamonds}. 
 
Il suffit de prouver que l'on a un isomorphisme naturel 
  $$\check{H}^i(\mathcal{U}, \underline{H}^j(\wt{\mathcal{F}}_{\pi}))\simeq H^i(G,  H^j_{\eet}(\mathcal{M}_{\infty}, \underline{\pi})).$$
Pour ce faire, rappelons que $E_2^{i,j}$ est la cohomologie en degr\'e $i$ du complexe 
  $$C^j: =(H^i_{\proeet}(\mathcal{M}_{\infty}, \wt{\mathcal{F}}_{\pi})\to H^i_{\proeet}(\mathcal{M}_{\infty}\times_{\mathbb{P}^{n-1}_{\C_p}} \mathcal{M}_{\infty},  \wt{\mathcal{F}}_{\pi})\to...)$$
  Puisque $\pi_{\rm GH}$ est un $G$-torseur pro\'etale, on a un isomorphisme de diamants\footnote{Si $T$
est un espace topologique, on note $\underline{T}$ le faisceau pro\'etale $\underline{T}(X)=\scc^0(|X|,T)$.} 
  $$\mathcal{M}_{\infty}\times_{\mathbb{P}^{n-1}_{\C_p}} \mathcal{M}_{\infty}\simeq \mathcal{M}_{\infty}\times \underline{G}.$$
Cela implique que (le membre de gauche a $k+1$ facteurs $\mathcal{M}_{\infty}$)
  $$
  \mathcal{M}_{\infty}\times_{\mathbb{P}^{n-1}_{\C_p}}\cdots \times_{\mathbb{P}^{n-1}_{\C_p}}
\mathcal{M}_{\infty} \simeq \mathcal{M}_{\infty}\times \underline{G}^k,\quad k\geq 1.
  $$
De plus, par cet isomorphisme, $\wt{\sff}_{\pi}\simeq {\rm pr}^\dual_{1}\wt{\underline{\pi}}$, 
o\`u ${\rm pr}_1: \mathcal{M}_{\infty}\times \underline{G}^k\to \mathcal{M}_{\infty}$ est la projection canonique et $\wt{\underline{\pi}}=\nu^\dual \underline{\pi}$.
Il s'ensuit que l'on a un
quasi-isomorphisme
  \begin{equation}
  \label{jeden}
  C^j\simeq (H^i_{\eet}(\mathcal{M}_{\infty}, \underline{\pi})\to H^i_{\eet}(\mathcal{M}_{\infty}\times \underline{G}, \underline{\pi})\to...)
  \end{equation}  
Il suffit maintenant de montrer que l'on a un quasi-isomorphisme
    \begin{equation}
    \label{dwa}
    C^j\simeq (\mathcal{C}^0(G^0, H^i_{\eet}(\mathcal{M}_{\infty}, \underline{\pi}))\to \mathcal{C}^0(G^1, H^i_{\eet}(\mathcal{M}_{\infty}, \underline{\pi}))\to \cdots)
    \end{equation}
 Commen\c{c}ons par montrer que, pour $t\geq 1$, on a un isomorphisme 
  \begin{align}
  \label{niedziela11}
  H^i_{\eet}(\mathcal{M}_{\infty}\times \underline{G}^t, \underline{\pi}) & \simeq \mathcal{C}^0(G^t, H^i_{\eet}(\mathcal{M}_{\infty}, \underline{\pi})).
  \end{align}
Soit $\{U_s\}_{s\in\N}$ un recouvrement Stein affino\"{\i}de perfecto\"{\i}de de $\mathcal{M}_{\infty}$.
 Soit $\{S^t_q\}_{q\in \N}$ une suite croissante de sous-ensembles compacts ouverts de $G^t$, dont la r\'eunion est $G^t$.
 On obtient alors un recouvrement 
$$U_s\times \underline{S^t_s}={\rm Spa}( \mathcal{C}^0(S^t_s, \mathcal{O}(U_s)), \mathcal{C}^0(S^t_s, \mathcal{O}^+(U_s)))$$
    de $\mathcal{M}_{\infty}\times \underline{G}^t$ par une suite croissante d'ouverts affino\"{\i}des perfecto\"{\i}des.

    Le lemme 23.6 de \cite{Sz-diamonds} montre que 
    $$H^i_{\eet}(U_s\times \underline{S^t_s}, \underline{\pi})=H^i_{\eet}(U_s, \mathcal{C}^0(S^t_s, \Z/p^k)\otimes_{\Z/p^k} \underline{\pi}).$$
    Puisque $U_s$ est qcqs et $\pi$ est une limite inductive de $\Z/p^k$-modules finis, 
le dernier groupe de cohomologie n'est rien d'autre que 
$\mathcal{C}^0(S^t_s, \Z/p^k)\otimes_{\Z/p^k} \pi\otimes_{\Z/p^k} 
    H^i_{\eet}(U_s, \Z/p^k)$. Le lemme \ref{delicat} ci-dessous montre que pour tout $i$ le syst\`eme projectif 
    $\{H^i_{\eet}(U_s, \Z/p^k)\}_{s}$ est Mittag-Leffler. 
Puisque les $S^t_s$ forment une suite croissante d'ensembles profinis, le syst\`eme projectif 
$\{\mathcal{C}^0(S^t_s, \Z/p^k)\}_{q}$ est Mittag-Leffler (en fait
 $\mathcal{C}^0(S^t_{s+1}, \Z/p^k)\to \mathcal{C}^0(S^t_s, \Z/p^k)$ est surjective). 
On en d\'eduit que le  syst\`eme projectif 
$\{H^i_{\eet}(U_s\times \underline{S^t_s}, \underline{\pi})\}_{s}$ est
aussi Mittag-Leffler, et donc que 
    \begin{equation} 
    \label{niedziela1}H^i_{\eet}(\mathcal{M}_{\infty}\times \underline{G}^t, \underline{\pi})\simeq \lim_{s} 
\big(\mathcal{C}^0(S^t_s, \Z/p^k)\otimes_{\Z/p^k} \pi\otimes_{\Z/p^k} 
    H^i_{\eet}(U_s, \Z/p^k)\big).
    \end{equation}

Par ailleurs, on a
    \begin{align}
    \label{niedziela12}
    \mathcal{C}^0(G^t, H^i_{\eet}(\mathcal{M}_{\infty}, \underline{\pi})) & \stackrel{\sim}{\to} \lim_{s} \mathcal{C}^0(S^t_s, H^i_{\eet}(\mathcal{M}_{\infty}, \underline{\pi}))\\
     & \stackrel{\sim}{\to}\lim_{q,s} \big(\mathcal{C}^0(S^t_s, \Z/p^k)\otimes_{\Z/p^k} \pi\otimes_{\Z/p^k} H^i_{\eet}(U_{s}, \Z/p^k)\big)\notag
        \end{align}
Combin\'e avec l'isomorphisme (\ref{niedziela1}), cela fournit l'isomorphisme (\ref{niedziela11}) que
l'on voulait.
    
Il reste \`a v\'erifier que les isomorphismes  (\ref{niedziela11}) sont compatibles
avec les diff\'erentielles dans les complexes apparaissant dans (\ref{jeden}) et (\ref{dwa}). 
Pour cela, prenons un recouvrement $\{S_q\}_{q}$ de $G$ et recouvrons $G^t$ par les $S^t_q$. 
Cela fournit des complexes \`a partir des objets intervenant dans (\ref{niedziela1}) et (\ref{niedziela12})
et les fl\`eches qui s'en d\'eduisent sont des morphismes de complexes: 
les diff\'erentielles sont des sommes altern\'ees faisant intervenir des projections (qui ne posent
pas de probl\`emes gr\^ace \`a notre choix de recouvrement de $G^t$), des multiplications 
$S_a\times S_b\mapsto S_{c(a,b)}$ induites par la multiplication $\mu: G\times G\mapsto G$, et les actions
      $\underline{S_a}\times U_b\mapsto U_{c(a,b)}$ induites par 
l'action $\underline{G}\times \sm_{\infty}\mapsto \sm_{\infty}$ 
(l'existence de $c(a,b)$ dans les deux cas vient de ce que les $S_a$ sont compacts et les $U_b$ quasi-compacts). 
 \end{proof}

  \begin{lemm}\label{delicat}
  Soit $\{U_s\}_{s\in\N}$ un recouvrement Stein affino\"{\i}de perfecto\"{\i}de de $\mathcal{M}_{\infty}$. Le syst\`eme projectif 
  $\{H^i_{\eet}(U_s, \Z/p^k)\}_s$ est Mittag-Leffler pour tout $i$.
\end{lemm}

\begin{proof}  Par d\'evissage on se ram\`ene au cas $k=1$. Supposons d'abord que $i\geq 1$. Alors 
$$H^i_{\eet}(U_s,\mathbf{F}_p)\simeq H^i_{\eet}(U_s^{\flat},\mathbf{F}_p)$$
et $U_s^{\flat}$ est un affino\"{\i}de perfecto\"{\i}de. La suite d'Artin-Schreier 
$$0\to\mathbf{F}_p\to \mathcal{O}_{U_s^{\flat}}\to \mathcal{O}_{U_s^{\flat}}\to 0$$
et l'annulation de $H^i_{\eet}(U_s^{\flat}, \mathcal{O}_{U_j^{\flat}})$ pour $i\geq 1$ montrent que 
$H^i_{\eet}(U_s,\mathbf{F}_p)=0$ pour $i\geq 2$ et 
$$H^1_{\eet}(U_s,\mathbf{F}_p)\simeq \mathcal{O}(U_s^{\flat})/(1-\varphi).$$

Il suffit donc de prouver que la fl\`eche naturelle $\mathcal{O}(U_{s+1}^{\flat})\to \mathcal{O}(U_s^{\flat})$ 
est d'image dense. Par \cite[th. 2.4.3]{Kedl}, cela suit du m\^eme \'enonc\'e pour les d\'ebascul\'es, 
qui est standard. \label{G: faut ajouter que tout est compatible avec la topologie}
      
Il reste \`a traiter le cas $i=0$. On peut supposer que les $U_s$ sont les images inverses du 
recouvrement standard $\{V_s\}_s$ de l'espace en niveau $0$ (qui est un disque ouvert). Alors $U_s\to V_s$ est un $K$-torseur pro\'etale, donc 
$K$ agit transitivement sur $\pi_0(U_s)$. L'inclusion $U_s\to U_{s+1}$ 
correspond sur les $\pi_0$ \`a une fl\`eche $K/\Gamma_s\to K/\Gamma_{s+1}$, avec $\Gamma_s\subset \Gamma_{s+1}$ des sous-groupes ferm\'es de $K$,
et, sur les $H^0$, \`a l'application naturelle
 $\Hom_{\rm cont}(K/\Gamma_{s+1},\mathbf{F}_p)\to \Hom_{\rm cont}(K/\Gamma_s,\mathbf{F}_p)$.  
    
Il suffit donc de prouver que le prosyst\`eme $\{ \Hom_{\rm cont}(K/\Gamma_s,\mathbf{F}_p)\}_s$ est Mittag-Leffler.
Or d'apr\`es~\cite[prop 2.4]{Glo}, $K$ contient un sous-groupe ouvert $K'$ tel que toute suite
croissante de sous-groupes ferm\'es soit stationnaire. Mais alors $K$ a la m\^eme propri\'et\'e puisque
$K'$ est d'indice fini dans $K$. Il s'ensuit que notre syst\`eme est en fait constant pour $s$ assez grand,
et donc, a fortiori, Mittag-Leffler.
\end{proof}

 \begin{coro}\label{SS}
 Soit $\pi$ un $A[G]$-module lisse admissible, tu\'e par $p^k$. Il existe une suite spectrale de $D_p^{\dual}\times W_{F}$-modules 
 \begin{equation}\label{SS11}
  E_2^{i,j}=H^i(G, {\rm Hom}_{\zp}^{\rm cont}(\pi^{\vee}, H^j_{\eet}(\mathcal{M}_{\infty}, \Z/p^k)))\Longrightarrow 
 S^{i+j}(\pi).
 \end{equation}
 De plus $H^j_{\eet}(\mathcal{M}_{\infty}, \Z/p^k)=0$ pour $j\geq 2$.
 \end{coro}
 \begin{proof}
Compte-tenu du th.\,\ref{compare}, il suffit de v\'erifier que
 $$
 H^j_{\eet}(\mathcal{M}_{\infty}, \underline{\pi})\simeq {\rm Hom}_{\zp}^{\rm cont}(\pi^{\vee}, H^j_{\eet}(\mathcal{M}_{\infty}, \Z/p^k)).
 $$
En raisonnant comme dans la preuve dudit th\'eor\`eme et en reprenant ses notations, on obtient:
   \begin{align*}
   H^j_{\eet}(\mathcal{M}_{\infty}, \underline{\pi}) & \stackrel{\sim}{\to} \lim_{s} (\pi\otimes H^j_{\eet}(U_s, \Z/p^k))\simeq 
     \lim_{s} {\rm Hom}_{\zp}^{\rm cont}(\pi^{\vee}, H^j_{\eet}(U_s, \Z/p^k))\\
      & \simeq {\rm Hom}_{\zp}^{\rm cont}(\pi^{\vee}, H^j_{\eet}(\mathcal{M}_{\infty}, \Z/p^k)).\qedhere
    \end{align*}
 \end{proof}
 
    Pour $G=\gl_2(\Q_p)$ on obtient le r\'esultat suivant: 
 
\begin{coro}\label{cafe52}Soit $\pi$ un $k_L[G]$-module lisse admissible, irr\'eductible. On dispose d'un morphisme naturel
$$S^1(\pi)
     \to {\rm Hom}_{k_L[G]}^{\rm cont}(\pi^{\vee}, H^1_{\eet}(\mathcal{M}_{\infty},
   k_L))$$
qui est un isomorphisme si $\pi$ n'appartient pas \`a un twist 
$\{\chi,{\rm St}\otimes\chi, I({\chi,\chi\epsilon})\}$ 
du bloc de la Steinberg, et dont les noyau et conoyau sont
de dimension finie sur $k_L$ si $\pi$ appartient \`a un twist du bloc de la Steinberg.
\end{coro}
 \begin{proof}La suite spectrale (\ref{SS11})
      fournit une suite exacte 
   \begin{align}
   \label{seq11}
   0  \to H^1(G, {\rm Hom}_{\zp}^{\rm cont}(\pi^{\vee},  & H^0_{\eet}(\mathcal{M}_{\infty}, 
   k_L))) \stackrel{f_{\pi}}{\to} S^1(\pi)
     \to {\rm Hom}_{k_L[G]}^{\rm cont}(\pi^{\vee}, H^1_{\eet}(\mathcal{M}_{\infty}, 
   k_L))\\
   &  \to H^2(G, {\rm Hom}_{\zp}^{\rm cont}(\pi^{\vee},   H^0_{\eet}(\mathcal{M}_{\infty}, 
   k_L))).\notag
   \end{align}
On a un isomorphisme 
  $${\rm Hom}_{k_L}^{\rm cont}(\pi^{\vee}, H^0_{\eet}(\mathcal{M}_{\infty}, 
   k_L)) \simeq {\rm Ind}_{{\rm SL}_2(\qp)}^{G} (\pi).$$
   Le lemme de Shapiro \cite[prop.\,IX.2.3]{BW} 
(notons que le d\'eterminant $G\to\Q_p^{\dual}$ a une section continue) montre alors que 
$$V_i(\pi):=H^i(G, {\rm Hom}_{k_L}^{\rm cont}(\pi^{\vee}, H^0_{\eet}(\mathcal{M}_{\infty}, 
   k_L)))\simeq H^i({\rm SL}_2(\qp), \pi).$$
L'\'enonc\'e du corollaire (i.e.~la finitude des $V_i(\pi)$ et leur nullit\'e si $\pi$ n'appartient pas
\`a un twist du bloc de la Steinberg)
est donc une cons\'equence de r\'esultats de Fust~\cite[th.\,1.2, lemma 6.3]{Fust}
(utiliser le lemme \ref{irred} pour se ramener
au cas absolument irr\'eductible et noter qu'une repr\'esentation admissible de 
$G$ avec un caract\`ere central est d\'ej\`a admissible en tant que repr\'esentation de 
${\rm SL}_2(\qp)$). 
\end{proof}

\section{Longueur finie}\label{lg1}
Dans ce chapitre, on suppose que $F=\Q_p$, avec\footnote{\label{inutile1}On suppose que $p>2$ pour pouvoir utiliser l'existence de globalisations avec de bonnes propri\'et\'es (le seul endroit o\`u cette hypoth\`ese est utilis\'ee est dans la prop.\,\ref{globaliseDPS2} ci-dessous).}
 $p>2$ et on prend $\varpi=p$.
Notre but est de prouver le r\'esultat de finitude suivant.
\begin{theo}\label{main11}
    Pour toute extension finie $K$ de $\qp$, 
le $G$-module $H^1_{\eet}( \mathcal{M}_{n,K}^{ p}, k_L)$ est le dual d'une repr\'esentation lisse de $G$, de longueur finie.
   \end{theo}
   
\begin{rema}
   On a bien s\^ur un \'enonc\'e analogue avec des coefficients $\O_L/{\goth m}_L^k$, 
qui se d\'eduit formellement du th\'eor\`eme ci-dessus et du th.\,\ref{smooth}. 
En combinant cela avec le \no\ref{hecke2}, on obtient (par voie assez d\'etourn\'ee\footnote{Le th.\,\ref{main10bis} semble cependant fournir des obstructions s\'erieuses \`a un argument purement g\'eom\'etrique.}) l'admissibilit\'e 
   des $G$-modules lisses $H^1_{\eet}( \mathcal{M}_{n,K}^{ p}, \O_L/\varpi_L^k)^{\vee}$. 
\end{rema}
      
\subsection{Compatibilit\'e local-global}\label{lg2}
  Le but de ce paragraphe est de d\'emontrer la prop.\,\ref{key} ci-dessous, 
qui joue un r\^ole cl\'e dans la preuve du th\'eor\`eme de longueur finie. 
Les ingr\'edients de la preuve sont les travaux de Scholze \cite{SLT}, 
  Ludwig \cite{Ludwig} et \paskunas\, \cite{Pask}.

\begin{prop}\label{key}
Soit $(\chi_1, \chi_2)$ un couple de caract\`eres
lisses $\qp^{\dual}\to \overline{\bf F}_p^{\dual}$, g\'en\'erique.
Si $K$ est une extension finie de $\qp$ telle que\footnote{
La repr\'esentation $I(\chi_1,\chi_2)$ est celle d\'efinie au \no\ref{BL11}.}
$S^1(I(\chi_1,\chi_2))^{\mathcal{G}_K}\ne 0$, alors 
$\chi_1$ ou $\chi_2$ se factorise par ${\rm Gal}(K/\qp)$.
\end{prop}

\begin{rema}
 Nous renvoyons le lecteur au chapitre 8 de \cite{HW} pour des r\'esultats plus fins concernant la structure de $S^1(I(\chi_1,\chi_2))$.
\end{rema}
      
  M\^eme si la prop.\,\ref{key} est purement locale, nous aurons besoin 
d'un certain nombre d'objets de nature globale pour la prouver. 
Nous supposerons\footnote{Cela ne pose pas de probl\`eme puisque l'on peut toujours remplacer $L$ par une extension finie dans ce qui suit.} que 
$[k_L:\mathbf{F}_p]\geq 3$.  

\subsubsection{Globalisation}\label{lg3}
Soit 
$\bar{r}: \mathcal{G}_{\qp}\to \gl_2(k_L)$ une repr\'esentation continue. Le r\'esultat suivant \cite[prop.\,8.1]{DPS2} est une cons\'equence directe de \cite[cor.\,A.3]{GK}:

\begin{prop}\label{globaliseDPS2}
  Il existe un corps totalement r\'eel $E$, de degr\'e pair sur $\mathbf{Q}$, dans lequel $p$ est totalement d\'ecompos\'e, 
ainsi qu'une repr\'esentation automorphe cuspidale r\'eguli\`ere, alg\'ebrique, de poids $0$ de 
  $\gl_2(\mathbf{A}_E)$, telle que:
  
  $\bullet$ la repr\'esentation galoisienne $\rho_{\pi}: \mathcal{G}_E\to \gl_2(\overline{\mathbf{Q}}_p)$ associ\'ee \`a $\pi$ est non ramifi\'ee en dehors de $p$;
  
  $\bullet$ on a ${\rm SL}_2(k_L)\subset {\rm Im}(\overline{\rho}_{\pi})\subset \gl_2(k_L)$ 
{\rm (en particulier $\overline{\rho}_{\pi}$ est absolument irr\'eductible)}
 et $ \overline{\rho}_{\pi,v}:=\overline{\rho}_{\pi}|_{\mathcal{G}_{E_v}}$ est isomorphe \`a $\bar{r}$ pour toute place $v$ de $E$ divisant $p$.
\end{prop}

  Fixons pour toute la suite $E$ et $\pi$ comme dans la proposition ci-dessus. Notons simplement 
   $\mathbf{A}$ l'anneau des ad\`eles de $E$, $\mathbf{A}_f$ les ad\`eles finis, et 
$$\bar{\rho}:=\bar{\rho}_{\pi}: \mathcal{G}_E\to \gl_2(k_L), 
\quad \rho=\rho_{\pi}: \mathcal{G}_E\to \gl_2(\overline{\mathbf{Q}}_p).$$ 
Quitte \`a remplacer $L$ par une extension finie et \`a conjuguer 
   $\rho$, on peut supposer que $\rho(\mathcal{G}_{E})\subset \gl_2(\mathcal{O}_L)$. 
   
   On regarde le caract\`ere unitaire 
$$\psi:=\chi_{\rm cyc}\cdot \det \rho: \mathcal{G}_{E}\to \mathcal{O}_L^{\dual}$$
  comme un caract\`ere continu de $\mathbf{A}_{f}^{\dual}/E^{\dual}$, par la th\'eorie du corps 
de classes (normalis\'ee de telle sorte que les uniformisantes s'envoient sur 
des frobenius g\'eom\'etriques). On note $\psi_v$ la restriction de $\psi$ \`a un sous-groupe de d\'ecomposition en $v$, et on consid\`ere $\psi_v$ comme un caract\`ere de $E_v^{\dual}$.

  Fixons enfin une place $\pp$ de $E$ divisant $p$ (on a donc $E_{\pp}=\qp$) ainsi qu'une place infinie $\infty_{0}$ de $E$, et notons
  $$\zeta=\psi_{\pp}: \qp^{\dual}\to \mathcal{O}_L^{\dual}.$$ 
      On consid\`ere aussi $
      \zeta$ comme un caract\`ere de $G$.

     \subsubsection{Groupes quaternioniques}
   
  Soit $D_0$ une alg\`ebre de quaternions sur $E$, compacte (modulo le centre) en toute place infinie de 
   $E$ et d\'eploy\'ee en toute place finie de~$E$. Soit $\mathcal{O}_{D_0}$ un ordre maximal de 
   $D_0$, et fixons des isomorphismes $(\mathcal{O}_{D_0})_v\simeq M_2(\mathcal{O}_{E_v})$ pour toute place finie $v$. 
   Soit $D$ l'alg\`ebre de quaternions sur $E$ obtenue \`a partir de $D_0$ en \'echangeant les invariants en 
   $\pp$ et $\infty_{0}$. On fixe un isomorphisme 
\begin{equation}\label{isom}
D_0\otimes_E \mathbf{A}^{\pp,\infty_0}\simeq D\otimes_E \mathbf{A}^{\pp,\infty_0}
\end{equation}

Notons $\check{\mathbb G}$ et ${\mathbb G}$ les groupes alg\'ebriques associ\'es
\`a $D^\dual $ et $D_0^\dual$, i.e. $\check{\mathbb G}(R)=(R\otimes_E D)^{\dual}$
et ${\mathbb G}(R)=(R\otimes_E D_0)^{\dual}$
pour toute $E$-alg\`ebre $R$.
En particulier 
$${\mathbb G}(E_{\pp})\simeq G=\gl_2(\qp)
\quad{\rm et}\quad 
\check{\mathbb G}(E_{\pp})\simeq \check G$$ o\`u $\check G$ est le groupe des unit\'es de 
l'alg\`ebre de quaternions non d\'eploy\'ee sur $\Q_p$. 
On identifie $\mathbf{A}_f^{\dual}$ aux centres de ${\mathbb G}(\mathbf{A}_f)$ 
et $\check{\mathbb G}(\mathbf{A}_f)$. 

 On peut d\'ecomposer ${\mathbb G}(\mathbf{A}_f)$ comme une r\'eunion disjointe finie de doubles classes 
 ${\mathbb G}(E)t_i U_{\rm max} \mathbf{A}_f^{\dual}$, le groupe $U_{\rm max}:=\prod_{v\nmid \infty} \gl_2(\mathcal{O}_{E_v})$ \'etant identifi\'e (par les choix faits ci-dessus) \`a un sous-groupe ouvert compact de 
 ${\mathbb G}(\mathbf{A}_f)$. Les
 $(U_{\rm max}  \mathbf{A}_f^{\dual}\cap t_i {\mathbb G}(E) t_i^{-1})/E^{\dual}$ 
sont des groupes finis, et on note 
 $N$ le produit des cardinaux de ces groupes finis. Le lemme 8.2 de \cite{DPS2} montre l'existence d'une place finie $w_1$ de $E$ telle que $\mathbf{N}(w_1)$ 
soit premier avec $2Np$ et non congru \`a $1$ modulo $p$, et telle que le quotient des valeurs propres de $\bar{\rho}({\rm Frob}_{w_1})$ ne soit pas 
 $1$ ou $\mathbf{N}(w_1)^{\pm 1}$. On fixe une telle place $w_1$ et on d\'efinit un sous-groupe ouvert compact $U$ de ${\mathbb G}(\mathbf{A}_f)$ par 
 $$U=\prod_{v\nmid \infty} U_v,$$
 avec $U_v=\gl_2(\mathcal{O}_{E_v})$ pour $v\ne w_1$ et 
 $$U_{w_1}=\{g\in \gl_2(\mathcal{O}_{E_{w_1}})|\,\, g\equiv \matrice{1}{*}{0}{1} \pmod{\varpi_{w_1}}\}.$$
    
    Le sous-groupe $U$ de ${\mathbb G}(\mathbf{A}_f)$ donne naissance \`a 
    un sous-groupe ouvert compact $U^{\pp}:=\prod_{v\nmid \pp\infty} U_v$ de 
  ${\mathbb G}(\mathbf{A}_{f}^{\pp})$, que l'on identifie aussi \`a un sous-groupe ouvert compact de 
  $\check{\mathbb G}(\mathbf{A}_f^{\pp})$ par l'isomorphisme 
${\mathbb G}(\mathbf{A}_f^{\pp})\simeq\check{\mathbb G}(\mathbf{A}_f^{\pp})$ 
induit par (\ref{isom}). Posons 

$$U_p:=(\mathcal{O}_{D_0}\otimes \zp)^{\dual}=\prod_{v\mid p}Ê\gl_2(\zp),\,\,U_p^{\pp}:=\prod_{v\mid p, v\ne \pp} U_v,\,\, U^p:=\prod_{v\nmid p\infty} U_v.$$

\subsubsection{Formes automorphes et op\'erateurs de Hecke}\label{whatthehecke}
  Soit $A$ un $\mathcal{O}_L$-module topologique. Le $\mathcal{O}_L$-module 
  $$S(U^p, A):=\mathcal{C}^0({\mathbb G}(E)\backslash {\mathbb G}(\mathbf{A}_{f})/U^p, A)$$
 est
une repr\'esentation de $\mathbf{A}_{f}^{\dual}\,{\mathbb G}(E\otimes_{\mathbf{Q}} \qp)$ 
(agissant par translation \`a droite). 
On note
$S_{\psi}(U^p, A)$ le sous-espace $\psi$-isotypique (pour l'action de $\mathbf{A}_{f}^{\dual}$) de 
$S(U^p,A)$. 
Si $\lambda$ est une representation continue de 
$U_p^{\pp}$ sur un $\mathcal{O}_L$-module libre de type fini, telle que 
$\mathbf{A}_{f}^{\dual}\cap U_p^{\pp}$ agisse \`a travers $\psi$, on 
dispose du 
$G$-module
$$S_{\psi,\lambda}(U^{\pp}, A):={\rm Hom}_{U_p^{\pp}}
(\lambda, S_{\psi}(U^p, A)).$$

   Soit $S$ un ensemble fini de places de $E$, contenant $w_1$, les places divisant $p$ et les places infinies de $E$. L'alg\`ebre 
     $$\mathbb{T}_S^{\rm univ}=\mathcal{O}[T_v, S_v]_{v\notin S}$$ 
 de polyn\^omes en les variables indiqu\'ees agit\footnote{$S_v$ (resp. $T_v$) agit \`a travers la double classe 
$U_v \matrice{\varpi_v}{0}{0}{\varpi_v}$ (resp. $U_v \matrice{\varpi_v}{0}{0}{1}U_v$).} 
  sur $S(U^p, A)$, et cette action commute avec celle de 
$\mathbf{A}_{f}^{\dual}\,{\mathbb G}(E\otimes_{\mathbf{Q}} \qp)$.
 Ainsi $S(U^p, A)$ devient un 
$\mathbb{T}_S^{\rm univ}[\mathbf{A}_{f}^{\dual}\,{\mathbb G}(E\otimes_{\mathbf{Q}} \qp)]$-module et $S_{\psi}(U^p, A)$ un $\mathbb{T}_S^{\rm univ}[{\mathbb G}(E\otimes_{\mathbf{Q}} \qp)]$-module.
   
   La repr\'esentation $\rho: \mathcal{G}_E\to \gl_2(\mathcal{O}_L)$ est non ramifi\'ee en dehors de 
   $S$ et induit un morphisme  $\mathbb{T}_S^{\rm univ}\to k_L$ envoyant $T_v$ (respectivement $S_v$) sur ${\rm tr}(\bar{\rho}({\rm Frob}_v))$ (respectivement $\mathbf{N}(v)^{-1}\det (\bar{\rho}({\rm Frob}_v))$). On note 
    $$  \mathfrak{m}=\ker(\mathbb{T}_S^{\rm univ}\to k_L)$$
le noyau de ce morphisme, c'est donc un 
    id\'eal maximal $\mathfrak{m}$ de $\mathbb{T}_S^{\rm univ}$, de corps r\'esiduel $k_L$, tel que $\psi({\rm Frob}_v)\equiv S_v\pmod {\mathfrak{m}}$ pour $v\notin S$.

\begin{prop}
{\rm a)} Le localis\'e $S_{\psi}(U^p, L/\mathcal{O}_L)_{\mathfrak{m}}$ est non nul, admissible et injectif dans 
${\rm Rep}^{\rm lisse, \psi_p}\, U_p$, o\`u $\psi_p$ est le caract\`ere du centre de 
$U_p=\prod_{v\mid p} \gl_2(\zp)$ induit par~$\psi$. 

{\rm b)} On peut choisir $\lambda$ de la forme $\lambda=\otimes_{v\mid p, v\ne \pp} \lambda_v$, de telle sorte que 
$\lambda_v$ ait $\psi_v$ pour caract\`ere central et $S_{\psi,\lambda}(U^{\pp}, L/\mathcal{O}_L)_{\mathfrak{m}}\ne 0$. 
\end{prop}

\begin{proof} Le premier point est le contenu de \cite[prop.\,8.4]{DPS2}
et le second est d\'emontr\'e dans le lemme 8.6 de loc.cit. 
\end{proof}

 On fixe par la suite une repr\'esentation $\lambda$ comme dans la proposition ci-dessus et on note pour simplifier
 $$S_{\psi,\lambda, \mathfrak{m}}:=S_{\psi, \lambda}(U^{\pp}, L/\mathcal{O}_L)_{\mathfrak{m}}.$$ 

\begin{prop}\label{Paskpuzz}  Soit $\mathcal{B}$ le bloc associ\'e \`a $\bar{r}^{\rm ss}$. Alors 
       $S_{\psi,\lambda, \mathfrak{m}}$ est un objet admissible de ${\rm Rep}_{\mathcal{B}}^{\zeta}\, G$, dont la restriction \`a $K={\rm GL}_2(\zp)$ est injective dans ${\rm Rep}^{\zeta}\, K$.

\end{prop}

\begin{proof}
  Toutes les r\'ef\'erences ci-dessous sont par rapport \`a l'article de \paskunas\, \cite{Pask}. Le premier point est \'etabli dans les lemmes 5.1, 5.3 et la prop.\,5.4, pour le second voir les lemmes 5.2 et 5.3. 
\end{proof}

\begin{rema} Si l'on munit 
    $$    S_{\psi, \lambda}(U^{\pp},\mathcal{O}_L)_{\mathfrak{m}}:=\varprojlim_{n} 
     S_{\psi, \lambda}(U^{\pp},\mathcal{O}_L/p^n)_{\mathfrak{m}}$$
     de la topologie $p$-adique, alors   $$S_{\psi,\lambda}(U^{\pp}, \mathcal{O}_L)_{\mathfrak{m}}^d
:={\rm Hom}_{\mathcal{O}_L}^{\rm cont}(S_{\psi,\lambda}(U^{\pp}, \mathcal{O}_L)_{\mathfrak{m}}, \mathcal{O}_L)$$
est un $\mathcal{O}_L$-module compact pour la topologie de la convergence simple et 
on dispose d'un isomorphisme 
naturel (combiner la relation $(27)$ dans \cite{Pask} avec la dualit\'e de Schikhof)
$$S_{\psi,\lambda, \mathfrak{m}}^{\vee}\simeq S_{\psi,\lambda}(U^{\pp}, \mathcal{O}_L)_{\mathfrak{m}}^d.$$
\end{rema}
 
 \begin{prop}\label{Paskpuzzle}
  Le $G$-module
   $$M_{\psi,\lambda, \mathfrak{m}}=(k_L\otimes_{R^{{\rm ps}, \zeta}_{\mathcal{B}}} S_{\psi,\lambda, \mathfrak{m}}^{\vee})^{\vee}$$
   est un objet de longueur finie de ${\rm Rep}_{\mathcal{B}}^{\zeta}\, G$, et tout $\pi\in \mathcal{B}$
   est un sous-quotient irr\'eductible de $M_{\psi,\lambda, \mathfrak{m}}$. 
 \end{prop}
 
 \begin{proof} Notons que $M_{\psi,\lambda, \mathfrak{m}}$ a bien un sens gr\^ace \`a la prop.\,\ref{Paskpuzz}. Le fait que 
 $M_{\psi,\lambda, \mathfrak{m}}$ est de longueur finie est une cons\'equence de l'admissibilit\'e de $S_{\psi,\lambda, \mathfrak{m}}^{\vee}$ (prop.\,\ref{Paskpuzz}) et du cor.\,6.7 de \cite{PT}. 
 
   Supposons qu'il existe $\pi\in \mathcal{B}$ qui n'est pas un sous-quotient irr\'eductible de $M_{\psi,\lambda, \mathfrak{m}}$. Soit 
   $S=S_{\psi,\lambda, \mathfrak{m}}$ et $M=M_{\psi,\lambda, \mathfrak{m}}$, et soit $m_{{\rm ps}}$ l'id\'eal maximal de $R^{{\rm ps}, \zeta}_{\mathcal{B}}$. Comme $\pi^{\vee}$ n'intervient pas dans 
   $M^{\vee}\simeq S^{\vee}/m_{{\rm ps}} S^{\vee}$, il n'intervient pas non plus dans 
   $S^{\vee}/m_{{\rm ps}}^j S^{\vee}$ pour tout $j\geq 1$. Si $P$ est une enveloppe projective de 
   $\pi^{\vee}$ dans ${\rm Rep}^{{\rm ladm}, \zeta} \, G$, on a donc ${\rm Hom}_G(P, S^{\vee}/m_{{\rm ps}}^j S^{\vee})=0$ pour tout $j$, en particulier
   ${\rm Hom}_G(P, \varprojlim_{j} S^{\vee}/m_{{\rm ps}}^j S^{\vee})=0$. Comme $S^{\vee}$ est la limite inverse de ses quotients de longueur finie, et chacun est tu\'e par une puissance de $m_{{\rm ps}}$, on a $ \varprojlim_{j} S^{\vee}/m_{{\rm ps}}^j S^{\vee}\simeq S^{\vee}$. Par projectivit\'e de $P$, on en d\'eduit que ${\rm Hom}_G(P, \Theta^d)=0$ pour tout r\'eseau $G$-stable $\Theta$ de 
   $\Pi_x:=S^{\vee}[1/p]/\ker(x)$, et tout morphisme $x: R^{{\rm ps}, \zeta}_{\mathcal{B}}\to \overline{\mathbf{Q}}_p$ (avec comme d'habitude 
   $\Theta^d:={\rm Hom}_{\mathcal{O}_L}^{\rm cont}(\Theta, \mathcal{O}_L)$).
   
   On va voir ci-dessous qu'il existe un morphisme $x: R^{{\rm ps}, \zeta}_{\mathcal{B}}\to \overline{\mathbf{Q}}_p$ tel que 
    la sp\'ecialisation du pseudo-caract\`ere universel en $x$ soit la trace d'une repr\'esentation absolument irr\'eductible (quitte \`a remplacer $L$ par une extension finie) et tel que 
    $\Pi_x$ soit non nul.
     Alors $\Pi_x\simeq \Pi^{\oplus d}$ pour un $d\geq 1$ et une repr\'esentation absolument irr\'eductible non ordinaire $\Pi$ de $G$ (cela suit de \cite[cor.\,6.16]{PT}). 
Mais la compatibilit\'e entre les correspondances de Langlands $p$-adique et modulo $p$ 
(\cite{bergercompat} et \cite[prop.\,III.55, rem.\,III.56]{CD}) montre que 
   $\pi$ intervient dans la r\'eduction modulo $\varpi$ d'un r\'eseau $G$-stable de $\Pi$, une contradiction.     
   
   Pour finir, il reste \`a justifier l'existence d'un tel $x$.
 L'argument est identique \`a celui utilis\'e dans la preuve du th.\,6.5 de \cite{Pask}, avec $S_{\psi,\lambda}(U^{\pp}, \mathcal{O}_L)_{\mathfrak{m}}[1/p]$ \`a la place de $S_{\psi,\lambda}(U^p, \mathcal{O}_L)_{\mathfrak{m}}[1/p]$.
 Plus pr\'ecis\'ement, comme dans loc.cit. on peut construire un pro-$\pp$ sous-groupe ouvert $V_{\pp}$ de $U_{\pp}$ et une repr\'esentation $\gamma$ de $V_{\pp}$ de dimension finie sur 
   $L$ telle que tout vecteur propre pour les op\'erateurs de Hecke agissant sur l'espace de dimension finie ${\rm Hom}_{V_{\pp}}(\gamma, S_{\psi,\lambda}(U^{\pp}, \mathcal{O}_L)_{\mathfrak{m}}[1/p])$ induit une forme automorphe qui est supercuspidale en $\pp$, et donc la repr\'esentation galoisienne associ\'ee est absolument irr\'eductible en restriction \`a $\mathcal{G}_{E_{\pp}}$.
 Enfin, le fait que l'espace ${\rm Hom}_{V_{\pp}}(\gamma, S_{\psi,\lambda}(U^{\pp}, \mathcal{O}_L)_{\mathfrak{m}}[1/p])$ est non nul est expliqu\'e \`a la fin du deuxi\`eme paragraphe dans la preuve de loc.cit.
 \end{proof}
   
\subsubsection{Cohomologie compl\'et\'ee et compatibilit\'e local-global}\label{lgc}

 Soit 
  $D_{\pp}=D\otimes_E E_{\pp}$ l'alg\`ebre de quaternions non d\'eploy\'ee sur $\qp=E_{\pp}$. Si $V$ est un sous-groupe ouvert de 
     $U_p^{\pp} \mathcal{O}_{D_{\pp}}^{\dual}$, notons ${\rm Sh}_{VU^p}$ la courbe de Shimura d\'efinie sur $E$, associ\'ee \`a $VU^p$. 
   
     Le $\mathcal{O}_L$-module\footnote{La limite est prise sur les sous-groupes ouverts compacts $V$ de $U_p^{\pp} \mathcal{O}_{D_{\pp}}^{\dual}$.}
     $$\wh{H}^1(U^p, \mathcal{O}_L/\varpi_L^n):=\varinjlim_{V} H^1_{\eet}({\rm Sh}_{VU^{p}, \bar{E}}, \mathcal{O}_L/\varpi_L^n)$$
    est naturellement un $\mathcal{G}_{E,S}\times \mathbb{T}_S^{\rm univ}[\mathbf{A}_{f}^{\dual}\check{\mathbb G}(E\otimes_{\mathbf{Q}} \qp)]$-module. 
        Soit $\wh{H}^1_{\psi}(U^p, \mathcal{O}_L/\varpi_L^n)_{\mathfrak{m}}$ la partie $\psi$-isotypique (pour l'action de $\mathbf{A}_{f}^{\dual}$) du localis\'e $\wh{H}^1(U^p, \mathcal{O}_L/\varpi_L^n)_{\mathfrak{m}}$ et soit 
    $$\wh{H}^1_{\psi, \lambda}(U^{\pp}, \mathcal{O}_L/\varpi_L^n)_{\mathfrak{m}}:={\rm Hom}_{U_p^{\pp}}
(\lambda, \wh{H}^1_{\psi}(U^p, \mathcal{O}_L/\varpi_L^n)_{\mathfrak{m}}).$$
Posons   $$\wh{H}^1_{\psi, \lambda, \mathfrak{m}}:=\varinjlim_{n} \wh{H}^1_{\psi, \lambda}(U^{\pp}, \mathcal{O}_L/\varpi_L^n)_{\mathfrak{m}},$$
    les applications de transition \'etant induites par celles qui fournissent l'isomorphisme usuel $\varinjlim_{n} \mathcal{O}_L/\varpi_L^n\simeq L/\mathcal{O}_L$.

 Le r\'esultat de compatibilit\'e local-global suivant est obtenu dans \cite[lemma 6.2]{Pask}, en utilisant les r\'esultats de Scholze \cite{SLT}.
     
     \begin{prop}\label{LGC}
      Il existe un isomorphisme naturel de $\mathbb{T}_S^{\rm univ}[\mathcal{G}_{\qp}\times D_{\pp}^{\dual}]$-modules 
      $$S^1(S_{\psi,\lambda, \mathfrak{m}})\simeq \wh{H}^1_{\psi, \lambda, \mathfrak{m}}.$$
     \end{prop}

En munissant $$\wh{H}^1_{\psi, \lambda}(U^{\pp}, \mathcal{O}_L)_{\mathfrak{m}}:=\varprojlim_{n} \wh{H}^1_{\psi, \lambda}(U^{\pp}, \mathcal{O}_L/\varpi_L^n)_{\mathfrak{m}}$$
de la topologie $p$-adique et 
$$\wh{H}^1_{\psi, \lambda}(U^{\pp}, \mathcal{O}_L)_{\mathfrak{m}}^d:={\rm Hom}_{\mathcal{O}_L}^{\rm cont}(\wh{H}^1_{\psi, \lambda}(U^{\pp}, \mathcal{O}_L)_{\mathfrak{m}}, \mathcal{O}_L)$$
de la topologie de la convergence simple (pour laquelle il est compact), on obtient un isomorphisme 
de $\mathbb{T}_S^{\rm univ}[\mathcal{G}_{\qp}\times D_{\pp}^{\dual}]$-modules 
  $$ (\wh{H}^1_{\psi, \lambda, \mathfrak{m}})^{\vee}\simeq  
  (\wh{H}^1_{\psi, \lambda}(U^{\pp}, \mathcal{O}_L)_{\mathfrak{m}})^{d}.$$

La repr\'esentation $\bar{\rho}$ se d\'eforme en une repr\'esentation $$\rho_{\mathfrak{m}}: \mathcal{G}_{E,S}\to \gl_2(\mathbb{T}(U^{\pp})_{\mathfrak{m}}),$$ telle que le polyn\^ome caract\'eristique de $\rho_{\mathfrak{m}}({\rm Frob}_v)$ soit 
$X^2-T_v X+N(v)S_v$ pour $v\notin S$ (cf. prop.\,5.7 de \cite{SLT}).

\begin{prop}\label{isotypic}
 Si $N={\rm Hom}_{\mathbb{T}(U^{\pp})_{\mathfrak{m}}[\mathcal{G}_{E,S}]}(\rho_{\mathfrak{m}}, \wh{H}^1_{\psi,\lambda}(U^{\pp}, \mathcal{O}_L)_{\mathfrak{m}})$, alors le morphisme naturel 
 $$\rho_{\mathfrak{m}}\otimes_{\mathbb{T}(U^{\pp})_{\mathfrak{m}}} N\to \wh{H}^1_{\psi,\lambda}(U^{\pp}, \mathcal{O}_L)_{\mathfrak{m}}$$
 est un isomorphisme.
\end{prop}

\begin{proof} Combiner les prop.\,5.3, 5.4, 5.8 de l'article de Scholze \cite{SLT}.
\end{proof}
   
\subsubsection{Preuve de la prop.\,\ref{key}}\label{lg5} On suppose que 
$\bar{r}=\bar{r}_{\chi_1,\chi_2}$, avec $(\chi_1, \chi_2)$ g\'en\'erique. 
Notons pour simplifier $V=V_{\psi,\lambda, \mathfrak{m}}$, $A=R^{{\rm ps}, \zeta}_{\mathcal{B}}$ et $M=M_{\psi,\lambda, \mathfrak{m}}=(k_L\otimes_{A} S_{\psi,\lambda, \mathfrak{m}}^{\vee})^{\vee}$. 
Alors 
$M$ est 
 un $k_L[G]$-module lisse, de longueur finie, dont l'ensemble des sous-quotients irr\'eductibles est $\{I(\chi_1,\chi_2), I(\chi_2,\chi_1)\}$, d'apr\`es la prop.\,\ref{Paskpuzzle}.

 Le th\'eor\`eme d'annulation de Ludwig (r\'esultat principal de \cite{Ludwig}) 
affirme que $S^2(I(\chi_1,\chi_2))=0=S^2(I(\chi_2,\chi_1))$. Comme de plus 
$S^0(I(\chi_1,\chi_2))=0=S^0(I(\chi_2,\chi_1))$ (utiliser la prop.\,4.7 de \cite{SLT} et le fait que $I(\chi_1, \chi_2)$ et 
$I(\chi_2, \chi_1)$ n'ont pas d'invariants non nuls sous ${\rm SL}_2(\qp)$),  
 on en d\'eduit que $S^1(M)$ est une extension successive de 
copies de $S^1(I(\chi_1,\chi_2))$ et $S^1(I(\chi_2,\chi_1))$. D'autre part,
 la prop.\,\ref{LGC} combin\'ee avec le lemme 6.1 de \cite{Pask} montrent que $$S^1(M)\simeq \wh{H}^1_{\psi,\lambda}(U^{\pp}, k_L)_{\mathfrak{m}}[\mathfrak{m}_A].$$
On conclut que $S^1(I(\chi_1,\chi_2))$ est un sous-quotient de $\wh{H}^1_{\psi,\lambda}(U^{\pp}, k_L)_{\mathfrak{m}}[\mathfrak{m}_A]$ en tant que $\mathcal{G}_{\qp}$-module. 

   Puisque 
$\bar{\rho}$ est une globalisation de $\bar{r}$, la prop.\,\ref{isotypic} montre que 
$\wh{H}^1_{\psi,\lambda}(U^{\pp}, k_L)_{\mathfrak{m}}[\mathfrak{m}_A]$ est de la forme $\bar{r}\otimes_{k_L} N'$ pour un $k_L$-espace vectoriel $N'$ muni de l'action triviale de $\mathcal{G}_{\qp}$. Ainsi, $S^1(I(\chi_1,\chi_2))$ est un sous-quotient de $\bar{r}\otimes_{k_L} N'$ en tant que $\mathcal{G}_{\qp}$-module. En utilisant la suite exacte 
$$0\to \chi_1\otimes_{k_L} N'\to \bar{r}\otimes_{k_L} N'\to \chi_2\otimes_{k_L} N'\to 0$$
on voit que tout sous-quotient de $\bar{r}\otimes_{k_L} N'$ 
contient une copie de $\chi_1$ ou de $\chi_2$, et donc si $S^1(I(\chi_1,\chi_2))^{\mathcal{G}_K}\ne 0$ alors 
$\chi_1$ ou $\chi_2$ se factorise par ${\rm Gal}(K/\qp)$, ce qui permet de conclure.

\subsection{Preuve du th\'eor\`eme de longueur finie}\label{lg6}
  Dans ce paragraphe nous mettons ensemble les r\'esultats \'etablis ci-dessus pour d\'emontrer 
le th\'eor\`eme 
de finitude (th.\,\ref{main11} ci-dessus).

   \begin{proof}({\em Du  th.\,\ref{main11}})
   Puisque $\mathcal{M}_{n,K}^{ p}$ n'a qu'un nombre fini de composantes connexes,
$H^0_{\eet}({\cal M}^{ p}_{n,K'},L)$ est un groupe fini pour tout $K'$ avec $[K':K]<\infty$, et
 une application de la suite de Hochschild-Serre permet de d\'eduire le r\'esultat pour $K$
du r\'esultat pour~$K'$. L'argument est analogue (mais plus simple) \`a celui utilis\'e
dans la preuve du th.\,\ref{smooth}.
Il suffit donc de d\'emontrer le th\'eor\`eme pour 
   $K$ assez grand. 
   
   D'apr\`es le th.\,\ref{smooth}, pour une telle extension $K$, le 
   $G$-module 
    $$V:=H^1_{\eet}(\mathcal{M}_{n,K}^{ p}, k_L)^{\vee}$$
    est lisse, de type fini. Comme $V$ est un $G'$-module\footnote{Rappelons que $G'=G/p^{\Z}$.}, par le 
  th.\,\ref{pierre1}, {\it il suffit de montrer}
 que le cosocle de $V$ est de longueur finie, autrement dit 
{\it que, pour tout $k_L[G']$-module lisse irr\'eductible $\pi$, l'espace  
     $${\rm Hom}_{k_L[G']}(V, \pi)\simeq {\rm Hom}^{\rm cont}_{k_L[G']}(\pi^{\vee}, H^1_{\eet}(\mathcal{M}_{n,K}^{ p}, k_L))$$ est de dimension finie sur $k_L$, nul pour presque tout~$\pi$}
 (i.e.~sauf un nombre fini). 
     
Notons que par la preuve du lemme \ref{cc} le caract\`ere central 
d'un $k_L[G']$-module lisse et irr\'eductible $\pi$  
ne prend qu'un nombre fini de valeurs, {\it on peut donc supposer 
le caract\`ere central de tous les $\pi$ fix\'e, \'egal \`a $\delta$}. 

La preuve consiste \`a suivre le diagramme suivant d'espaces:
     $$
     \xymatrix@R=5mm@C=3mm{
     \sm_{{\rm Dr},\infty} \ar[d]^{p^{\Z}}\ar@{-}[rr]^-{\sim}& &     \sm_{{\rm LT},\infty}\ar[ddr]^{\pi_{\rm GH}}_{G}\\
    \sm^{ p}_{{\rm Dr},\infty} \ar[d]^{\check{G}_n} \\
       \sm^{ p}_{n,\C_p} \ar[d]^{\sg_K}& &&\mathbb{P}^1_{\C_p}\\
       \sm^{ p}_{n,K}
       }
     $$
      \begin{enumerate}
     \item On monte de $\sm_{n,K}^{p}$ au niveau infini $\sm_{{\rm Dr},\infty} $ 
de la tour de Drinfeld, ce qui remplace 
${\rm Hom}^{\rm cont}_{k_L[G']}(\pi^{\vee}, H^1_{\eet}(\mathcal{M}_{n,K}^{{p}}, k_L))$
     par ${\rm Hom}^{\rm cont}_{k_L[G]}(\pi^{\vee}, H^1_{\eet}(\mathcal{M}_{{\rm Dr},\infty}, k_L))^{W_K\times \check{G}_n}$.
     \item 
On passe du c\^ot\'e Lubin-Tate, o\`u est d\'efini le foncteur de Scholze,
 via l'isomorphisme 
$ \sm_{{\rm Dr},\infty}\simeq\sm_{{\rm LT},\infty}:={\cal M}_\infty$
de Faltings~\cite{Faltings}.
     \item Le cor.\,\ref{cafe52}
permet de remplacer
${\rm Hom}^{\rm cont}_{k_L[G]}(\pi^{\vee}, H^1_{\eet}(\mathcal{M}_{\infty}, k_L))^{\mathcal{G}_K\times \check{G}_n}$
par $S^1(\pi)^{W_K\times \check{G}_n}$
 \item On \'etudie $S^1(\pi)^{\mathcal{G}_K\times \check{G}_n}$ par des m\'ethodes globales.
     \end{enumerate}

On dit qu'un foncteur $\pi\mapsto V(\pi)$ en $k_L$-modules est {\it d'image finie}
 si $V(\pi)$ est de dimension finie sur $k_L$ pour tout 
$k_L[G']$-module lisse  irr\'eductible $\pi$, et nul pour presque tout tel~$\pi$.
{\it On cherche \`a prouver que 
$\pi\mapsto {\rm Hom}^{\rm cont}_{k_L[G']}(\pi^{\vee}, H^1_{\eet}(\mathcal{M}_{n,K}^{{p}}, k_L))$
est d'image finie.}
\vskip2mm
    $\bullet$  {\em Mont\'ee en niveau infini.}  

     \begin{lemm}\label{lg11}
Il suffit de prouver que le foncteur
$$\pi\mapsto Y(\pi): ={\rm Hom}^{\rm cont}_{k_L[G]}(\pi^{\vee}, H^1_{\eet}(\mathcal{M}_{\infty}, k_L))^{{W}_K\times \check{G}_n}$$
est d'image finie.
     \end{lemm}
     \begin{proof} On part du $k_L$-module
     $$
      {\rm Hom}^{\rm cont}_{k_L[G']}(\pi^{\vee},   H^1_{\eet}(\mathcal{M}_{n,K}^{{p}}, k_L))
     $$
   et on monte de $\mathcal{M}_{n,K}^{{p}}$ \`a $\mathcal{M}_{\infty}$ en trois \'etapes.
\vskip2mm
\noindent
$\bullet$
{\it \'Etape 1: de $\mathcal{M}_{n,K}^{{p}}$ \`a $\mathcal{M}_{n,\C_p}^{{p}}$.}
On a 
une suite exacte 
      \begin{equation}
      \label{pierwszy}
      0\to H^1(\sg_K, H^0_{\eet}(\mathcal{M}_{n,\C_p}^{{p}}, k_L))\to H^1_{\eet}(\mathcal{M}_{n,K}^{{p}}, k_L)\to H^1_{\eet}(\mathcal{M}_{n,\C_p}^{{p}}, k_L)^{\mathcal{G}_K}.
      \end{equation}
Les fl\`eches dans la suite ci-dessus sont continues et la suite est strictement exacte.
Cela se voit par les m\^emes arguments que
       dans l'\'etape $3$ de la preuve du th.\,\ref{smooth}:
       \begin{itemize}
     \item  On \'ecrit
     des suites exactes analogues pour les ouverts d'un recouvrement Stein de $\mathcal{M}^p_{n,K}$; 
elles existent par le corollaire $3.7.5$ de \cite{dJvdp} et sont strictement exactes car tous les termes sont naturellement munis de la topologie discr\`ete.
     \item    
Comme le premier terme ne d\'epend pas de l'ouvert et donc est  Mittag-Leffler,
on voit que la suite exacte (\ref{pierwszy}) est la limite des suites exactes pour les ouverts;
c'est donc une suite strictement exacte de modules prodiscrets.
     \end{itemize}      
      En appliquant ${\rm Ho}(-):={\rm Hom}^{\rm cont}_{k_L[G]}(\pi^{\vee},-)$ 
\`a (\ref{pierwszy}),
 on obtient donc une suite exacte 
      \begin{align*}
      0\to {\rm Ho}( H^1(\sg_K, H^0_{\eet}(\mathcal{M}_{n,\C_p}^{{p}}, k_L)))\to 
{\rm Ho}(H^1_{\eet}(\mathcal{M}_{n,K}^{{p}}, k_L))
         \to {\rm Ho}(H^1_{\eet}(\mathcal{M}_{n,\C_p}^{{p}}, k_L))^{\mathcal{G}_K}.
        \end{align*}
      Puisque $\mathcal{M}_{n,\C_p}^{{p}}$ n'a qu'un nombre fini de composantes connexes, il suffit de montrer que 
      $\pi\mapsto{\rm Hom}^{\rm cont}_{k_L[G']}(\pi^{\vee},H^1_{\eet}(\mathcal{M}_n^{{p}}, k_L))^{\mathcal{G}_K}$ 
est d'image finie.
      
\vskip2mm
\noindent
$\bullet$
{\it \'Etape 2: 
 de $\mathcal{M}_{n,\C_p}^{{p}}$ \`a $\mathcal{M}_{{\rm Dr},\infty}^{{p}}$.}  
Nous allons maintenant utiliser un argument similaire pour le
$\check{G}_n$-torseur pro\'etale $\mathcal{M}_{{\rm Dr},\infty}^{{p}}\to \mathcal{M}_{n,\C_p}^{{p}}$. 
En utilisant un argument analogue \`a celui de la preuve du th.\,\ref{compare}, 
 on obtient une suite spectrale 
      $$E_2^{i,j}=H^i(\check{G}_n, H^j_{\eet}(\mathcal{M}_{{\rm Dr},\infty}^{{p}}, k_L))\Longrightarrow H^{i+j}_{\eet}(\mathcal{M}_{n,\C_p}^{{p}}, k_L).$$
Celle-ci fournit une suite strictement exacte (m\^eme argument que dans l'\'etape $1$) 
$$0\to H^1(\check{G}_n, H^0_{\eet}(\mathcal{M}_{{\rm Dr},\infty}^{{p}}, k_L))\to H^1_{\eet}(\mathcal{M}_{n,\C_p}^{{p}}, k_L)\to H^1_{\eet}(\mathcal{M}_{{\rm Dr},\infty}^{{p}}, k_L)^{\check{G}_n}.$$
En appliquant ${\rm Ho}(-):={\rm Hom}^{\rm cont}_{k_L[G']}(\pi^{\vee},-)^{\sg_K}$, 
on obtient une suite exacte 
\begin{align*}
      0\to {\rm Ho}(H^1(\check{G}_n, H^0_{\eet}(\mathcal{M}_{{\rm Dr},\infty}^{{p}}, k_L)))^{\sg_K}\to 
    {\rm Ho}(H^1_{\eet}(\mathcal{M}_{n,\C_p}^{{p}}, k_L))^{\sg_K}
         \to {\rm Ho}(H^1_{\eet}(\mathcal{M}_{{\rm Dr},\infty}^{{p}}, k_L))^{\mathcal{G}_K\times\check{G}_n}.
        \end{align*}

      Contrairement \`a $\pi_0(\mathcal{M}_{n,\C_p}^{{p}})$, l'espace 
      $\pi_0(\mathcal{M}_{{\rm Dr},\infty}^{{p}})$ est infini, donc on doit faire plus attention. 
       La description des composantes connexes g\'eom\'etriques de la tour de Lubin-Tate \cite{Strauchgeocc} montre que 
   $\pi_0(\mathcal{M}_{{\rm Dr},\infty})\simeq \qp^{\dual}$,  les groupes $\check{G}$ et $G$ agissant par la norme r\'eduite et le d\'eterminant, respectivement. 
    En prenant les invariants par $p\in Z(G)$, on voit que $\pi_0(\mathcal{M}_{{\rm Dr},\infty}^{{p}})\simeq \qp^{\dual}/p^{\Z}$.
Donc
   \begin{equation}
   \label{ciemno11}
   H^0_{\eet}(\mathcal{M}_{{\rm Dr},\infty}, 
   k_L)\simeq \mathcal{C}^0(\qp^{\dual}, k_L),\quad 
            H^0_{\eet}(\mathcal{M}_{{\rm Dr},\infty}^{{p}}, k_L)\simeq \mathcal{C}^0(\qp^{\dual}/p^{\Z}, k_L).
            \end{equation}
Comme $\mathcal{C}^0(\qp^{\dual}/p^{\Z}, k_L)$ est une repr\'esentation lisse admissible de 
$\check{G}_n$, et comme $\check{G}_n$ est un groupe de Lie $p$-adique compact, sans torsion, 
les $ H^i(\check{G}_n, H^0_{\eet}(\mathcal{M}_{{\rm Dr},\infty}^{{p}}, k_L))$  
sont de dimension finie sur $k_L$, pour tout $i\geq 0$ (cf.~\cite[ lemma 3.13]{SLT}).
      
Il suffit donc de v\'erifier que
$\pi\mapsto{\rm Hom}_{k_L[G']}(\pi^{\vee}, H^1_{\eet}(\mathcal{M}_{{\rm Dr},\infty}^{{p}}, k_L))^{\mathcal{G}_K\times \check{G}_n}$
est d'image finie.
      
\vskip2mm
\noindent
$\bullet$
{\it \'Etape 3:
 de $\mathcal{M}_{{\rm Dr},\infty}^{{p}}$ \`a $\mathcal{M}_{\infty}$.}  
Enfin, on utilise un argument similaire pour le
      $p^{\Z}$-torseur \'etale $\mathcal{M}_{{\rm Dr},\infty}\to \mathcal{M}_{{\rm Dr},\infty}^{{p}}$. 
En utilisant un argument analogue \`a celui de la preuve du th.\,\ref{compare} 
(notons que $p^{\Z}$ est un groupe localement profini),  on obtient une suite spectrale 
      $$E_2^{i,j}=H^i(p^{\Z}, H^j_{\eet}(\mathcal{M}_{{\rm Dr},\infty}, k_L))\Longrightarrow H^{i+j}_{\eet}(\mathcal{M}_{{\rm Dr},\infty}^{{p}}, k_L).$$
Celle-ci fournit une suite strictement exacte (m\^eme argument que dans l'\'etape $1$) 
      \begin{equation}\label{pic1}0\to H^1(p^{\Z}, H^0_{\eet}(\mathcal{M}_{{\rm Dr},\infty}, k_L))\to H^1_{\eet}(\mathcal{M}_{{\rm Dr},\infty}^{{p}}, k_L)\to 
      H^0(p^{\Z},H^1_{\eet}(\mathcal{M}_{{\rm Dr},\infty}, k_L))\to 0
      \end{equation}
      Notons que  \begin{align*}
 H^1(p^{\mathbf{Z}}, H^0_{\eet}(\mathcal{M}_{{\rm Dr},\infty}, k_L))) 
&\simeq  H^1(p^{\mathbf{Z}},\mathcal{C}^0(\qp^{\dual}, k_L))\\
&\simeq  H^1(p^{\mathbf{Z}}, {\rm Ind}^{p^{\Z}}_{\{1\}} \mathcal{C}^0(\Z_p^{\dual}, k_L))
 \simeq H^1(\{1\},\mathcal{C}^0(\Z_p^{\dual}, k_L))=0,
 \end{align*}
o\`u le premier isomorphisme vient de (\ref{ciemno11}) 
et le troisi\`eme du lemme de Shapiro~\cite[prop.\,IX.2.3]{BW}. 

La seconde fl\`eche dans (\ref{pic1}) est donc un isomorphisme, et comme $p^\Z$ agit trivialement sur $\pi^\vee$,
on a $${\rm Hom}_{k_L[G]}(\pi^{\vee}, H^0(p^\Z,H^1_{\eet}(\mathcal{M}_{{\rm Dr},\infty}, k_L)))
={\rm Hom}_{k_L[G]}(\pi^{\vee}, H^1_{\eet}(\mathcal{M}^p_{{\rm Dr},\infty}, k_L)).$$
On peut remplacer
$\mathcal{M}_{{\rm Dr},\infty}$ par $\sm_{\infty}$ gr\^ace \`a l'isomorphisme
de Faltings, ce qui permet de conclure.
\end{proof}

      $\bullet$  {\em Passage au foncteur de Scholze.}---
 Nous allons \'etudier l'espace $Y(\pi)$ via la suite spectrale \ref{SS}.
\begin{lemm}\label{lg12}
Il suffit de d\'emontrer que le foncteur
$\pi\mapsto S^1(\pi)^{\mathcal{G}_K\times \check{G}_n}$ est d'image finie.
      \end{lemm}
      \begin{proof} 
On suppose que $\pi\mapsto S^1(\pi)^{\mathcal{G}_K\times \check{G}_n}$ est d'image finie;
on va montrer qu'il en est de m\^eme de $\pi\mapsto Y(\pi)$.
      
      La suite exacte (\ref{seq11}) de la preuve du cor.\,\ref{cafe52} s'\'ecrit
   \begin{align}
   \label{seq111}
   0  \to V_1(\pi)
\stackrel{f_{\pi}}{\to} S^1(\pi)
 {\to} {\rm Hom}_{k_L[G]}^{\rm cont}(\pi^{\vee}, H^1_{\eet}(\mathcal{M}_{\infty},k_L))\to V_2(\pi)
   \end{align}
o\`u $\pi\mapsto V_1(\pi)$ et $\pi\mapsto V_2(\pi)$, d\'efinis dans la
preuve du cor.\,\ref{cafe52}, sont d'images finies.
    Soit $A(\pi)$ le conoyau de la fl\`eche $f_{\pi}$ de (\ref{seq111}). 
On a une suite exacte
     \begin{align*}
     0\to  & H^0(W_K\times \check{G}_n,V_1(\pi))\to H^0(W_K\times \check{G}_n,S^1(\pi))\stackrel{f_{\pi}}{\to} H^0(W_K\times \check{G}_n,A(\pi))\\
      & \to H^1(W_K\times \check{G}_n,V_1(\pi)).
     \end{align*}
Nous pr\'etendons que $\pi\mapsto H^0(W_K\times \check{G}_n,A(\pi))$ est d'image finie.
En effet, ceci est, par hypoth\`ese, vrai pour
 $\pi\mapsto H^0(W_K\times \check{G}_n,S^1(\pi))$. Il suffit donc de v\'erifier
que c'est vrai aussi pour $H^1(W_K\times \check{G}_n,V_1(\pi))$, ce qui est clair
 car $ \check{G}_n$ est un groupe de Lie $p$-adique compact et $\pi\mapsto V_1(\pi)$ est d'image finie.
   
La suite exacte
$$0\to H^0(W_K\times \check{G}_n,A(\pi))\to Y(\pi)\to 
H^0(W_K\times \check{G}_n,V_2(\pi))$$
ce que l'on sait de $V_2(\pi)$,
et ce que l'on vient de d\'emontrer sur $H^0(W_K\times \check{G}_n,A(\pi))$,
impliquent que $\pi\mapsto Y(\pi)$ est d'image finie, ce que l'on voulait.
\end{proof}

     $\bullet$  {\em  Analyse de $S^1(\pi)^{\mathcal{G}_K\times \check{G}_n}$.} 
Soit  $\pi$ un $k_L[G']$-module, lisse irr\'eductible. Il r\'esulte du lemme~\ref{cc} et 
     du \no\ref{hecke2} que le $G$-module $\pi$ est admissible, 
et donc le th\'eor\`eme de finitude de Scholze \ref{Good} 
permet de conclure que $S^1(\pi)^{\mathcal{G}_K\times \check{G}_n}$ est bien de dimension finie sur $k_L$. 
     
Il reste \`a prouver que $S^1(\pi)^{\mathcal{G}_K\times \check{G}_n}\neq 0$ 
pour seulement un nombre fini de $k_L[G']$-modules $\pi$, lisses et irr\'eductibles. 
     Pour tout tel $\pi$, posons 
      $$\bar{\pi}=\overline{\mathbf{F}}_p\otimes_{k_L} \pi.$$
      Le morphisme naturel $\overline{\mathbf{F}}_p\otimes_{k_L} \mathcal{F}_{\pi}\to 
     \mathcal{F}_{\bar{\pi}}$ est un isomorphisme (cela peut se tester sur les fibres g\'eom\'etriques) et on a un isomorphisme
\'equivariant
      \begin{align}
     \label{iso1}
          S^1(\bar{\pi}) \simeq \overline{\mathbf{F}}_p\otimes_{k_L}  S^1(\pi)
     \end{align}
obtenu en composant les isomorphismes \'equivariants
   \begin{align*}
     S^1(\bar{\pi})  =H^1_{\eet}(\mathbb{P}^1_{\C_p},  \mathcal{F}_{\bar{\pi}})\simeq H^1_{\eet}(\mathbb{P}^1_{\C_p}, \overline{\mathbf{F}}_p{\otimes_{k_L}} \mathcal{F}_{\pi})
         \simeq \overline{\mathbf{F}}_p{\otimes_{k_L}}
 H^1_{\eet}(\mathbb{P}^1_{\C_p}, \mathcal{F}_{\pi})\simeq \overline{\mathbf{F}}_p
{\otimes_{k_L}}  S^1(\pi),
        \end{align*}
(pour le troisi\`eme isomorphisme on a utilis\'e le fait que,
$\mathbb{P}^1_{\C_p}$ \'etant quasi-compact, la cohomologie 
de $\mathbb{P}^1_{\C_p}$ commute aux limites inductives filtrantes).
     
      Soit $M$ l'ensemble des (classes d'isomorphisme de) $k_L[G']$-modules lisses irr\'eductibles 
      $\pi$ pour lesquels $S^1(\pi)^{\mathcal{G}_K\times \check{G}_n}\ne 0$ et soit 
      $T$ l'ensemble 
      des (classes d'isomorphisme de) $\overline{\mathbf{F}}_p[G']$-modules lisses irr\'eductibles 
     $\pi$ tels que $S^1(\pi)^{\mathcal{G}_K}\ne 0$. 
       La prop.\,\ref{key} combin\'ee avec la rem.\,\ref{BL113} montrent que $T$ est fini.        
       
     Si $\pi\in M$, l'isomorphisme (\ref{iso1})  montre que 
 $\bar{\pi}$ poss\`ede un facteur irr\'eductible dans~$T$. 
 Si $M$ est infini, il existe 
 $\pi_1,\pi_2\in M$ non isomorphes 
 et tels que $\overline{\pi}_1$ et $\overline{\pi}_2$ poss\`edent un facteur irr\'eductible commun. 
 Comme $\overline{\pi}_1$ et $\overline{\pi}_2$ sont semi-simples (lemme \ref{irred}), 
il s'ensuit que ${\rm Hom}_{\overline{\mathbf{F}}_p[G]}(\bar{\pi}_1, \bar{\pi}_2)\ne 0$. 
Le lemme~\ref{51} montre que 
 ${\rm Hom}_{k_L[G]}({\pi}_1, {\pi}_2)\ne 0$, une contradiction. 
Donc $M$ est fini, ce qui finit la preuve du th\'eor\`eme.  
        \end{proof}
 
 \begin{rema} Dans la preuve du th.\,\ref{main11}, 
on n'a utilis\'e l'action de $\G_K$ que pour montrer la nullit\'e de 
${\rm Hom}_{k_L[G']}^{\rm cont}(\pi^\vee, H^1_{\eet}(\sm_{n,K}^p,k_L))$ pour presque tout $\pi$,
pas pour la finitude de ce module. La m\^eme preuve (partant de $\sm_{n,\C_p}^p$ au lieu de $\sm_{n,K}^p$)
permet donc de montrer que:
\begin{coro}\label{lg13}
{\it ${\rm Hom}_{k_L[G']}^{\rm cont}(\pi^\vee, H^1_{\eet}(\sm_{n,\C_p}^p,k_L))$ est de dimension finie
pour tout $k_L[G']$-module $\pi$ lisse et admissible}.
\end{coro}
\end{rema}

\subsection{Compl\'ements}\label{lg66}
  La preuve du th\'eor\`eme de finitude a d'autres applications; nous en donnons trois ci-dessous.

\subsubsection{Multiplicit\'es des repr\'esentations de $G$}
La premi\`ere application est valable pour $F$ quelconque:

  \begin{theo}\label{scholzevariation}
  Si $\pi$ est une $k_L$-repr\'esentation lisse admissible de $\gl_2(F)$, \`a caract\`ere central, alors 
  ${\rm Hom}_{k_L[G]}^{\rm cont}(\pi^{\vee}, H^1_{\eet}(\mathcal{M}_{n,\C_p},k_L))$ est de dimension finie sur $k_L$.
  \end{theo}

\begin{proof}
L'argument utilis\'e dans la preuve du cor.\,\ref{cafe52} n'utilise pas l'hypoth\`ese $F=\Q_p$ implicite dans ce corollaire, et fournit une suite exacte 
$$0\to H^1({\rm SL}_2(F), \pi)\to S^1(\pi)\to {\rm Hom}_{k_L[G]}^{\rm cont}(\pi^{\vee}, H^1_{\eet}(\mathcal{M}_{\infty}, k_L))\to H^2({\rm SL}_2(F), \pi).$$
Puisque $\pi$ est admissible en tant que $\gl_2(F)$-repr\'esentation et poss\`ede un caract\`ere central (forc\'ement trivial sur un sous-groupe ouvert de $F^{\dual}$), $\pi$ est admissible en tant que 
${\rm SL}_2(F)$-repr\'esentation lisse. Le lemme 6.3 de \cite{Fust} montre alors que 
$H^i({\rm SL}_2(F), \pi)$ est de dimension finie sur $k_L$ pour tout $i$. Puisque 
$S^1(\pi)$ est admissible en tant que $\check{G}$-repr\'esentation par le th\'eor\`eme de finitude de Scholze, on en d\'eduit que ${\rm Hom}_{k_L[G]}^{\rm cont}(\pi^{\vee}, H^1_{\eet}(\mathcal{M}_{\infty}, k_L))^{\check{G}_n}$ est de dimension finie sur $k_L$. 
D'autre part, les arguments (purement g\'eom\'etriques et donc
valables aussi quand $F\ne \Q_p$) utilis\'es dans les \'etapes $1-3$ de la preuve du lemme \ref{lg11}
fournissent une suite exacte 
$$0\to H^1(\check{G}_n, \mathcal{C}^0(F^{\times}, k_L))\to H^1_{\eet}(\mathcal{M}_{n,\C_p},k_L)\to 
H^1_{\eet}(\mathcal{M}_{\infty}, k_L)^{\check{G}_n}.$$
En appliquant le foncteur ${\rm Hom}_{k_L[G]}^{\rm cont}(\pi^{\vee}, -)$ \`a cette suite exacte 
et en tenant compte de ce que l'on vient de d\'emontrer, il suffit de v\'erifier que 
$$\dim_{k_L} {\rm Hom}_{k_L[G]}^{\rm cont}(\pi^{\vee}, H^1(\check{G}_n, \mathcal{C}^0(F^{\times}, k_L)))<\infty.$$
Le $k_L$-espace vectoriel 
$W=H^1(\check{G}_n, \mathcal{C}^0(\mathcal{O}_F^{\times}, k_L))$ est de dimension finie (par admissibilit\'e de 
la $\check{G}$-repr\'esentation $\mathcal{C}^0(\mathcal{O}_F^{\times}, k_L)$) et 
si $\varpi$ est une uniformisante de $F$,
$$H^1(\check{G}_n, \mathcal{C}^0(F^{\times}, k_L))=H^1(\check{G}_n, {\rm Ind}_{\{1\}}^{\varpi^{\mathbf{Z}}} \mathcal{C}^0(\mathcal{O}_F^{\times}, k_L))={\rm Ind}_{\{1\}}^{\varpi^{\mathbf{Z}}}  W.$$
On en d\'eduit une injection 
$${\rm Hom}_{k_L[G]}^{\rm cont}(\pi^{\vee}, H^1(\check{G}_n, \mathcal{C}^0(F^{\times}, k_L)))
\hookrightarrow (\pi\otimes_{k_L} W)^{{\rm SL}_2(F)},$$ 
et le membre de droite est de dimension finie sur $k_L$
car $\pi$ est admissible en tant que 
${\rm SL}_2(F)$-repr\'esentation et $W$ est de dimension finie sur $k_L$. 
\end{proof}

\subsubsection{Multiplicit\'e des repr\'esentations de $\G_{\Q_p}$}
 Si $\bar{\rho}: \mathcal{G}_{\Q_p}\to {\rm GL}_d(k_L)$ est une repr\'esentation continue, on pose
 $$\bPi^{\rm geo}_n(\bar{\rho}):={\rm Hom}_{k_L[\mathcal{G}_{\Q_p}]}(\bar{\rho}, H^1_{\eet}(\mathcal{M}_{n,\C_p}^{p},k_L))^{\vee}.$$
 Le reste de ce num\'ero est consacr\'e \`a la preuve du r\'esultat suivant:
 
\begin{theo}\label{pointwise2} Soit 
   $\bar{\rho}: \mathcal{G}_{\Q_p}\to {\rm GL}_d(k_L)$ une repr\'esentation continue.
   
   {\rm a)} $\bPi^{\rm geo}_n(\bar{\rho})$ est un $G$-module lisse, de longueur finie.   
   
   {\rm b)} Si $d>2$ et si $\bar{\rho}$ est absolument irr\'eductible, alors $\bPi^{\rm geo}_n(\bar{\rho})=0$.
   
 \end{theo}
 
  \begin{proof}
a) Analogue \`a la preuve du cor.\,\ref{pointwise1}, gr\^ace au th.\,\ref{main11}. 

b) Supposons que $\bPi^{\rm geo}_n(\bar{\rho})\neq 0$, 
et prenons un quotient irr\'eductible $\pi$ de $\bPi^{\rm geo}_n(\bar{\rho})$. Alors 
$\pi^{\vee}$ s'injecte dans ${\rm Hom}_{k_L[\mathcal{G}_{\Q_p}]}(\bar{\rho}, H^1_{\eet}(\mathcal{M}_{n,\C_p}^{p},k_L))$ et donc 
$\bar{\rho}$ s'injecte dans ${\rm Hom}_G(\pi^{\vee}, H^1_{\eet}(\mathcal{M}_{n,\C_p}^{p},k_L))$. Les m\^emes arguments
que dans la preuve du lemme \ref{lg11} combin\'es au fait que $\bar{\rho}$ ne peut pas intervenir dans 
$H^0_{\eet}(\mathcal{M}_{\infty}, k_L)$ et \`a la suite exacte \ref{seq11} montrent que 
${\rm Hom}_{k_L[\mathcal{G}_{\Q_p}]}(\bar{\rho}, S^1(\pi))\ne 0$. 

 Notons que 
$\pi$ n'est pas un caract\`ere, car $S^1$ tue les caract\`eres. Quitte \`a remplacer $L$ par une extension finie, on peut supposer que 
$\pi$ est absolument irr\'eductible. Soit $\mathcal{B}$ le bloc contenant $\pi$ et soit 
$\bar{r}$ la repr\'esentation semi-simple de dimension $2$ associ\'ee \`a $\mathcal{B}$. On globalise 
$\bar{r}$ comme au d\'ebut de ce chapitre et on consid\`ere le module $M=M_{\psi, \lambda, \mathfrak{m}}$ introduit dans la prop.\,\ref{Paskpuzzle}.

 Supposons que $\pi$ est supersinguli\`ere. Alors $\mathcal{B}=\{\pi\}$ et la prop.\,\ref{Paskpuzzle} montre que l'on dispose d'une suite exacte 
 $$0\to \pi\to M\to \pi'\to 0,$$
 $\pi'$ \'etant une extension successive de copies de $\pi$. En particulier $(\pi')^{{\rm SL}_2(\qp)}=0$, donc 
 $S^0(\pi')=0$ et $S^1(\pi)$ s'injecte dans $S^1(M)$, qui lui-m\^eme s'injecte dans un $\mathcal{G}_{\qp}$-module
  $\bar{r}$-isotypique (prop.\,\ref{isotypic} et \ref{LGC}). On en d\'eduit (via la prop.\,5.4 de \cite{SLT}) que 
 $S^1(\pi)$ est $\bar{r}$-isotypique, 
ce qui contredit le fait que ${\rm Hom}_{k_L[\mathcal{G}_{\Q_p}]}(\bar{\rho}, S^1(\pi))\ne 0$. 
 
  Le cas o\`u $\pi$ n'est pas supersinguli\`ere est plus p\'enible. Dans ce cas on peut \'ecrire $\bar{r}=\chi_1\oplus \chi_2$ pour deux caract\`eres 
  galoisiens $\chi_1, \chi_2$. 
  
  \begin{lemm}\label{chiant}
   Soit $Z$ un sous $G$-module de $M$ et soit $H$ un sous-groupe ouvert compact de 
   $D_p^{\dual}$. Si $\mathcal{E}\in \{S^1(Z)^H, H^1(H, S^1(Z))\}$, alors 
   $\mathcal{E}$ est une 
    repr\'esentation de 
   dimension finie de $\mathcal{G}_{\qp}$, 
   dont les    
  sous-quotients irr\'eductibles sont de dimension $1$.
    \end{lemm}
  
  \begin{proof} Le fait que $\dim \mathcal{E}<\infty$ est une cons\'equence de l'admissibilit\'e de 
  $S^1(Z)$. Soit $\alpha: S^1(Z)\to S^1(M)$ la fl\`eche induite par l'inclusion $Z\subset M$. On a une suite exacte 
  $$0\to (\ker(\alpha))^H\to S^1(Z)^H\to ({\rm Im}(\alpha))^H\to H^1(H, \ker(\alpha))\to H^1(H, S^1(Z))\to H^1(H, {\rm Im}(\alpha))$$
  de repr\'esentations de dimension finie de $\mathcal{G}_{\qp}$. La suite exacte 
   $$S^0(M/Z)\to S^1(Z)\to S^1(M)$$
montre que $\ker(\alpha)$ est un quotient de $S^0(M/Z)=S^0((M/Z)^{{\rm SL}_2(\qp)})$ en tant que 
$D^{\dual}\times \mathcal{G}_{\qp}$-module. Mais $(M/Z)^{{\rm SL}_2(\qp)}$ est une extension successive de caract\`eres (quitte \`a remplacer $L$ par une extension finie), donc $S^0((M/Z)^{{\rm SL}_2(\qp)})$ est une extension successive finie de caract\`eres en tant que 
  $D^{\dual}$-module et en tant que $\mathcal{G}_{\qp}$-module. On en d\'eduit que $(\ker(\alpha))^H$ et $H^1(H, \ker(\alpha))$ sont des extensions successives de caract\`eres galoisiens. 
  
  Soit $T={\rm Im}(\alpha)$, alors $T$ est un sous $D^{\dual}\times \mathcal{G}_{\qp}$-module de $S^1(M)$. 
  On sait (prop.\,\ref{isotypic} et \ref{LGC}) que l'on dispose d'une injection 
  $D^{\dual}\times \mathcal{G}_{\qp}$-\'equivariante 
  $S^1(M)\subset \bar{r}\otimes_{k_L} U$ pour un $D^{\dual}$-module lisse admissible 
  $U$ avec action triviale de $\mathcal{G}_{\qp}$. On a donc une injection $D^{\dual}\times \mathcal{G}_{\qp}$-\'equivariante
  $T\subset \bar{r}\otimes_{k_L} U$. On a une suite exacte $D^{\dual}\times \mathcal{G}_{\qp}$-\'equivariante
  $$0\to \chi_1\otimes U\to \bar{r}\otimes U\to \chi_2\otimes U\to 0$$
  et donc une suite exacte $D^{\dual}\times \mathcal{G}_{\qp}$-\'equivariante
  $$0\to \chi_1\otimes U_1\to T\to \chi_2\otimes U_2\to 0,$$
  avec $U_i$ des sous $D^{\dual}$-modules de $U$. En passant \`a la $H$-cohomologie continue on voit que 
  $T^H$ et $H^1(T,H)$ sont des extensions successives de $\chi_1$ et $\chi_2$, ce qui permet de conclure.
    \end{proof}

    Revenons \`a la preuve du th\'eor\`eme. Soit $H$ un sous-groupe ouvert pro-$p$ de $D^{\dual}$. Comme 
    ${\rm Hom}_{k_L[\mathcal{G}_{\Q_p}]}(\bar{\rho}, S^1(\pi))$ est lisse et non nulle,  
     ${\rm Hom}_{k_L[\mathcal{G}_{\Q_p}]}(\bar{\rho}, S^1(\pi)^H)\ne 0$. D'autre part, 
     on dispose (prop.\,\ref{Paskpuzzle}) d'une suite exacte de $G$-modules 
$$0\to X\to Y\to \pi\to 0,$$
avec $X,Y$ des sous $G$-modules de $M$. Comme $\pi$ n'est pas un caract\`ere (donc $\pi^{{\rm SL}_2(\qp)}=0$) on a 
$S^0(\pi)=0$, d'o\`u une suite exacte 
$$0\to S^1(X)\to S^1(Y)\to S^1(\pi)\to S^2(X),$$
ou encore 
$$0\to (S^1(Y)/S^1(X))^H\to S^1(\pi)^H\to S^2(X)^H.$$
Pour arriver \`a une contradiction il suffit de montrer que 
${\rm Hom}_{k_L[\mathcal{G}_{\Q_p}]}(\bar{\rho}, \mathcal{E})=0$ pour $\mathcal{E}\in \{(S^1(Y)/S^1(X))^H,  S^2(X)^H\}$. Si 
$\mathcal{E}=(S^1(Y)/S^1(X))^H$ cela suit du lemme \ref{chiant}, supposons donc que 
$\mathcal{E}=S^2(X)^H$.
Comme $X$ est une extension successive de repr\'esentations qui sont soit des steinberg, soit des induites paraboliques, soit des caract\`eres, et comme $S^2$ tue la steinberg et les induites (par le th\'eor\`eme de Ludwig) et transforme les caract\`eres en des caract\`eres, il s'ensuit que 
$S^2(X)^H$ est une extension successive de caract\`eres en tant que $\mathcal{G}_{\qp}$-module, ce qui permet de conclure.   
 \end{proof}

\begin{rema}\label{pointwise2.1}
 Si $d=2$ et $\bar{\rho}$ est absolument irr\'eductible, alors on peut montrer 
(avec des arguments semblables \`a ceux utilis\'es ci-dessus) que $\bPi^{\rm geo}_n(\bar{\rho})$ 
est dans le bloc de ${\rm Rep}^{{\rm ladm}}\, G$ d\'ecoup\'e par 
  $\bar{\rho}$. Ce r\'esultat n'est pas vrai quand $\bar{\rho}$ est r\'eductible car un morphisme
$\bar\rho\to H^1_{\eet}({\cal M}^p_{n,\C_p},k_L)$ peut se factoriser par un quotient strict
de $\bar\rho$.  Il est possible qu'il reste valable si on impose que
$\bar\rho\to H^1_{\eet}({\cal M}^p_{n,\C_p},k_L)$ soit injectif.
\end{rema}

\subsubsection{Vecteurs lisses du c\^ot\'e Lubin-Tate}
   Soit ${\rm LT}_{n,\C_p}:=\mathcal{M}_{\rm LT, G_n}\otimes_{\breve{\Q}_p} {\C_p}$ 
le $n$-\`eme \'etage de la tour de Lubin-Tate, et soit 
   ${\rm LT}_{n,\C_p}^{p}$ le quotient par l'action de $p^{\Z}$, 
$p$ \'etant vu comme \'el\'ement du centre de $\check{G}$.  Alors 
   ${\rm LT}_{n,\C_p}^{p}$ est d\'efini sur $F$, 
ce qui permet de d\'efinir ${\rm LT}_{n,K}^{p}$ pour toute extension finie $K$ de $F$. 

   \begin{theo}\label{LTfin}
    On suppose que $F=\Q_p$. Si $K$ est une extension finie de 
    $\Q_p$, l'espace des vecteurs $\check{G}$-lisses de 
    $H^1_{\eet}({\rm LT}^p_{n,K}, k_L)$ est de dimension finie sur $k_L$.   
   \end{theo}
   
   \begin{proof}
     Comme dans l'\'etape $1$ de la preuve du lemme \ref{lg11}, il suffit de voir que l'espace des vecteurs lisses de
     $H^1_{\eet}({\rm LT}_{n,\C_p}^p, k_L)^{\mathcal{G}_K}$ est de dimension finie. Comme dans l'\'etape $2$ de la m\^eme preuve on obtient une suite exacte 
     $0\to A\to H^1_{\eet}({\rm LT}_{n,\C_p}^p, k_L)\to H^1_{\eet}(\mathcal{M}_{\infty}^p, k_L)^{G_n}$ avec $A$ de dimension finie, donc il suffit de voir que 
     $\varinjlim_{j} H^1_{\eet}(\mathcal{M}_{\infty}^p, k_L)^{G_n\times \mathcal{G}_K\times \check{G}_j}$ est de dimension finie. On va montrer que cet espace est nul.

\begin{lemm}
 L'espace des vecteurs $G$-lisses de 
$H^1_{\eet}(\mathcal{M}_{\infty}^p, k_L)^{\mathcal{G}_K\times \check{G}_j}$ est de dimension finie.
\end{lemm}

\begin{proof} Les suites spectrales usuelles (voir la preuve du lemme \ref{lg11}) montrent que l'on dispose d'une fl\`eche 
$H^1_{\eet}(\mathcal{M}_{j,K}^p, k_L)\to H^1_{\eet}(\mathcal{M}_{\infty}^p, k_L)^{\mathcal{G}_K\times \check{G}_j}$ dont le noyau et le conoyau sont de dimension finie sur $k_L$. Si $f: A\to B$ est un morphisme $G$-\'equivariant entre des $k_L[G]$-modules, dont le noyau et le conoyau sont de dimension finie sur $k_L$, et si $A^{G-{\rm lisse}}$ est de dimension finie, alors $B^{G-{\rm lisse}}$ est aussi de dimension finie. Il suffit donc de voir que 
$H^1_{\eet}(\mathcal{M}_{j,K}^p, k_L)^{G-{\rm lisse}}$ est de dimension finie, ce qui d\'ecoule du fait que $H^1_{\eet}(\mathcal{M}_{j,K}^p, k_L)$ est le dual d'une repr\'esentation lisse admissible de $G$.
\end{proof}

\begin{lemm}
 L'espace $H^1_{\eet}(\mathcal{M}_{\infty}^p, k_L)$ 
ne contient pas de sous-$G$-module non nul, de dimension finie sur $k_L$.
\end{lemm}     

\begin{proof}
Puisque $L$ est arbitraire, il suffit de voir que $H^1_{\eet}(\mathcal{M}_{\infty}^p, k_L)$ ne contient pas de caract\`ere lisse 
de $G$. Par l'\'etape $3$ de la preuve du lemme \ref{lg11} il suffit de montrer la m\^eme assertion pour $H^1_{\eet}(\mathcal{M}_{\infty}, k_L)$. On veut donc montrer que ${\rm Hom}_G^{\rm cont}(\delta, H^1_{\eet}(\mathcal{M}_{\infty}, k_L))=0$ pour tout caract\`ere lisse $\delta$ de $G$. Comme dans la preuve du cor.\,\ref{cafe52} on a une suite exacte 
$$S^1(\delta^{-1})\to {\rm Hom}_G^{\rm cont}(\delta, H^1_{\eet}(\mathcal{M}_{\infty}, k_L))\to H^2({\rm SL}_2(\Q_p), \delta^{-1}).$$
Le lemme 7.4 de \cite{DPS2} combin\'e avec la trivialit\'e de $S^1(1)$ (qui vient simplement du fait que le faisceau associ\'e \`a la repr\'esentation triviale est le faisceau constant, qui a une cohomologie nulle en degr\'e $1$ sur $\mathbf{P}^1$) montre que le terme \`a gauche est nul. Celui \`a droite l'est aussi: il suffit de voir que $H^2({\rm SL}_2(\Q_p), \mathbf{F}_p)=0$. En utilisant la suite exacte 
$0\to \mathbf{F}_p\to \mathbf{Q}/\mathbf{Z}\to \mathbf{Q}/\mathbf{Z}\to 0$ et le fait que l'ab\'elianis\'e de 
${\rm SL}_2(\Q_p)$ est trivial, il suffit de voir que la $p$-torsion de $H^2({\rm SL}_2(\Q_p), \mathbf{Q}/\mathbf{Z})$ est nulle. Mais 
on voit facilement que $H^2({\rm SL}_2(\Q_p), \mathbf{Q}/\mathbf{Z})\simeq H^2({\rm SL}_2(\Q_p), \mathbf{R}/\mathbf{Z})$ et ce dernier groupe a \'et\'e calcul\'e par Moore \cite{Moore}, et il vaut le dual de Pontryagin du groupe des racines de l'unit\'e dans $\Q_p$. Le r\'esultat s'en d\'eduit. 
\end{proof}

 Si l'on combine les deux lemmes on voit que l'espace des vecteurs $G$-lisses de $H^1_{\eet}(\mathcal{M}_{\infty}^p, k_L)^{\mathcal{G}_K\times \check{G}_j}$ est nul pour tout $j$, 
et donc $\varinjlim_{j} H^1_{\eet}(\mathcal{M}_{\infty}^p, k_L)^{G_n\times \mathcal{G}_K\times \check{G}_j}=0$. 
    \end{proof}

      \begin{rema}
      Le th.\,\ref{LTappli} combin\'e avec l'argument ci-dessus montre que les th.\,\ref{LTfin} et \ref{main11} sont en fait 
      \'equivalents. Il semble raisonnable de penser que le th.\,\ref{LTfin} reste valable pour $F\ne \Q_p$ (alors que nous avons vu que le th.\,\ref{main11} est bien sp\'ecifique \`a $F=\Q_p$).
      \end{rema} 


\section{Factorisation de $H^1_{\eet}(\mathcal{M}_{n,\Qbar_p}, L(1))$}
Dans ce chapitre, on d\'emontre notre r\'esultat principal, \`a savoir la factorisation
\`a la Emerton de la cohomologie \'etale g\'eom\'etrique de
${\cal M}_{n}^p$ (th.\,\ref{factor10} ci-dessous).  

\vskip2mm
Dans tout le chapitre,
$M$ d\'esigne un $(\varphi,N,\G_{\Q_p})$-module, et nous renvoyons au \no\ref{noti5} pour les
d\'efinitions des objets qui lui sont associ\'es: ${\rm LL}(M)$, ${\rm JL}(M)$, $M_{\rm dR}$,
${\cal L}$, $V_{M,{\cal L}}$, $\Pi_{M,{\cal L}}$ ainsi que pour les notions
de {\og supercuspidal\fg}, {\og sp\'ecial\fg} et {\og de niveau~$\leq n$\fg}.

On note $\Phi{\rm N}$ l'ensemble des $M$ sp\'eciaux ou supercuspidaux,
$\Phi{\rm N}^p\subset \Phi{\rm N}$ le sous-ensemble des $M$ tels que
$p\in Z(\check G)$ agit trivialement sur ${\rm JL}(M)$, et on rajoute un $n$
en indice (i.e.~$\Phi{\rm N}_n$ et $\Phi{\rm N}^p_n$)
pour indiquer le sous-ensemble des $M$ de niveau~$\leq n$.
\vskip1mm
Si $X$ est un $L[\check G]$-module lisse avec action triviale de $p\in Z(\check G)$, alors $X$ a une d\'ecomposition
$$X=\oplus_MX[M]\otimes_L {\rm JL}(M),$$
o\`u $\check{G}$ agit trivialement sur $X[M]$.
\vskip1mm
On pose aussi, pour all\'eger un peu les notations,
$$H^1_{\proet}(-):=H^1_{\proet}(-,L(1)),\quad H^1_{\eet}(-):=H^1_{\eet}(-,L(1)),
\quad H^1_{\eet}(-)^+:=H^1_{\eet}(-,\O_L(1)).$$

\Subsection{Compl\'ements \`a \cite{CDN1}}
\subsubsection{Multiplicit\'es des repr\'esentations de $\G_{\Q_p}$}
Le r\'esultat suivant est l'un des r\'esultats principaux de~\cite{CDN1}.
(Si $\delta$ est un caract\`ere de $\Q_p^\dual$, on note 
$$\delta_G:=\delta\circ\nu_G\quad{\rm et}\quad \delta_{\check G}:=\delta\circ\nu_{\check G}$$
les caract\`eres de $G$ et $\check G$, et simplement
$\delta$ le caract\`ere de ${\rm W}_{\Q_p}$ qui lui sont associ\'es.)
   \begin{theo}\label{cdn1.1}
Si $V\in {\rm Rep}_L\G_{\Q_p}$ est absolument irr\'eductible, de dimension~$\leq 2$, alors
    $${\rm Hom}_{{\rm W}_{\Q_p}}(V, H^1_{\eet}(\mathcal{M}_{n,\C_p}))\simeq
\begin{cases}\Pi_{M,{\cal L}}^\dual\otimes_L {\rm JL}(M) 
&{\text{si $V=V_{M,{\cal L}}$ et $M$ de niveau~$\leq n$}},\\
({\rm St}^{\rm cont}\otimes\delta_G)^\dual
\otimes\delta_{\check G} &{\text{si $V=\delta$ et
$\delta$ est de niveau~$\leq n$,}}\\
0 &{\text{si $V$ n'est pas de niveau~$\leq n$.}}\end{cases}$$
   \end{theo}
\begin{rema}\label{cdn1.3}
{\rm (i)} Dans~\cite{CDN1}, le cas des caract\`eres n'est pas trait\'e, mais il se prouve facilement en
remarquant que, si $M$ est supercuspidal, $X_{\rm st}(M)$ n'a pas de sous-$\G_{\Q_p}$-repr\'esentation
de dimension~$1$ (cf. (i) de~\cite[prop.\,2.5]{CDN1}), et donc que les caract\`eres
de $\G_{\Q_p}$ vivent dans les $[M]$-composantes, pour $M$ sp\'ecial. Comme celles-ci
proviennent, par torsion par un caract\`ere, de la cohomologie du demi-plan de Drinfeld,
on conclut en utilisant~\cite[th.\,1.7]{CDN1}.

{\rm (ii)}
Dans~\cite{CDN1},
on prouve aussi que ${\rm Hom}_{{\rm W}_{\Q_p}}(V, H^1_{\eet}(\mathcal{M}_{n,\C_p}))=0$
si $V$ est absolument irr\'eductible, de dimension~$\geq 3$.
La preuve qui y est donn\'ee est combinatoirement assez p\'enible. Le lecteur trouvera
une preuve plus naturelle de cet \'enonc\'e un peu plus loin (th.\,\ref{aju3}).
\end{rema}

\subsubsection{Consid\'erations topologiques}
Si $\mathcal{F}$ est un $L$-espace vectoriel localement convexe muni d'une action continue de $G$,
 on note $\mathcal{F}^b$ l'espace des vecteurs 
     $G$-born\'es de $\mathcal{F}$,
 i.e. les $v\in \mathcal{F}$ pour lesquels $\{g\cdot v,\  g\in G\}$ est born\'e dans 
     $\mathcal{F}$.

     \begin{lemm}\label{b}

{\rm    a)} Si
    $\Pi$ est une $L$-repr\'esentation de Banach unitaire de $G$, alors $(\Pi^\dual)^b=\Pi^\dual$.

{\rm      b)} Si $\pi$ est une $L$-repr\'esentation lisse admissible de 
     $G$, de longueur finie, alors $(\pi^\dual)^b\simeq \widehat{\pi}^\dual$,
    o\`u $\widehat{\pi}$ est le compl\'et\'e unitaire universel de $\pi$. 
    
{\rm    c)} Si $\Pi$ est une repr\'esentation unitaire admissible de $G$, de longueur finie, alors 
    $(\Pi^{\rm an,*})^b=\Pi^\dual$.
         \end{lemm}
     
     \begin{proof}
D'apr\`es le th\'eor\`eme de Banach-Steinhaus, si $\Pi$ est un fr\'echet (en particulier, si $\Pi$ est un banach),
$\mu\in \Pi^\dual$ est $G$-born\'e si et seulement si $\{\langle \mu,g^{-1}\cdot v\rangle,\ g\in G\}$
est born\'e, pour tout $v\in G$.  On en d\'eduit facilement que, si $\Pi$ est un banach unitaire, tout $\mu$ est
$G$-born\'e, ce qui prouve le a).

Pour prouver le b), il suffit de remarquer que $\widehat{\pi}$ est le compl\'et\'e par la norme d\'efinie
par $\O_L[G]\cdot v_1+\cdots +\O_L[G]\cdot v_r$, pour tous $v_1,\dots,v_r$ tels que ce module soit un r\'eseau
de $\pi$. Un \'el\'ement de $\widehat{\pi}^\dual$ est donc un \'el\'ement de $\pi^\dual$
born\'e sur $\O_L[G]\cdot v_1+\cdots  +\O_L[G]\cdot v_r$, pour tous les $v_1,\dots,v_r$ comme ci-dessus;
c'est donc un \'el\'ement $G$-born\'e de $\pi^\dual$ d'apr\`es le th\'eor\`eme de Banach-Steinhaus.

Enfin, le c) est~\cite[th.\,0.2]{CD}. 
     \end{proof}
      Si $X,Y$ sont des $L$-espaces vectoriels localement convexes, 
on \'ecrit ${\rm Hom}(X,Y)$ pour l'espace des applications $L$-lin\'eaires continues de $X$ dans $Y$, muni de la topologie forte (i.e. de la convergence sur les parties born\'ees de $X$), et
on note simplement $X'={\rm Hom}(X,L)$ le dual fort de $X$.
      
      \begin{lemm}\label{aju1}
       Soit $E$ un $L$-fr\'echet nucl\'eaire et soient $F,D$ des $L$-fr\'echets. Il existe un isomorphisme naturel 
       $${\rm Hom}(E, F\widehat{\otimes}_L D)\simeq {\rm Hom}(E, F)\widehat{\otimes}_L D.$$
      \end{lemm}
      
      \begin{proof}
      Comme $E$ est un $L$-fr\'echet nucl\'eaire, il est bornologique et semi-r\'eflexif, 
donc par \cite[prop.\,18.8]{livreSchneider}, 
on a un isomorphisme naturel $E'\widehat{\otimes}_{L} X\simeq {\rm Hom}(E, X)$ pour tout $L$-fr\'echet $X$. On a donc 
      $${\rm Hom}(E, F\widehat{\otimes}_L D)\simeq E'\widehat{\otimes}_{L} ( F\widehat{\otimes}_L D)\simeq 
      (E'\widehat{\otimes}_{L} F)\widehat{\otimes}_{L} D\simeq {\rm Hom}(E, F)\widehat{\otimes}_{L} D,$$
      ce qui permet de conclure.      
      \end{proof}
      
      \begin{lemm}\label{aju2}
       Si $X$ est un $L$-espace localement convexe s\'epar\'e muni d'une action continue de 
       $G$, et si $B$ est un $L$-banach muni de l'action triviale de $G$, alors 
       on a un isomorphisme naturel $(X\widehat{\otimes}_L B)^G\simeq X^G\widehat{\otimes}_L B$.
      \end{lemm}
      
      \begin{proof}
Si $(e_i)_{i\in I}$ est une base orthonormale de $B$, alors $X\widehat{\otimes}_L B=\ell^\infty_0(I,X)$
(suites tendant vers $0$ suivant le filtre des compl\'ementaires des parties finies), et
$(X\widehat{\otimes}_L B)^G=\ell^\infty_0(I,X^G)=X^G\widehat{\otimes}_L B$.
\end{proof}

     Rappelons que $H^1_{\proet}(\mathcal{M}_{n, \C_p})$ est un espace de Fr\'echet, en tant que limite inverse des espaces de Banach $H^1_{\eet}(U_i)$ (de boule unit\'e $H^1_{\eet}(U_i)^+$), 
si $\{U_i\}_{i\geq 1}$ est un recouvrement Stein de $\mathcal{M}_{n, \C_p}$. Le r\'esultat suivant est 
\cite[prop.\,2.12]{CDN1}.       
     \begin{prop}\label{etproet}
     L'application naturelle $H^1_{\eet}(\mathcal{M}_{n, \C_p})\to H^1_{\proet}(\mathcal{M}_{n, \C_p})$ est injective et induit un isomorphisme     
      $$H^1_{\eet}(\mathcal{M}_{n, \C_p})\simeq H^1_{\proet}(\mathcal{M}_{n, \C_p})^b_.$$
            \end{prop}
            
\Subsubsection{Multiplicit\'es de repr\'esentations de $G$}

\begin{prop}\label{omega}
   Soit $\Pi$ une $L$-repr\'esentation de Banach unitaire, admissible, absolument irr\'eductible de 
   $G$. Alors 
   $${\rm Hom}_G(\Pi^\dual, L\otimes \Omega^1(\mathcal{M}^p_{n,\Q_p}))\simeq
\begin{cases}\mathcal{L}\otimes_L {\rm JL}(M) &{\text{si $\Pi=\Pi_{M,{\cal L}}$ et $M$ de niveau~$\leq n$}},\\
 \delta_{\check G} &{\text{si $\Pi={\rm St}^{\rm cont}\otimes \delta_{G}$, 
$\delta$ de niveau~$\leq n $}},\\
0 &{\text{sinon.}}\end{cases}$$

   \end{prop}

\begin{proof}
  D'apr\`es la discussion ci-dessus il suffit de comprendre la multiplicit\'e de $\Pi^\dual$ 
dans $(L\otimes\Omega^1(\mathcal{M}^p_{n,\Q_p}))[M]$ pour tout $M\in \Phi{\rm N}^p$. 
  
   Si $M$ est sp\'ecial, on se ram\`ene, par torsion, \`a $M={\rm Sp}$,
et $(L\otimes\Omega^1(\mathcal{M}_{n,\Q_p}^p))[M]$ 
est alors l'espace des formes diff\'erentielles sur le demi-plan de Drinfeld,
 dont la structure est tr\`es simple, voir~\cite[prop.\,1.6]{CDN1}. 
Le r\'esultat s'en d\'eduit aisement dans ce cas. 

  On suppose dans ce qui suit que $M$ est supercuspidal.
   Le th\'eor\`eme principal de~\cite{DL} fournit un isomorphisme\footnote{Dans \cite{DL} il y a une hypoth\`ese ambiante que les caract\`eres centraux des repr\'esentations sont triviaux, et donc une identification implicite entre $M_{\rm dR}$ et son dual.} 
$$j: (L\otimes H^1_{\rm dR}(\mathcal{M}^p_{n,\Q_p}))[M]\simeq M_{\rm dR}\otimes_L {\rm LL}(M)^\dual$$
  et, pour chaque 
   $\mathcal{L}$, une suite exacte 
   $$0\to (\Pi_{M,\mathcal{L}}^{\rm an})^\dual\to (L\otimes\Omega^1(\mathcal{M}^p_{n,\Q_p}))[M]\to 
(M_{\rm dR}/\mathcal{L})\otimes_{L} {\rm LL}(M)^\dual\to 0,$$
    la fl\`eche 
$(L\otimes\Omega^1(\mathcal{M}^p_{n,\Q_p}))[M]\to (M_{\rm dR}/\mathcal{L})\otimes_{L} {\rm LL}(M)^\dual$ 
\'etant induite par la fl\`eche naturelle 
$\Omega^1(\mathcal{M}^p_{n,\Q_p})\to H^1_{\rm dR}(\mathcal{M}^p_{n,\Q_p})$ et par l'isomorphisme $j$. Distinguons deux cas:
   
\vskip1mm
   $\bullet$ $\Pi$ est une s\'erie principale ou une tordue de la Steinberg.
Dans ce cas, 
   ${\rm Hom}_G (\Pi^\dual, (\Pi_{M,\mathcal{L}}^{\rm an})^\dual)=0$ et 
   ${\rm Hom}_G(\Pi^\dual, {\rm LL}(M)^\dual)=0$ (le premier point d\'ecoule du point c) du lemme \ref{b}, 
le  second r\'esulte de ce que ${\rm LL}(M)$ est supercuspidale, alors que les vecteurs lisses de $\Pi$
sont de la s\'erie principale).
   
\vskip1mm
   $\bullet$ $\Pi=\bPi(V)$ avec $V$ irr\'eductible de dimension $2$. 

\quad $\diamond$ Si 
   $V$ n'est pas de type $M$, alors ${\rm Hom}_G(\Pi^\dual, (\Pi_{M,\mathcal{L}}^{\rm an})^\dual)=0$ 
(par le point c) du lemme \ref{b} et l'injectivit\'e 
de la correspondance $V\mapsto \bPi(V)$ pour $V$ irr\'eductible de dimension $2$), 
et ${\rm Hom}_G(\Pi^\dual, {\rm LL}(M)^\dual)=0$ 
(par la compatibilit\'e de la correspondance de Langlands locale $p$-adique avec la classique), 
donc ${\rm Hom}_G(\Pi^\dual, (L\otimes\Omega^1(\mathcal{M}^p_{n,\Q_p}))[M])=0$ dans ce cas. 

\quad $\diamond$ Si $V$ est de type $M$, disons $V=V_{M, \mathcal{L}}$, en appliquant 
   ${\rm Hom}_G(\Pi^\dual, -)$ \`a la suite exacte ci-dessus on voit que 
   ${\rm Hom}_G(\Pi^\dual, (L\otimes\Omega^1(\mathcal{M}^p_{n,\Q_p}))[M])\ne 0$. 
En appliquant le m\^eme foncteur \`a une suite exacte du type ci-dessus,
 mais avec un autre $\mathcal{L}'\subset M_{\rm dR}$ 
(et en utilisant le fait que ${\rm Hom}_G(\Pi^\dual, (\Pi_{M, \mathcal{L'}}^{\rm an, *}))=0$,
 qui d\'ecoule encore du lemme \ref{b}),
 on en d\'eduit que ${\rm Hom}_G(\Pi^\dual, (L\otimes\Omega^1(\mathcal{M}^p_{n,\Q_p}))[M])$ 
est de dimension au plus $1$ et s'injecte dans $\mathcal{L}$,
 ce qui permet de conclure.
\end{proof}

   \begin{theo}\label{cdn1.2}
    Soit $\Pi$ une $L$-repr\'esentation de Banach unitaire, admissible, absolument irr\'eductible de 
   $G$. Alors 
   $${\rm Hom}_G(\Pi^\dual, H^1_{\eet}(\mathcal{M}^p_{n, \C_p}))\simeq
\begin{cases}V_{M,{\cal L}}\otimes_L {\rm JL}(M) &{\text{si $\Pi=\Pi_{M,{\cal L}}$, et $M$ de niveau~$\leq n$}},\\
\delta\otimes \delta_{\check G}
 &{\text{si $\Pi={\rm St}^{\rm cont}\otimes \delta_G$, $\delta$ de niveau~$\leq n $}},\\
0 &{\text{sinon}}.\end{cases}$$
   \end{theo}
   
  \begin{proof} On d\'ecompose $H^1_{\eet}(\mathcal{M}^p_{n, \C_p})$ selon les divers 
  $M\in \Phi{\rm N}^p_n$. Si $M$ est sp\'ecial, par torsion on se ram\`ene au cas $M={\rm Sp}$, et alors   
 $$ H^1_{\eet}(\mathcal{M}_{n,\C_p})[M]\simeq 
   ({\rm St}^{\rm cont})^\dual,$$
   par \cite{CDN1}, ce qui permet de conclure.
   
   On suppose par la suite que $M$ est supercuspidal.  
   Nous allons montrer d'abord que l'on peut remplacer $\Pi$ par 
  $\Pi^{\rm an}$ et $H^1_{\eet}(\mathcal{M}^p_{n, \C_p})$ par $H^1_{\proeet}(\mathcal{M}^p_{n, \C_p})$ sans changer les multiplicit\'es.
      (Le passage de $\Pi$ \`a $\Pi^{\rm an}$ va nous permettre d'utiliser le lemme~\ref{aju1}.)

      \begin{lemm}\label{passtoan}
      Les injections $\Pi^\dual\to (\Pi^{\rm an})^\dual$ et $H^1_{\eet}(\mathcal{M}^p_{n, \C_p})\to H^1_{\proeet}(\mathcal{M}^p_{n, \C_p})$ induisent un isomorphisme 
      $${\rm Hom}_G((\Pi^{\rm an})^\dual, H^1_{\proeet}(\mathcal{M}^p_{n, \C_p}))
\overset{\sim}{\to}
 {\rm Hom}_G(\Pi^\dual, H^1_{\eet}(\mathcal{M}^p_{n, \C_p})).$$
      \end{lemm}
      
      \begin{proof}
       Tout morphisme $G$-\'equivariant $f: (\Pi^{\rm an})^\dual\to H^1_{\proeet}(\mathcal{M}^p_{n, \C_p})$
       doit envoyer $\Pi^\dual$ dans $H^1_{\proeet}(\mathcal{M}^p_{n, \C_p})^b$, qui s'identifie \`a 
       $H^1_{\eet}(\mathcal{M}^p_{n, \C_p})$ par la prop.\,\ref{etproet}. Cela fournit une fl\`eche 
       $$\iota: {\rm Hom}_G((\Pi^{\rm an})^\dual, H^1_{\proeet}(\mathcal{M}^p_{n, \C_p}))
\to {\rm Hom}_G(\Pi^\dual, H^1_{\eet}(\mathcal{M}^p_{n, \C_p})),$$
       dont l'injectivit\'e est une cons\'equence de la densit\'e de 
       $\Pi^\dual$ dans $(\Pi^{\rm an})^\dual$.
        
      Soit $f: \Pi^\dual\to H^1_{\eet}(\mathcal{M}^p_{n, \C_p})$ 
       une fl\`eche $G$-\'equivariante continue. Soit $K$ un sous-groupe ouvert compact de 
       $G$ et soit $D(K)$ (resp. $\Lambda(K)$) l'alg\`ebre des distributions (resp. des mesures) sur 
       $K$ \`a valeurs dans $L$. Le $G$-Fr\'echet $H^1_{\proeet}(\mathcal{M}^p_{n, \C_p})$ est un 
       $D(K)$-module, car il s'identifie \`a un sous-fr\'echet $K$-stable de 
$\Omega^1(\mathcal{M}^p_{n, \C_p})=
       \Omega^1(\mathcal{M}_{n,\Q_p}^p)\widehat{\otimes}_{\Q_p} \C_p$, et ce dernier est un $D(K)$-module topologique par 
       le th.\,3.2 de \cite{DL}. Donc $f$ se prolonge en une application 
       $f: \Pi^\dual\otimes_{\Lambda(K)} D(K)\to H^1_{\proeet}(\mathcal{M}^p_{n, \C_p})$. Par le th\'eor\`eme de Schneider et Teitelbaum (th.\,7.1 de \cite{STInv}) le terme de gauche s'identifie \`a $(\Pi^{\rm an})^\dual$, et l'application qui s'en d\'eduit 
       $\tilde{f}: (\Pi^{\rm an})^\dual\to H^1_{\proeet}(\mathcal{M}^p_{n, \C_p})$ est $G$-\'equivariante, continue et un ant\'ec\'edent de $f$, ce qui montre que $\iota$ est bien un isomorphisme.
      \end{proof}

     Soit $M\in \Phi{\rm N}^p_{n}$ supercuspidal. Le th.\,5.11 de \cite{CDN1} 
  montre l'existence d'un diagramme commutatif de $G$-fr\'echets 
$$\xymatrix@R=.5cm@C=.4cm{0 \ar[r]&\ (L\otimes\mathcal{O}({\cal M}^p_{n, \C_p}))[M]\ar[r]^-{{\rm exp}}\ar@{=}[d]&
H^1_{\proet}({\cal M}^p_{n, \C_p})[M]\ar[d]^{{\rm dlog}}\ar[r]&
X_{\rm st}^+(M)\widehat{\otimes}_L{\rm LL}(M)^\dual \ar[d]^{\theta}\ar[r]&0\\
0\ar[r]& (L\otimes\mathcal{O}({\cal M}^p_{n, \C_p}))[M] \ar[r]^-d&\Omega^1({\cal M}^p_{n, \C_p})[M]\ar[r]
&({\C_p} \otimes_{\Q_p} M_{\rm dR})\widehat{\otimes}_L {\rm LL}(M)^\dual \ar[r]&0
},$$
les fl\`eches verticales \'etant injectives 
(car $M$ est de pente $1/2$, et $X^+_{\rm st}(M)\to {\C_p}\otimes_{\Q_p}M_{\rm dR}$ est injective), 
d'image ferm\'ee (comme le montre le th.\,1.8 de \cite{CDN3}). 
      On en d\'eduit une suite exacte de $G$-fr\'echets 
 $$0\to  H^1_{\proeet}(\mathcal{M}^p_{n,\C_p})[M]\to 
(L\otimes\Omega^1(\mathcal{M}^p_{n, \C_p}))[M]
\to \tfrac{M_{\rm dR}\otimes {\C_p}}{X_{\rm st}^+(M)}\widehat{\otimes}_L {\rm LL}(M)^\dual\to 0.$$
 En appliquant ${\rm H}(-):={\rm Hom}_G((\Pi^{\rm an})^\dual,-)$, 
et en tenant compte des r\'esultats ci-dessus,  
 on obtient 
\begin{align*}
{\rm H}\big(H^1_{\eet}(\mathcal{M}^p_{n, \C_p})&[M]\big)\simeq {\rm H}\big(H^1_{\proeet}(\mathcal{M}_{n,\C_p})[M]\big)\\
&\simeq{\rm Ker}\big[\,
 {\rm H}\big((L\otimes\Omega^1(\mathcal{M}_{n,\Q_p}))[M]\widehat{\otimes}_{\Q_p} \C_p\big)
\to {\rm H}\big(\tfrac{M_{\rm dR}\otimes {\C_p}}{X_{\rm st}^+(M)}\widehat{\otimes}_L {\rm LL}(M)^\dual\big)\,\big]\\
& \simeq {\rm Ker}\big[\, {\rm H}\big((L\otimes\Omega^1(\mathcal{M}_{n,\Q_p}))[M]\big)\widehat{\otimes}_{\Q_p} \C_p\to 
 \tfrac{M_{\rm dR}\otimes {\C_p}}{X_{\rm st}^+(M)}\widehat{\otimes}_L {\rm H}\big({\rm LL}(M)^\dual\big)\,\big]
\end{align*}
 
On a ${\rm Hom}_G((\Pi^{\rm an})^\dual, (L\otimes\Omega^1(\mathcal{M}_{n,\Q_p}))[M])=0$, 
 par la prop.\,\ref{omega} et un argument comme dans le lemme \ref{passtoan},
sauf si 
 $\Pi=\Pi_{M, \mathcal{L}}$, auquel cas l'espace en question est~$\mathcal{L}$.
 On peut donc supposer que $\Pi=\Pi_{M, \mathcal{L}}$, 
et alors $ {\rm Hom}_G((\Pi^{\rm an})^\dual, {\rm LL}(M)^\dual)\simeq {\cal L}$, donc 
\begin{align*}
 {\rm Hom}_G(\Pi^\dual, H^1_{\eet}(\mathcal{M}^p_{n, \C_p}))&\simeq 
{\rm Ker}\big(\mathcal{L}\otimes_{\Q_p} \C_p\to 
   \tfrac{M_{\rm dR}\otimes {\C_p}}{X_{\rm st}^+(M)}\big)\\
&\simeq X_{\rm st}^+(M)\cap (\mathcal{L}\otimes_{\Q_p} \C_p)\simeq V_{M,\mathcal{L}},
\end{align*}
ce qui permet de conclure. 
     \end{proof}

     Le th\'eor\`eme de finitude~\ref{main11} combin\'e au r\'esultat ci-dessus permet
 d'obtenir une preuve nettement plus naturelle (mais pas forc\'ement plus simple...) du r\'esultat suivant, un des plus d\'elicats de \cite{CDN1}.
     
\begin{theo}\label{aju3}
Si $V$ est une $L$-repr\'esentation absolument irr\'eductible de $\G_{\Q_p}$, de dimension $\geq 3$,
alors 
$${\rm Hom}_{\G_{\Q_p}}(V, H^1_{\eet}(\mathcal{M}^p_{n, \C_p}))=0.$$
\end{theo}
   
   \begin{proof}
   Supposons que ce n'est pas le cas, alors la $G$-repr\'esentation $\Pi(V)={\rm Hom}_{\G_{\Q_p}}(V, H^1_{\eet}(\mathcal{M}^p_{n, \C_p}))^\dual$ est une $L$-repr\'esentation de Banach (lemme\,\ref{bana1}) 
qui poss\`ede un r\'eseau $\Pi(V^+)$ dont la r\'eduction modulo ${\goth m}_L$ est de longueur finie 
(par le th\'eor\`eme de finitude), 
donc admissible. Ainsi $\Pi(V)$ est une $L$-repr\'esentation de Banach unitaire, admissible, de longueur finie. Par une cons\'equence du lemme de Schur~\cite[lemme\,3.14]{DS}, quitte \`a remplacer
   $L$ par une extension finie,
 on peut supposer que $\Pi(V)$ contient une $L$-repr\'esentation de Banach unitaire, admissible, 
absolument irr\'eductible $\Pi$. Mais alors on obtient une injection 
$V\to {\rm Hom}_G(\Pi^\dual, H^1_{\eet}(\mathcal{M}^p_{n, \C_p}))$, ce qui contredit le th.\,\ref{cdn1.2}.
   \end{proof}
   
\subsection{Familles de repr\'esentations potentiellement semi-stables}\label{fami1}
Si $M={\rm Sp}\otimes\eta$ est sp\'ecial (note~\ref{special}), 
on pose $R_{M,{\cal B}}=L$.
On pose aussi
$\rho_{{\cal B},M}=0$ et $\bPi^\dual(\rho_{{\cal B},M})=0$, sauf si
${\cal B}$ est le bloc de
${\rm St}\otimes\bar\eta$, o\`u $\bar\eta$ est la r\'eduction modulo~$p$ de
$\eta$, auquel cas 
on pose
$\rho_{{\cal B},M}=\eta$ 
(vu comme caract\`ere de $\G_{\Q_p}$
via la th\'eorie locale du corps de classes), 
et $\bPi^\dual(\rho_{{\cal B},M})=({\rm St}^{\rm cont}\otimes\eta)^\dual$.

    Supposons maintenant que $M$ est supercuspidal et fixons un bloc $\cal B$.
     Soit $I_{M, \cal B}$ l'intersection des id\'eaux maximaux $\pp$ de 
     $R_{\cal B}^{{\rm ps},\delta_M}[1/p]$ tels que la sp\'ecialisation en 
     $\pp$ du pseudo-caract\`ere universel soit la trace d'une repr\'esentation 
     de de Rham \`a poids~$0$ et~$1$, et de type $M$ 
(et donc de d\'eterminant $\delta_M\epsilon$). L'anneau 
$$R_{{\cal B},M}:=R_{\cal B}^{{\rm ps},\delta_M}[1/p]/I_{\cal B, M}$$ est alors r\'eduit, de Jacobson, et on note
$R^+_{{\cal B},M}$ l'image de $R_{\cal B}^{{\rm ps},\delta_M}$ dans $R_{{\cal B},M}$. On a donc 
$R_{{\cal B},M}:=R^+_{{\cal B},M}[\tfrac{1}{p}]$. 

\begin{theo}\label{fami5}
L'anneau $R_{{\cal B},M}$ s'identifie
\`a l'anneau des fonctions analytiques born\'ees sur un ouvert de $\piqp$,
est un produit fini d'anneaux principaux,
 et il existe une unique {\rm (\`a isomorphisme pr\`es)} 
repr\'esentation $$\rho_{{\cal B},M}:\G_{\Q_p}\to\gl_2(R_{{\cal B},M})$$
 telle que
$${\rm Tr}\circ \rho_{{\cal B},M}=\alpha_{{\cal B},M}\circ T_{\cal B}^{\delta_M}.$$
\end{theo}

\begin{proof}
L'unicit\'e de $\rho_{{\cal B},M}$ est une cons\'equence du lemme~\ref{abstrait} ci-dessous et de l'irr\'eductibilit\'e
des $V_{M,{\cal L}}$ (cons\'equence de celle de $M$). L'existence de $\rho_{{\cal B},M}$ et les autres propri\'et\'es de $R_{{\cal B},M}$ sont \'etablies dans \cite[th.\,0.1]{petit}.  
\end{proof}

 \begin{lemm}\label{abstrait}
  Soit $A$ un anneau principal ayant une infinit\'e d'id\'eaux premiers (ou un produit de tels anneaux),
et soient $V,W$ des repr\'esentations de $\mathcal{G}_{\Q_p}$ localement libres de type fini sur $A$ 
et dont les sp\'ecialisations en tout point ferm\'e de 
  ${\rm Spec}(A)$ sont absolument irr\'eductibles. 
Si ces sp\'ecialisations sont isomorphes pour tout point ferm\'e, alors $V$ et $W$ sont isomorphes. 
 \end{lemm}
\begin{proof}
Supposons $A$ principal (donc int\`egre).
Les repr\'esentations ont m\^eme trace modulo tout id\'eal maximal, et donc ont m\^eme trace.
Elles sont irr\'eductibles sur le corps des fractions $K$ de $A$ car, sinon, elles seraient r\'eductibles
modulo presque tout id\'eal maximal de $A$; 
il s'ensuit qu'elles sont isomorphes sur $K$. Si $V_\sigma$ et $W_\sigma$
d\'esignent les matrices de $\sigma\in\G_{\Q_p}$ dans des bases de $V$ et $W$ sur $A$, il existe
$M\in\gl_d(K)$ telle que $V_\sigma=M^{-1}W_\sigma M$, pour tout $\sigma$.

La th\'eorie des diviseurs \'el\'ementaires permet d'\'ecrire $M$ sous la forme $PDQ$, avec
$P,Q\in\gl_d(A)$ et $D$ diagonale de $(\lambda_1,\dots,\lambda_d)$ et $\lambda_i\mid\lambda_{i+1}$.
Changer de bases permet de supposer que $P=Q=1$, et quitte \`a multiplier $M$ par une matrice scalaire,
on peut supposer que $\lambda_1$ est une unit\'e.  Il s'agit de prouver que tous les $\lambda_i$ sont des
unit\'es. Dans le cas contraire, il y a un $i$ minimum qui n'est pas une unit\'e, et la relation
$V_\sigma=D^{-1}W_\sigma D$ implique que $V$ et $W$ ne sont pas irr\'eductibles modulo les id\'eaux maximaux
divisant $\lambda_i$, contrairement \`a l'hypoth\`ese.

Ceci permet de conclure dans le cas $A$ principal. 
Le cas g\'en\'eral s'en d\'eduit en d\'ecomposant tout
via les idempotents de $A$ correspondant \`a chaque facteur de $A$.
\end{proof}

\begin{rema}
 Pour la plupart des blocs $\mathcal{B}$ le th.\,\ref{fami5}
 d\'ecoule facilement des r\'esultats de Kisin 
 \cite{Kis} et de l'observation \cite{Sandra} que l'espace rigide associ\'e \`a $R_{{\cal B},M}$ 
s'identifie \`a un ouvert de la droite projective analytique 
(en utilisant la filtration de Hodge des diverses repr\'esentations 
 interpol\'ees par cet anneau), ce qui permet de montrer que cet anneau est un produit fini d'anneaux principaux. Les blocs d\'elicats sont ceux correspondant \`a un twist des repr\'esentations 
 $1\oplus 1$ et $1\oplus \epsilon$ (cette derni\`ere seulement pour $p=3$). 
L'approche de \cite{petit} permet de traiter tous les blocs sur le m\^eme pied.
\end{rema}

  Si $x$ est un id\'eal maximal de $R_{{\cal B},M}$ on note
  $\kappa(x)$ son corps r\'esiduel (une extension finie de $L$) et 
  $\rho_x$ la sp\'ecialisation de $\rho_{{\cal B},M}$ en $x$.

\begin{defi}\label{factor12}
On d\'efinit un $R_{{\cal B},M}[G]$-module $\bPi^\dual(\rho_{{\cal B},M})$ de telle sorte que pour tout id\'eal maximal
$x$ de $R_{{\cal B},M}$ on ait 
$$\kappa(x)\otimes_{R_{{\cal B},M}} \bPi^\dual(\rho_{{\cal B},M})=\bPi(\rho_x)^\dual.$$
Pour cela, on part d'un $R^+_{{\cal B},M}$-r\'eseau $\mathcal{G}_{\qp}$-stable $\rho_{{\cal B},M}^{\diamond,+}$ du
$R_{{\cal B},M}$-dual de $\rho_{{\cal B},M}$.
Le dual de Pontryagin de $\rho_{{\cal B},M}^{\diamond,+}$ est une limite inductive $\varinjlim_i V_i$
 de repr\'esentations de dimension $2$ sur des quotients de $R_{{\cal B},M}^+$ et on pose
$\bPi^\dual(\rho_{{\cal B},M})=L\otimes_{\O_L}\varprojlim_i\bPi(V_i)^\vee$.
On a aussi $\bPi^\dual(\rho_{{\cal B},M})={\bf D}(\rho_{{\cal B},M}^{\diamond,+})^\natural\boxtimes\piqp$ ou encore
$\bPi^\dual(\rho_{{\cal B},M})={\rm Hom}_{R_{{\cal B},M}}({\rm LL}_{M,{\cal B}}^{[0,1]},R_{{\cal B},M})$, o\`u ${\rm LL}_{M,{\cal B}}^{[0,1]}$ est la compl\'et\'ee ${\cal B}$-adique
de ${\rm LL}(M)$ (cf.~\cite{petit}) et on ne consid\`ere que les morphismes continus.
\end{defi}

\subsection{Factorisation en niveau fini}\label{qc1}
Ce paragraphe est consacr\'e \`a la preuve du r\'esultat suivant (l'espace $H^1_{\eet}(\sm^p_{n,\Qbar_p})$ est d\'efini dans l'introduction, voir aussi le num\'ero ci-dessous, et la compl\'etion est $p$-adique). 

\begin{theo}\label{factor10}
Si $L$ est assez grand pour que tous les ${\rm JL}(M)$ soient d\'efinis sur $L$, on a une d\'ecomposition
$$H^1_{\eet}(\sm^p_{n,\Qbar_p})\simeq
\oplus_M\big(\widehat{\oplus}_{\cal B} \ \bPi^\dual(\rho_{{\cal B},M})
\otimes \rho_{{\cal B},M}\otimes \check{R}_{{\cal B},M}\big)\otimes_L{\rm JL}(M),$$
o\`u les produits tensoriels non sp\'ecifi\'es sont
au-dessus de $R_{{\cal B},M}$ et $\check{R}_{{\cal B},M}={\rm Hom}(\check{R}_{{\cal B},M},L)$.
\end{theo}
\begin{proof}
C'est une cons\'equence de la prop.\,\ref{dec entier}, de la rem.\,\ref{factor4}, et du (ii) du th.\,\ref{factor5}.
\end{proof}

\subsubsection{Extension des scalaires \`a $\Qbar_p$ ou \`a $\C_p$}\label{qc2}
On pose
$$H^{i,+}_{\Qbar_p}:=H^i_{\eet}({\cal M}^p_{n,\Qbar_p})^+=\varprojlim\nolimits_k\big(\varinjlim\nolimits_{[K:\Q_p]<\infty}
H^i_{\eet}({\cal M}^p_{n,K},(\O_L/{p^k})(1))\big),$$
i.e. la cohomologie compl\'et\'ee de la tour $({\cal M}^p_{n,K})_{[K:\Q_p]<\infty}$ (noter 
que $n$ est fix\'e ici). Soit $H^{i}_{\Qbar_p}=L\otimes_{\mathcal{O}_L} H^{i,+}_{\Qbar_p}$.

Pour simplifier les notations, posons pour $i\geq 0$ et $?\in \{\C_p, K\}$, avec $[K:\Q_p]<\infty$,
$$H^i_{k,?}:=H_{\eet}^i({\cal M}^p_{n,?}, (\O_L/{p^k})(1)),
\quad H^{i,+}_{?}:=H_{\eet}^i({\cal M}^p_{n,?})^+,\quad H^i_{k,\Qbar_p}:=\varinjlim_K H^i_{k,K}.$$
On a donc $H^{i,+}_{?}\simeq \varprojlim_k H^i_{k,?}$ pour $i\leq 1$ (seuls cas dont nous aurons besoin) et $H^{i,+}_{\Qbar_p}\simeq \varprojlim_{k} H^i_{k,\Qbar_p}$.

 \begin{prop}\label{QC1}
 L'application naturelle
 $$H^{1,+}_{\Qbar_p}\to H^{1,+}_{\C_p}$$
 est injective et identifie 
 $H^{1,+}_{\Qbar_p}$ \`a l'ensemble des $x\in H^{1,+}_{\C_p}$
presque lisses sous l'action de~$\G_{\Q_p}$ {\rm (i.e.~qui,
pour pour tout $k\geq 1$, sont fixes modulo $p^k$ par
un sous-groupe ouvert de $\G_{\Q_p}$)}.
 \end{prop}
 \begin{proof}

 La suite de Hochschild-Serre fournit des suites exactes
 $$0\to H_{\eet}^1(\sg_K, H^0_{k,\C_p})\to H^1_{k,K}\to (H^1_{k,\C_p})^{\sg_K}
\to H^2(\sg_K,  H^0_{k,\C_p})$$
Or $H^0_{k,\C_p}$ est un groupe fini car ${\cal M}^p_{n,\C_p}$ n'a
qu'un nombre fini de composantes connexes. Il s'ensuit que
  $\varinjlim_K H^j(\sg_K, H^0_{k,\C_p})=0$ pour $j\in \{1,2\}$ 
  et donc 
  $$H^1_{k,\Qbar_p}=\varinjlim\nolimits_{K} H^1_{k,K}\simeq (H^1_{k,\C_p})^{\sg_{\qp}{\text{-lisse}}}.$$
  En particulier la fl\`eche $H^1_{k,\Qbar_p}\to H^1_{k,\C_p}$ est injective et on obtient l'injectivit\'e de l'application naturelle $H^{1,+}_{\Qbar_p}\to H^{1,+}_{\C_p}$ en passant \`a la limite sur $k$.
  
Le morphisme \'evident $\gamma_k: H^{1,+}_{\C_p}\to H^1_{k, \C_p}$ se factorise 
$\gamma_k=\alpha_k\circ \beta_k$, o\`u 
$\beta_k: H^{1,+}_{\C_p}\to  H^{1,+}_{\C_p}/p^k$ est la projection canonique et 
$\alpha_k: H^{1,+}_{\C_p}/p^k\to H^1_{k, \C_p}$ est
$\sg_{\qp}$-\'equivariante et injective. En particulier pour tout $z\in H^{1,+}_{\C_p}$ 
et tout $k\geq 1$ les deux assertions suivantes sont \'equivalentes: 
     
$\bullet$ $\beta_k(z)$ est fixe par un sous-groupe ouvert de $\mathcal{G}_{\qp}$. 
     
$\bullet$ $\gamma_k(z)$ est fixe par un sous-groupe ouvert de $\mathcal{G}_{\qp}$.

La seconde se traduit par  
$\gamma_k(z)\in (H^1_{k,\C_p})^{\sg_{\qp}{\text{-lisse}}}\simeq H^1_{k,\Qbar_p}$.
En passant \`a la limite sur $k$, on voit que l'ensemble des $x\in H^{1,+}_{\C_p}$ presque lisses sous l'action de $\G_{\Q_p}$
 s'identifie \`a $\varprojlim_k (H^1_{k,\C_p})^{\sg_{\qp}{\text{-lisse}}}\simeq \varprojlim_k H^1_{k,\Qbar_p}=H^{1,+}_{\Qbar_p}$.
\end{proof}
                   
\begin{coro}\label{QC2}
{\rm a)} $H^{1,+}_{\Qbar_p}$ est sans $p$-torsion, $p$-adiquement complet et $p$-satur\'e dans $H^{1,+}_{\C_p}$.

{\rm b)} Pour tout $k\geq 1$ le $\mathcal{G}_{\qp}$-module $H^{1,+}_{\Qbar_p}/p^k$ est lisse. 
\end{coro}
\begin{proof}
a) C'est une cons\'equence directe de la proposition ci-dessus et du fait que $H^{1,+}_{\C_p}$ est sans $p$-torsion et $p$-adiquement complet. 

b) Soit 
$x\in H^{1,+}_{\Qbar_p}$ et soit $K$ tel que $gx-x\in p^k H^{1,+}_{\C_p}$ pour $g\in \mathcal{G}_K$. Par 
a) on a $gx-x\in p^k H^{1,+}_{\Qbar_p}$ pour $g\in \mathcal{G}_K$, 
ce qui montre que $H^{1,+}_{\Qbar_p}/p^k=\varinjlim_{K} (H^{1,+}_{\Qbar_p}/p^k )^{\mathcal{G}_K}$.
\end{proof}

\begin{ques}\label{QC3}
$H^{1,+}_{\Qbar_p}$ est-il dense dans $H^{1,+}_{\C_p}$?
\end{ques}

\subsubsection{D\'ecomposition suivant les blocs}\label{qc8}
  Le but de ce num\'ero est de d\'emontrer l'existence d'une d\'ecomposition au niveau entier, de la forme 
  $$H^{1,+}_{\Qbar_p}\simeq\widehat{\oplus}_{\mathcal{B}} [{P}_{\mathcal{B}}\otimes_{{E}_{\mathcal{B}}} {\mathbbm m}_{\mathcal{B}}(H^{1,+}_{\Qbar_p})],$$
  la compl\'etion \'etant $p$-adique et ${\mathbbm m}_{\mathcal{B}}(H^{1,+}_{\Qbar_p})$ \'etant un 
  $\O_L$-module sans torsion, s\'epar\'e et complet
pour la topologie $p$-adique, dont le $\O_L$-dual continu est de type fini sur~$R_{\cal B}^{\rm ps}$.
Le num\'ero suivant d\'ecrit la fibre g\'en\'erique de 
${\mathbbm m}_{\mathcal{B}}(H^{1,+}_{\Qbar_p})$ 
(il ne semble pas facile de d\'ecrire le module 
${\mathbbm m}_{\mathcal{B}}(H^{1,+}_{\Qbar_p})$ lui-m\^eme). 

\begin{prop}\label{damntricky}
 $(H^{1,+}_{\Qbar_p}/p^k)^{\mathcal{G}_K}$ est le dual d'un $\O_L[G']$-module lisse de longueur finie pour tous 
  $k\geq 1$ et $[K:\qp]<\infty$.
\end{prop}

\begin{proof} Notons pour simplifier $Y=(H^{1,+}_{\Qbar_p}/p^k)^{\mathcal{G}_K}$, muni de la topologie induite par celle de $H^{1,+}_{\Qbar_p}$. On dispose d'une injection continue $\iota: Y\to Z:=(H^1_{k, \C_p})^{\mathcal{G}_K}$, compos\'ee des injections 
$Y\to (H^{1,+}_{\C_p}/p^k)^{\mathcal{G}_K}$ et $(H^{1,+}_{\C_p}/p^k)^{\mathcal{G}_K}\to Z$
(cf. cor.\,\ref{QC2} pour la premi\`ere). 

Par le th\'eor\`eme de finitude,
 $Z$ est un $\O_L$-module profini, dual d'une repr\'esentation lisse de longueur finie, donc admissible, de $G$. En particulier la topologie sur $Y$ est s\'epar\'ee. 
Soit $H$ un sous-groupe ouvert compact de $G$. Par le th\'eor\`eme de finitude l'action de $H$ sur chaque 
$H^1_{k,K}$ se prolonge en une structure de $\O_L[[H]]$-module topologique, donc l'action de 
$H$ sur $Y$ se prolonge en une structure de $\O_L[[H]]$-module topologique, et 
$\iota$ est $\O_L[[H]]$-lin\'eaire (car $H$-\'equivariante et continue). On en d\'eduit que $Y$ est de type fini comme $\O_L[[H]]$-module, puisque $Z$ l'est. Comme $Y$ est aussi s\'epar\'e, la topologie sur $Y$ d\'efinie par sa structure de $\O_L[[H]]$-module de type fini (pour laquelle il est profini) est la m\^eme que sa topologie de d\'epart, donc $Y$ est profini et $\iota$ est un hom\'eomorphisme sur son image, ce qui permet de conclure.
\end{proof}

\begin{rema} Indiquons une autre preuve de la proposition ci-dessus.
Si $X=\varprojlim_k X_k$ o\`u les $X_k$ sont des espaces topologiques, et si $Y_k={\rm Im}(X\to X_k)$
est muni de la topologie induite par celle de $X_k$, alors $X\hookrightarrow\prod_k X_k$ se factorise
par $X\hookrightarrow\prod_k Y_k\hookrightarrow\prod_k X_k$ 
et $\prod_kY_k\subset\prod_k X_k$ est muni de la topologie induite, et
donc $X=\varprojlim Y_k$ comme espace topologique.
On va utiliser ce qui pr\'ec\`ede pour $X_k=H^{1}_{k,\Qbar_p}$ (et alors $Y_k=
H^{1,+}_{\Qbar_p}/p^k$).

 Montrons maintenant que 
si $K$ est une extension finie de $\Q_p$, alors
$Y_k\cap H^1_{k,K}$ est un module compact, topologiquement de longueur finie comme $\O_L[G]$-module. Il suffit de prouver que $Y_k\cap H^1_{k,K}$ est ferm\'e dans $H^1_{k,K}$ puisque
ce dernier est compact, topologiquement de longueur finie comme $\O_L[G]$-module.

Soit $(v_i)_{i\in I}$ une famille finie d'\'el\'ements de $H^{1,+}_{\Qbar_p}$ dont les images
modulo $p^k$ appartiennent \`a $Y_k\cap H^1_{k,K}$.
Si $j\geq k$, soit $\Pi_j$ l'adh\'erence dans $H^1_{j,\C_p}$ du sous-$\O_L[G\times\G_{\Q_p}]$-module
engendr\'e par les images des~$v_i$. Comme la famille est finie, il existe une extension finie
$K_j$ de $\Q_p$ telle que 
$\Pi_j\subset H^1_{j,K_j}$, et comme $H^1_{j,K_j}$ est topologiquement de longueur finie et $\Pi_j$
ferm\'e par hypoth\`ese, on en d\'eduit que $\Pi_j$ est compact et topologiquement de longueur finie.
La fl\`eche naturelle $\Pi_{j+1}\to\Pi_j$ est surjective car l'image est compacte et donc ferm\'ee, 
stable par $G$ et $\G_{\Q_p}$, et contient les images des~$v_i$.  La limite projective $\Pi$ des $\Pi_j$
est compacte et donc s'identifie \`a un sous-module ferm\'e de $H^{1,+}_{\C_p}$ contenu dans
$H^{1,+}_{\Qbar_p}$, stable par $G$ et $\G_{\Q_p}$; c'est donc l'adh\'erence dans 
$H^{1,+}_{\Qbar_p}$ du sous-$\O_L[G\times\G_{\Q_p}]$-module
engendr\'e par les $v_i$.  La fl\`eche $\Pi\to\Pi_k$ est surjective puisque toutes les fl\`eches de
transition sont surjectives. On a donc d\'emontr\'e que $\Pi_k\subset Y_k\cap H^1_{k,K}$. Comme
$H^1_{k,K}$ est de longueur finie, il en est de m\^eme de $\Pi_k$ et la longueur de
$\Pi_k$ est born\'ee ind\'ependamment de $(v_i)_{i\in I}$. Si la famille $(v_i)_{i\in I}$ est
choisie pour que la longueur de $\Pi_k$ soit maximale, on a
$\Pi_k=Y_k\cap H^1_{k,K}$, ce qui permet de conclure puisque $\Pi_k$ est ferm\'e par d\'efinition.

\end{rema}

\begin{prop}\label{dec entier} Soit $${\mathbbm m}_{\mathcal{B}}(H^{1,+}_{\Qbar_p}):={\rm Hom}_G(P_{\cal B}, H^{1,+}_{\Qbar_p}).$$
L'application naturelle 
$${\oplus}_{\mathcal{B}}[{P}_{\mathcal{B}}\otimes_{{E}_{\mathcal{B}}} {\mathbbm m}_{\mathcal{B}}(H^{1,+}_{\Qbar_p})]\to H^{1,+}_{\Qbar_p}$$
identifie $H^{1,+}_{\Qbar_p}$ au compl\'et\'e $p$-adique de
${\oplus}_{\mathcal{B}}[{P}_{\mathcal{B}}\otimes_{{E}_{\mathcal{B}}} {\mathbbm m}_{\mathcal{B}}(H^{1,+}_{\Qbar_p})]$.

\end{prop}

\begin{proof}
 Notons pour simplifier $X=H^{1,+}_{\Qbar_p}$ et $X_{k,K}=(X/p^k)^{\mathcal{G}_K}$. Par la prop.\,\ref{damntricky} 
  on peut d\'ecomposer $X_{k,K}$ suivant les blocs:
    $$X_{k,K}=\oplus_{\mathcal{B}} [{P}_{\mathcal{B}}\otimes_{{E}_{\mathcal{B}}} {\mathbbm m}_{\cal B}(X_{k,K})],\quad
{\text{o\`u ${\mathbbm m}_{\cal B}(X_{k,K})= 
    {\rm Hom}_G({P}_{\mathcal{B}}, X_{k,K})$}}$$
est un $E_{\cal B}$-module compact, de longueur finie comme $\O_L$-module. Le cor.\,\ref{QC2} permet de d\'eduire que 
$$X/p^k=\varinjlim_{K} X_{k,K}=\oplus_{\mathcal{B}} [{P}_{\mathcal{B}} \otimes_{{E}_{\mathcal{B}}} (\varinjlim_{K} {\mathbbm m}_{\cal B}(X_{k,K}))]. $$
Puisque ${P}_{\mathcal{B}}$ est compact et les $X_{k,K}$ sont profinis, on a 
$$\varinjlim_{K} {\mathbbm m}_{\cal B}(X_{k,K})= {\mathbbm m}_{\cal B}(X/p^k):={\rm Hom}_G({P}_{\mathcal{B}}, X/p^k).$$
Pour conclure, il nous reste \`a d\'emontrer que ${\mathbbm m}_{\cal B}(X/p^k)\simeq {\mathbbm m}_{\mathcal{B}}(X)/p^k$. Comme $X$ est sans $p$-torsion, on a bien une injection ${\mathbbm m}_{\mathcal{B}}(X)/p^k\to {\mathbbm m}_{\cal B}(X/p^k)$. Comme $X$ est $p$-adiquement complet, pour montrer que cette injection est surjective, il suffit de montrer que les fl\`eches 
${\mathbbm m}_{\mathcal{B}}(X/p^{j+1})\to {\mathbbm m}_{\mathcal{B}}(X/p^{j})$ sont surjectives pour 
tout~$j$, ou encore que $\varinjlim_{K} {\mathbbm m}_{\mathcal{B}} (X_{j+1,K})\to \varinjlim_{K} {\mathbbm m}_{\mathcal{B}} (X_{j,K})$ est surjective. 

Fixons $K$ et notons $f: X/p^{j+1}\to X/p^j$ la projection canonique.  Alors 
$\cup_{[L:\qp]<\infty} f(X_{j+1,L})=f(X/p^{j+1})=X/p^j$, donc $X_{j,K}$ est la r\'eunion des 
$f(X_{j+1,L})\cap X_{j,K}$. Comme $X_{j,K}$ est de longueur finie, on en d\'eduit qu'il existe 
$L$ tel que $f(X_{j+1,L})$ contienne $X_{j,K}$. Comme $X_{j+1,L}$ est de longueur finie, tout morphisme 
$G$-\'equivariant $u: {P}_{\mathcal{B}}\to X_{j,K}$ se rel\`eve en un morphisme $G$-\'equivariant 
${P}_{\mathcal{B}}\to X_{j+1,L}$, ce qui permet de conclure. 
\end{proof}

Posons simplement
$${\mathbbm m}_{n,\mathcal{B}}:={\mathbbm m}_{\mathcal{B}}(H^{1,+}_{\Qbar_p})={\mathbbm m}_{\mathcal{B}}(H^1_{\eet}({\cal M}^p_{n,\Qbar_p})^+),
\quad
\check {\mathbbm m}_{n,\mathcal{B}}:={\rm Hom}^{\rm cont}_{\O_L}({\mathbbm m}_{n,\mathcal{B}},\O_L)$$
Alors ${\mathbbm m}_{n,\mathcal{B}}$ est un $\O_L$-module sans torsion, s\'epar\'e et complet
pour la topologie $p$-adique, tandis que $\check {\mathbbm m}_{n,\mathcal{B}}$ est un $\O_L$-module compact, 
sans torsion.

\begin{theo}\label{factor3}
$\check{\mathbbm m}_{n,\cal B}$ est un $R_{\cal B}^{\rm ps}$-module de type fini.
\end{theo}
\begin{proof}
Soit ${\goth m}$ l'id\'eal maximal de $R_{\cal B}^{\rm ps}$. Par le lemme de Nakayama topologique il suffit de voir que 
$\check {\mathbbm m}_{n,{\cal B}}/{\goth m}=({\mathbbm m}_{n,{\cal B}}/\varpi)^{\vee}/{\goth m}=(({\mathbbm m}_{n,{\cal B}}/\varpi)[{\goth m}])^{\vee}$ 
est de dimension finie sur $k_L$, autrement dit que 
$\dim ({\mathbbm m}_{n,{\cal B}}/\varpi)[{\goth m}]<\infty$. 
On a une injection 
$${\mathbbm m}_{n,{\cal B}}/\varpi\to {\rm Hom}_G({P}_{\mathcal{B}}/\varpi, H^{1,+}_{\Qbar_p}/\varpi)\to 
{\rm Hom}_G({P}_{\mathcal{B}}/\varpi, H^1_{\eet}({\cal M}^p_{n,\C_p}, k_L(1)))$$
et donc une injection 
$$({\mathbbm m}_{n,{\cal B}}/\varpi)[{\goth m}]\to {\rm Hom}_G(k_L\otimes_{R_{\cal B}^{\rm ps}} {P}_{\mathcal{B}}, H^1_{\eet}({\cal M}^p_{n,\C_p}, k_L(1))).$$
Le cor.\,6.7 de \cite{PT} montre que $k_L\otimes_{R_{\cal B}^{\rm ps}} {P}_{\mathcal{B}}=\pi^{\vee}
$ pour une repr\'esentation lisse, admissible (en fait de longueur finie) $\pi$, et on conclut en utilisant le 
cor.\,\ref{lg13}.
\end{proof}

\subsubsection{Structure de ${\mathbbm m}_{\cal B}(H^1_{\eet}({\cal M}^p_{n,\Qbar_p}))$ comme
$R_{\cal B}^{\rm ps}$-module}\label{qc9}
Nous allons maintenant rendre $p$ inversible pour pouvoir d\'ecomposer notre module
en composantes isotypiques pour les repr\'esentations de dimension finie de $\check G$.
Si $X$ est un $G$-module topologique on pose ${\mathbbm m}_{\cal B}(X)={\rm Hom}_G({P}_{\mathcal{B}}, X)$. On a donc ${\mathbbm m}_{\cal B}(H^{1}_{\Qbar_p})=L\otimes_{\mathcal{O}_L} 
 {\mathbbm m}_{\cal B}(H^{1,+}_{\Qbar_p})$ par compacit\'e de ${P}_{\mathcal{B}}$.

\begin{rema}\label{factor4}
{\rm (i)} Puisque $H^{1}_{\Qbar_p}$ s'injecte dans $H^1_{\eet}({\cal M}^p_{n,\C_p})$, sur lequel l'action de $\check{G}$ se factorise par un quotient fini,
il n'y a qu'un nombre fini de $M$ tels que
$H^{1}_{\Qbar_p}[M]$ soit non nul, et on a
$$H^{1}_{\Qbar_p}=\oplus_M
\big(H^{1}_{\Qbar_p}[M]\otimes_L {\rm JL}(M)\big)$$
et donc on dispose d'une d\'ecomposition 
$${\mathbbm m}_{\cal B}(H^{1}_{\Qbar_p})=\oplus_M 
\big({\mathbbm m}_{{\cal B},M}\otimes_L {\rm JL}(M)\big)$$
avec 
$${\mathbbm m}_{{\cal B},M}={\mathbbm m}_{{\cal B}}(H^{1}_{\Qbar_p}[M])={\rm Hom}_{G\times \check{G}} ({P}_{\mathcal{B}}\otimes {\rm  JL}(M), H^1_{\eet}({\cal M}^p_{n,\Qbar_p})).$$
Le module ${\mathbbm m}_{{\cal B},M}$ ne d\'epend pas du choix de $n$ sup\'erieur ou
\'egal au niveau de $M$ car $H^1_{\eet}({\cal M}^p_{n,\Qbar_p})$ est le sous-espace des $\check{G}_n$-invariants de $\varinjlim_n H^1_{\eet}({\cal M}^p_{n,\Qbar_p})$
(comme on a invers\'e $p$, on peut fabriquer des projecteurs 
$\check{G}$-\'equivariants en composant restrictions et corestrictions).

{\rm(ii)} Sur $H^{1}_{\Qbar_p}[M]$ et donc aussi
sur ${\mathbbm m}_{\cal B}(H^{1}_{\Qbar_p}[M])$
le centre de $G$ agit par un caract\`ere $\delta_M$ puisque c'est le cas de celui de $\check G$
et que les actions des deux centres sont inverses l'une de l'autre sur $H^1_{\eet}({\cal M}^p_{n,\C_p})$, qui contient $H^{1}_{\Qbar_p}$.
\end{rema}
Le point (ii) de la remarque ci-dessus va nous permettre d'utiliser le th.\,\ref{pasku6}
pour analyser le module ${\mathbbm m}_{{\cal B},M}$.

\begin{theo}\label{factor5}
{\rm (i)}  $R_{\cal B}^{{\rm ps},\delta_M}$ agit sur
${\mathbbm m}_{{\cal B},M}$ \`a travers son quotient $R_{{\cal B},M}$.

{\rm (ii)} On a un isomorphisme de $E_{\cal B}^{\delta_M}[\G_{\Q_p}]$-modules
$${\mathbbm m}_{{\cal B},M}\simeq 
{\mathbbm m}_{\cal B}^{\delta_M}(\bPi^\dual(\rho_{{\cal B},M}))
\otimes \rho_{{\cal B},M}\otimes \check{R}_{{\cal B},M}$$
 o\`u
$\check{R}_{{\cal B},M}$ est le $L$-dual de $R_{{\cal B},M}$ et les produits tensoriels sont
au-dessus de $R_{{\cal B},M}$.
\end{theo}
\begin{proof}
Le cas $M$ sp\'ecial est imm\'ediat. Supposons donc $M$ supercuspidal dans ce qui suit.

{\rm (i)} On veut montrer que $I:=\ker(R_{\cal B}^{{\rm ps}, \delta_M}[1/p]\to R_{{\cal B},M})$ tue 
${\mathbbm m}_{{\cal B},M}$, et il suffit de voir que $I$ tue ${\rm Hom}_{G} ({P}_{\mathcal{B}}, H^1_{\eet}({\cal M}^p_{n,\C_p})[M])$. Puisque $p>2$, on a une identification 
$Z_{\mathcal{B}}^{\delta_M}\simeq R_{\cal B}^{{\rm ps}, \delta_M}/\mathcal{O}_L{\text{-{\rm tors}}}$,
 ce qui permet de voir $Z_{\mathcal{B}}^{\delta_M}$ comme un $\mathcal{O}_L$-r\'eseau 
de $R_{\cal B}^{{\rm ps}, \delta_M}[1/p]$. 
Soit $J=Z_{\mathcal{B}}^{\delta_M}\cap I$, donc $I=J[1/p]$.

 Comme les fonctions born\'ees sur ${\cal M}^p_{n,\C_p}$ sont constantes\footnote{Voir le cor.\,2.10 de \cite{CDN2} pour un \'enonc\'e plus g\'en\'eral; dans le cas particulier de ${\cal M}^p_{n,\C_p}$ on peut utiliser la finitude du morphisme vers le demi-plan de Drinfeld et le fait que le r\'esultat est bien connu pour ce dernier espace.}
et comme ${\cal M}^p_{n,\C_p}$ poss\`ede un ouvert affino\"{\i}de dont les translat\'es sous l'action de $G$ recouvrent ${\cal M}^p_{n,\C_p}$, on a $\mathcal{O}({\cal M}^p_{n,\C_p})[M]^b=0$. 
En passant aux vecteurs $G$-born\'es dans la suite exacte (th.\,5.11 de \cite{CDN1}) 
$$0 \to \mathcal{O}({\cal M}^p_{n, \C_p})[M]\to 
H^1_{\proet}({\cal M}^p_{n, \C_p})[M]\to
X_{\rm st}^+(M)\widehat{\otimes}_L{\rm LL}(M)^\dual\to 0$$
et en utilisant la prop.\,\ref{etproet} et le lemme \ref{b} on voit que 
$H^1_{\eet}({\cal M}^p_{n,\C_p})[M]$ s'injecte dans $X_{\rm st}^+(M)\widehat{\otimes}_L \widehat{{\rm LL}}(M)^\dual$. Il suffit donc de montrer que $I$ tue ${\rm Hom}_{G} ({P}_{\mathcal{B}}, \widehat{{\rm LL}}(M)^\dual)$.

  Soit ${\rm LL}(M)^+$ un $\mathcal{O}_L$-r\'eseau $G$-stable dans ${\rm LL}(M)$ (un tel r\'eseau existe car ${\rm LL}(M)$ est supercuspidale, \`a caract\`ere central unitaire) et soit 
   $Y$ son compl\'et\'e $p$-adique. On a donc 
    $\widehat{{\rm LL}}(M)^\dual=Y^d[1/p]$, avec $Y^d:={\rm Hom}_{\mathcal{O}_L}^{\rm cont}(Y, \mathcal{O}_L)\simeq \varprojlim_{n} (Y/p^nY)^{\vee}$ (un $\mathcal{O}_L$-module compact). Par compacit\'e de ${P}_{\mathcal{B}}$ on a ${\rm Hom}_{G} ({P}_{\mathcal{B}}, \widehat{{\rm LL}}(M)^\dual)={\rm Hom}_G({P}_{\mathcal{B}}, 
    Y^d)[1/p]$. Il suffit donc de montrer que $J=Z_{\mathcal{B}}^{\delta_M}\cap I$ tue 
   $${\rm Hom}_G({P}_{\mathcal{B}}, 
    Y^d)\simeq \varprojlim_{n} {\rm Hom}_G({P}_{\mathcal{B}}, (Y/p^nY)^{\vee}).$$ Fixons $n$ par la suite, et montrons que $J$ tue ${\rm Hom}_G({P}_{\mathcal{B}}, (Y/p^nY)^{\vee})$.

   Soit $S$ l'ensemble 
   des quotients du $\mathcal{O}_L[G]$-module ${\rm LL}(M)^+$ qui sont des 
  objets de ${\rm Rep}_{\mathcal{B}}^{\delta_M} G$, i.e. des 
    repr\'esentations lisses de longueur finie de $G$, dans la cat\'egorie d\'ecoup\'ee par le bloc $\mathcal{B}$. Alors $S$ est au plus d\'enombrable, et puisque tout objet de ${\rm Rep}_{\mathcal{B}}^{\delta_M} G$ est tu\'e par une puissance de $p$, le morphisme naturel ${\rm LL}(M)^+\to Y$ identifie $S$ \`a l'ensemble des quotients de $Y$ appartenant \`a ${\rm Rep}_{\mathcal{B}}^{\delta_M} G$. 
  Le  
   $\mathcal{O}_L[G]$-module prodiscret 
   $$Y_{\mathcal{B}}:=\varprojlim\nolimits_{\sigma\in S} \sigma$$
   poss\`ede un morphisme naturel d'image dense $\iota: Y\to Y_{\mathcal{B}}$. Il est montr\'e dans \cite[cor.\,3.12]{petit} que $Y_{\mathcal{B}}$ est sans $p$-torsion, $p$-adiquement complet et que 
   $\iota$ induit une surjection de $Y/p^n Y$ sur $Y_{\mathcal{B}}/p^n Y_{\mathcal{B}}$
    pour tout $n$. Nous allons montrer que l'injection 
   ${\rm Hom}_G({P}_{\mathcal{B}}, (Y_{\mathcal{B}}/p^nY_{\mathcal{B}})^{\vee})\to {\rm Hom}_G({P}_{\mathcal{B}}, (Y/p^nY)^{\vee})$ 
   induite par le plongement $(Y_{\mathcal{B}}/p^n Y_{\mathcal{B}})^{\vee}\to (Y/p^n Y)^{\vee}$ 
   est un isomorphisme, et que $J$ tue ${\rm Hom}_G({P}_{\mathcal{B}}, (Y_{\mathcal{B}}/p^nY_{\mathcal{B}})^{\vee})$, ce qui permettra de conclure.

    Soit donc $\varphi: {P}_{\mathcal{B}}\to (Y/p^nY)^{\vee}$ un morphisme continu $G$-\'equivariant et soit $\sigma:=\varphi({P}_{\mathcal{B}})^{\vee}$. Alors $\sigma$ est un quotient de 
    $Y/p^nY\simeq {\rm LL}(M)^+/p^n {\rm LL}(M)^+ $ et $\sigma$ s'injecte dans 
    ${P}_{\mathcal{B}}^{\vee}$, qui est une limite inductive de repr\'esentations dans 
   ${\rm Rep}_{\mathcal{B}}^{\delta_M} G$, donc $\sigma$ est la r\'eunion de ses sous-repr\'esentations appartenant \`a ${\rm Rep}_{\mathcal{B}}^{\delta_M} G$. Puisque $\sigma$ est un quotient de ${\rm LL}(M)^+/p^n {\rm LL}(M)^+$, qui est de type fini comme 
   $\mathcal{O}_L[G]$-module, on en d\'eduit que $\sigma\in {\rm Rep}_{\mathcal{B}}^{\delta_M} G$ et $\varphi$ se factorise par $\sigma^{\vee}$, qui se plonge dans $(Y_{\mathcal{B}}/p^nY_{\mathcal{B}})^{\vee}$. Donc l'injection ${\rm Hom}_G({P}_{\mathcal{B}}, (Y_{\mathcal{B}}/p^nY_{\mathcal{B}})^{\vee})\subset {\rm Hom}_G({P}_{\mathcal{B}}, (Y/p^nY)^{\vee})$ est aussi surjective. 
   
    Il reste \`a expliquer pourquoi $J$ tue ${\rm Hom}_G({P}_{\mathcal{B}}, (Y_{\mathcal{B}}/p^nY_{\mathcal{B}})^{\vee})$. Pour cela, notons que $Y_{\mathcal{B}}$
   poss\`ede une structure naturelle de $Z_{\mathcal{B}}^{\delta_M}$-module, puisque chaque 
  $\sigma\in S$ en a une.    
   Il est montr\'e dans 
   \cite{petit} que $J$ tue $Y_{\mathcal{B}}$. On dispose de deux structures de $Z_{\mathcal{B}}^{\delta_M}$-module sur ${\rm Hom}_G({P}_{\mathcal{B}}, (Y_{\mathcal{B}}/p^nY_{\mathcal{B}})^{\vee}))$: une \`a partir de la structure de 
    $Z_{\mathcal{B}}^{\delta_M}$-module de ${P}_{\mathcal{B}}$ et l'autre \`a partir de la structure de $Z_{\mathcal{B}}^{\delta_M}$-module de $Y_{\mathcal{B}}$. Puisque 
    $Z_{\mathcal{B}}^{\delta_M}$ est le centre de ${\rm Rep}_{\mathcal{B}}^{\delta_M} G$, ces deux structures sont les m\^emes et donc $J$ tue ${\rm Hom}_G({P}_{\mathcal{B}}, (Y_{\mathcal{B}}/p^nY_{\mathcal{B}})^{\vee})$, ce qui finit la preuve du point (i).

    {\rm (ii)} Posons $X:={\rm Spm}(R_{\cal B}^{{\rm ps},\delta_M}[\frac{1}{p}])$,
et $X_M:={\rm Spm}(R_{{\cal B},M})$ et \'ecrivons ${\goth m}_x$ pour l'id\'eal maximal correspondant \`a $x\in X$. 

 Si $x\in X$, alors $\check{\mathbbm m}_{{\cal B},M}/{\goth m}_x$ est le dual de
${\mathbbm m}_{{\cal B},M}[{\goth m}_x]$. Mais
\begin{align*}
{\mathbbm m}_{{\cal B},M}[{\goth m}_x] & \simeq 
{\rm Hom}_G({P}_{\cal B}/{\goth m}_x, H^1_{\eet}({\cal M}^p_{n,\Qbar_p})[M])\\
&\simeq  {\rm Hom}_G({\mathbbm m}(\Pi_x^\dual)\otimes_L\Pi_x^{\dual}, H^1_{\eet}({\cal M}^p_{n,\C_p})[M])\\
&\simeq  \begin{cases}0&{\text{si $x\notin X_M$,}}\\
{\mathbbm m}(\Pi_x^\dual)\otimes_L \rho_x &{\text{si $x\in X_M$.}}\end{cases}
\end{align*}
(On passe de la premi\`ere \`a la seconde ligne en utilisant la prop.\,\ref{pasku2} (et la rem.\,\ref{pasku1}),
et de la seconde \`a la troisi\`eme en utilisant le th.\,\ref{cdn1.2}.)

On en d\'eduit que les $L$-duaux des deux membres de (ii) ont la m\^eme sp\'ecialisation en tout point.
Comme ces $L$-duaux sont des $R_{{\cal B},M}$-modules localement libres, de type fini,
et que $R_{{\cal B},M}$ est un produit d'anneaux principaux,
le r\'esultat s'en d\'eduit en adaptant\footnote{
\label{inutile2}Les deux membres sont des repr\'esentations de rang fini
de $E_{{\cal B},M}\times\G_{\Q_p}$ dont les
sp\'ecialisations sont absolument irr\'eductibles en tout point. 
Si $F$ est le corps des fractions
d'un des facteurs de $R_{{\cal B},M}$, alors $F\otimes_{R_{{\cal B},M}}E_{{\cal B},M}$ 
est une alg\`ebre
semi-simple sur $F$ (ce point repose sur la description explicite donn\'ee par \paskunas~\cite{Paskext} qui demande d'exclure les blocs contenant un twist de la steinberg si $p=2,3$), 
et donc ses $F$-repr\'esentations irr\'eductibles
de dimension finie sont d\'etermin\'ees
par leurs traces; il en est donc de m\^eme des $F$-repr\'esentations irr\'eductibles de 
$(F\otimes_{R_{{\cal B},M}}E_{{\cal B},M})\times \G_{\Q_p}$. Ceci permet, comme
dans la preuve du lemme~\ref{abstrait}, de prouver que
les deux membres deviennent isomorphes apr\`es extension des scalaires \`a $F$ et le reste
de la preuve s'adapte alors verbatim.}
 le lemme~\ref{abstrait} 
pour inclure l'action de 
$E_{{\cal B},M}:=R_{{\cal B},M}\otimes_{R^{{\rm ps},\delta_M}_{\cal B}} E_{\cal B}^{\delta_M}$.
\end{proof}


\end{document}